\mag=\magstep1 
\input epsf 
\documentclass[10pt,a4paper]{amsart} 
\usepackage{amssymb,latexsym} 
\usepackage{comment}

\def\F5{\bold F_5}

\def\Z{\bold Z}

\def\e2{{\eta_2}}

\def\no{\noindent}

\def\ti{\times}

\def\part{\partial}

\def\Cal{\mathcal} 
 
\def\frak{\mathfrak}

\def\Supp{\operatorname{Supp}}

\def\tr{\operatorname{tr}}

\def\<<{\langle } 
\def\>>{\rangle }

\def\nsubset{\subset\hspace{-0.7em}\not\hspace{0.7em}} 
\def\sha{\sqcup}

\numberwithin{equation}{section} 
\newtheorem{theorem}{Theorem}[section] 
\newtheorem{proposition}[theorem]{Proposition} 
\newtheorem{proposition-definition}[theorem]
{Proposition-Definition} 
\newtheorem{corollary}[theorem]{Corollary} 
\newtheorem{definition}[theorem]{Definition} 
 
\newtheorem{remark}[theorem]{Remark} 
 
\newtheorem{lemma}[theorem]{Lemma} 
 
\newtheorem{example}[theorem]{Example}

\def\Ker{\operatorname{Ker}}

\begin{document} 
 
\title{ 
Relative DGA and mixed elliptic motives
} 
 
\author{Kenichiro Kimura and Tomohide Terasoma} 
 
\maketitle 
\makeatletter 
\tableofcontents

\section{Introduction and convention}
\subsection{Introduction}

Let $K$ be a field.
In the paper \cite{V}, Voevodsky defined a triangulated category $D_{MM,K}$
of mixed motives. Under the assumption of Beilinson-Soule vanishing conjecture
for varieties over $K$, there exists a reasonable t-structure $\tau$ on
$D_{MM,K}$, and the abelian category $\Cal A_{MM,K}$ of mixed motives
is defined as the heart of $D_{MM,K}$ with respect to the t-structure $\tau$.

Let $E$ be an elliptic curve over $K$ and we assume Beilinson-Soule
vanishing conjectures for varieties over $K$.
The category of elliptic motives is defined as the smallest full
subcategory of $\Cal A_{MM,K}$ containing $h^1(E)$ closed under taking 
direct sums, direct summands,
and tensor products. The category of mixed elliptic motives is the
smallest full subcategory of $\Cal A_{MM,K}$ containing elliptic motives
and closed under extensions.
That is, an object $M$ in $\Cal A_{MM,K}$ is a mixed elliptic motif
if there exists a filtration $M=F^0M\supset F^1M \supset \cdots \supset F^nM=0$
such that $F^iM/F^{i+1}M$ are elliptic motives for $i=0, \dots, n-1$.
In other words, the category of mixed elliptic
motives is the relative completion of the category of mixed motives with 
respect to the category of elliptic motives (\cite{H}).
For example, the category of mixed Tate motives is the relative
completion of the category of mixed motives with respect to the category
of pure Tate motives, which is the smallest full subcategory of 
$\Cal A_{MM,K}$ containing Tate objects $\bold Q(i)$ closed under taking direct sums.

In the paper \cite{BK}, Bloch and Kriz construct an abelian category of mixed Tate motives
as the category of comodules over a Hopf algebra obtained by
the bar construction of the DGA of cycle complexes.
One advantage of their construction is that
Beilinson-Soule vanishing conjectures is not necessary for their construction.
In this paper, we construct a Hopf algebra $H^0(Bar(\Cal C_{VEM}))$ such that
the abelian category of $H^0(Bar(\Cal C_{VEM}))$-comodules is expected
to be that of mixed elliptic motives. 
More precisely, we construct a DG category $(VMEM)$ of DG complexes of elliptic motives
and prove that it is homotopy equivalent to the DG category of $Bar(\Cal C_{VEM})$-comodules. 
As a consequence, the homotopy category of the subcategory of
$(VMEM)$ consisting of objects concentrated at degree zero is equivalent
to the abelian category of $H^0(Bar(\Cal C_{VEM}))$-comodules. Another
generalization of
Bloch-Kriz construction in the context of elliptic curves is
also studied in the earlier work by Bloch \cite{B} and Patashnick \cite{P}.
Hain and Matsumoto (\cite{HM}) studies Hodge and $l$-adic counter part
of the mixed elliptic motives.

Before going into this subject, it will be helpful to recall a similar
construction for the category of local systems over a manifold $X$ after R. Hain.
Let $F$ a field of characteristic zero,
$x$ a point in $X$ and $V$ be a $F$-local system on $X$.
Let $G=\pi_1(X,x)$ and
$\rho: G \to Aut(V)$ be the monodromy representation and
$S$ be the Zariski closure of the image of $\rho$.
Assume that $S$ is a reductive group over $F$.
\begin{definition}
Let $W$ be a finite dimensional representation of $G$.
If there exists a finite filtration 
$W=W^0 \supset W^1 \supset \cdots \supset W^n=0$
such that $Gr^i(W)=W^i/W^{i+1}$ is isomorphic to
the pull back of an algebraic representation of $S$,
the representation $W$ is called a successive extension
of algebraic representations of $S$.
The full subcategory of $G$ consisting of successive extensions
of algebraic representations of $S$ is denoted as $(Rep_{G})^S$.
\end{definition}
Then objects in $(Rep_{G})^S$ are stable under taking duals, 
direct products and tensor products, and
$(Rep_G)^S$ becomes a Tannakian category.
When $F=\bold R$,
Hain constructed a Hopf algebra $H$, in \cite{H},
using differential forms on $X$
and the connection associated to the local system $V$
such that the category of
comodules over $H$ is equivalent to $(Rep_G)^S$.
In other words,  
the Tannaka fundamental group $\pi_1((Rep_G)^S)$ of $(Rep_G)^S$ is isomorphic to $Spec(H)$.
His construction is called the relative bar construction.

In this paper, we reformulate Hain's construction using relative DGA's
so that it can be
applied to the motivic context. 
We give a general homotopical framework of relative bar complexes and DG-categories
in \S \ref{section:relative DGA and relative bar}
and \S \ref{section:relative DGA and DG category}. 
In \S \ref{section:relative DGA and relative bar}
we define a relative DGA,  and introduce the relative bar complex $Bar(A/\Cal O,\epsilon)$
and the relative simplicial bar complex $Bar_{simp}(A/\Cal O,\epsilon)$ of
a relative DGA $A$ with a relative augmentation $\epsilon$. We show
that these two bar complexes are quasi-isomorphic. The simplicial 
bar complex $Bar_{simp}(A/\Cal O,\epsilon)$ is defined in order to establish the equivalence of the category
of bar comodules to the category of DG complexes in $\Cal C(A/\Cal O)$ 
defined in the next section.
In \S \ref{section:relative DGA and DG category}, 
we introduce a DG category $\Cal C(A/\Cal O)$ arising from a relative DGA
$A$ over $\Cal O$. In this section, we prove that
the DG category $K\Cal C(A/\Cal O)$ of DG complexes in $\Cal C(A/\Cal O)$
is homotopy equivalent to that of $Bar(A/\Cal O,\epsilon)$-comodules.
The main theorem in this section is Theorem \ref{main theorem genereal}.

\noindent In \S 3.4 
and \S \ref{subsection:relative completion associated to 
represenataion} we recover Hain's construction in our formulation. Although
this is not necessary to construct mixed elliptic motives, it will give
an evidence that our formulation of relative bar complex is a right one.
By the homomorphism 
$\rho:G\to S(\bold k)$, the coordinate ring $\Cal O_S$ of the algebraic
group $S$ becomes a bi-$G$ module. The comodule $\Cal O_S$ as a left
$G$-module is written as $L\Cal O_S$.
The $G$-cochain complex 
$A=\underline{Hom}_G(L\Cal O_S,L\Cal O_S)$ becomes a bi-$\Cal O_S$ comodule
and is equipped with a multiplication arising from Yoneda pairing.
The complex $A$ is called the relative DGA associated to the map $\rho$.
In \S \ref{subsection:relative completion associated to 
represenataion}, we prove that
the Hopf algebra $\Cal O(\pi_1((Rep)^S))$ is isomorphic to the $0$-th
homology of relative bar complex defined in \S \ref{subsec:relative bar complex}.
Our proof is based on DG categories and DG complexes introduced in
\cite{BoK}.
(See also \cite{T}.)

In \S 4 we construct a DG category $(MEM)$ of naive mixed elliptic motives as the category of
DG complexes of pure elliptic motives.
We expect that the homotopy category $H^0(MEM)$ of $(MEM)$ is
isomorphic to the full sub-triangulated tensor category of $D_{MM,K}$
generated by $h^1(E)$ and their tensors.
We define a relative quasi-DGA $A_{EM}$ from algebraic cycles and then 
apply the general result in \S \ref{section:relative DGA and relative bar}
and \S \ref{section:relative DGA and DG category}
 to show that the 
category $(MEM)$ is homotopy
equivalent to the category of $Bar(A_{EM})$-comodules.
As a consequence, an object concentrated at degree zero defines an
$H^0(Bar(A_{EM}))$-comodule.
The main theorem in this section is
Theorem \ref{main theorem for naive mixed elliptic motives}.

The product structure on $H^0(Bar(A_{EM}))$ is given 
by shuffle product. This comes from
a tensor structure on the category $(MEM)$. 
However, as is explained in \S \ref{reason why naive does not distributive},
the category of naive elliptic motives does not have a tensor structure with distributive property
which gives rise to a shuffle product structure on the
bar complex $Bar(A_{EM})$.
So we introduce a category of virtual mixed elliptic motives $(VMEM)$
which is homotopy equivalent to $(MEM)$ and equipped with 
a tensor structure with distributive property. To show the homotopy
equivalence of categories $(EM)$ and $(VEM)$, the
injectivity of linear Chow group 
(Proposition \ref{injective to cohomology}) is necessary.
Though there is a proof of this proposition also in \cite{BL},
we will give a more direct proof.
Using this properties, we construct a quasi-DG category $(VMEM)$ which
has distributive tensor structure, and is homotopy equivalent to $(MEM)$.
(Theorem \ref{main theorem for virtual mixed elliptic motives}.)
By the shuffle product induced by the tensor structure, the coalgebra
$H^0(Bar(A_{EM}))=H^0(Bar(\Cal C_{VEM}))$ becomes a Hopf algebra.
(Theorem \ref{shuffle product commutative associative})
The definition of the Tannakian category of mixed elliptic motives
is defined in Definition \ref{defintion algebraic group}.

In \S 5 we construct an elliptic polylog motif $Pl_n$
as an example of object concentrated at degree zero in $(MEM)$.
To define the polylog motif, we first recall the elliptic polylog class
$P_n$ as an element of a higher Chow group introduced in
\cite{BL}.
Using the elliptic polylog class, we construct an elliptic polylog motif $Pl_n$.
Using the bijection of objects of $(MEM)$ and $Bar(A_{EM})$,
we write down the comodule structure on $Pl_n$ over
$H^0(Bar(A_{EM}))$. In this paper, we assume that
the elliptic curve $E$ has no complex multiplication.

\subsection{Convention}

The subset $\{i \in \bold Z\mid p \leq i \leq q\}$ of $\bold Z$ is denoted by
$[p,q]$.
Let $\bold k$ be a field of characteristic zero. The tensor products
mean those over the base field $\bold k$. 
\subsubsection{}
\label{def of cotensor product}
Let $\Cal O$ be a counitary coalgebra.
The counit $\Cal O \to \bold k$ is denoted by $u$
and the comultiplication of $\Cal O$ is denoted by
$\Delta_{\Cal O}:\Cal O \to \Cal O\otimes \Cal O$.
Let $M$ and $N$ be right and left
$\Cal O$ comodules. We define the cotensor product $M\otimes_{\Cal O}N$
by
$$
\Ker(M\otimes N \overset{d}\to M \otimes \Cal O \otimes N),
$$
where the map $d$ is defined by 
$\Delta_M \otimes id_N - id_M \otimes \Delta_N$.
Let $V_1, V_2$ be left $\Cal O$-comodules.
We define $Hom_{\Cal O}(V_1, V_2)$
by
$$
\Ker(Hom_k(V_1,V_2) \overset{d}\to Hom_k(V_1, \Cal O\otimes V_2))
$$
where the map $d$ is the difference of
$$
\Delta_{V_2*}:Hom_k(V_1, V_2) \to Hom_k(V_1, \Cal O\otimes V_2)
$$
and the composite
$$
Hom_k(V_1, V_2) \to Hom_k(\Cal O\otimes V_1, \Cal O\otimes
V_2)\overset{\Delta_{V_1}^*}\to
Hom_k(V_1, \Cal O\otimes V_2).
$$
Let $A$ be a bimodule over the coalgebra $\Cal O$,
and $\varphi\in Hom_{\Cal O}(V_1, V_2)$. Then
$1_A\otimes \varphi$ becomes an element in
$Hom_{\Cal O}(A\otimes_{\Cal O}V_1, A\otimes_{\Cal O}V_2)$.

\subsubsection{}
Let $M$ be an $\Cal O$-comodule. Then we have 
the following complex
$$
0\to M \overset{\Delta_M}{\to} \Cal O\otimes M 
\overset{\Delta_{\Cal O}\otimes 1 -1\otimes \Delta_M}
\longrightarrow \Cal O\otimes \Cal O \otimes M
$$
by the coassociativity of $M$.
This is exact since the maps
\begin{align*}
& u\otimes 1_M:\Cal O\otimes M \to M, \\
& u\otimes 1_{\Cal O}\otimes 1_M-
1_{\Cal O}\otimes u \otimes 1_M:
\Cal O\otimes \Cal O\otimes M \to \Cal O\otimes M 
\end{align*}
give homotopy.
As a consequence, the natural map $M\to \Cal O\otimes_{\Cal O}M$ 
is an isomorphism.

\subsubsection{}
\label{many copy of group its elements, etc}
Let $S$ be an algebraic group and $\Cal O_S$ the coordinate ring.
Then $\Cal O_S$ becomes a Hopf algebra whose coproduct $\Cal O_S\to \Cal
O_S\otimes \Cal O_S$ is obtained by
$$
S\times S\to S:(g,h)\mapsto hg.
$$
There is natural one to one
correspondence between left algebraic representations of $S$
and left comodules over $\Cal O_S$.
The correspondence is given as follows. Let $g$ be an element in
$S(\bold k)$. Then the evaluation at $g$ defines an algebra homomorphism
$ev_g:\Cal O_S\to \bold k$. Let $V$ be a left $\Cal O_S$-comodule and
$\Delta_V:V\to \Cal O_S\otimes V$ be the comodule structure on $V$.
The action of $g$ on $V$ is given by the composite of
$$
V\to \Cal O_S\otimes V\xrightarrow{ev_g} \bold k\otimes V=V.
$$
The left $S$ module $\Cal O_S$ is written by $L\Cal O_S$.
A reductive group $S$ is said to be split if
all irreducible representations
over $\bold k$ are absolutely irreducible. Let $Irr(S)$ be the set of isomorphism
classes of irreducible representations. If $S$ is split, then we have
$$
\Cal O_S=\oplus_{\alpha\in Irr(S)}V^{\alpha}\otimes {}^{\alpha}V,
$$
where $\alpha$ runs through the isomorphism classes of
irreducible representations of left $S$-modules
and $V^{\alpha}$ be the corresponding left $S$-module.
The dual vector space ${}^{\alpha}V$ of $V^{\alpha}$
becomes a right $S$-module.
The function $\varphi$ corresponding to
$(v\otimes v^*)$ for
$v\in V^{\alpha},v^* \in{}^{\alpha} V$
is defined by $\varphi(g)=(v^*,gv)=(v^*g,v)$.
Therefore the dual
$Hom_{\bold k}(\Cal O_S,\bold k)$ of $\Cal O_S$ is naturally isomorphic
to $\Cal O_S$. 
Here for an $S$ representation $W$, $Hom_{\bold k}(*,*)$ is defined by
$$
Hom_{\bold k}(W_1,W_2)=\oplus_{\alpha\in Irr(S)}
Hom_{\bold k}(V^{\alpha}\otimes Hom_{\Cal O_S}(V^{\alpha},W_1), 
W_2).
$$
The natural pairing is given by
$$
(v\otimes v^*)\otimes (w\otimes w^*)=(w^*,v)(v^*,w)
$$
for $v,w\in V^{\alpha},w^*,w^* \in \ ^{\alpha} V$.
This isomorphism does not depend on the choice
of representative of the isomorphism classes.
By descent theory, there exists a natural isomorphism between
$\Cal O_S$ and $Hom(\Cal O_S,\bold k)$.
\subsubsection{}
>From now on, we assume that $S$ is split. We set $\Cal O=\Cal O_S$
The multiplication $\Cal O\otimes \Cal O\to\Cal O$
on $\Cal O$ is defined by the diagonal embedding $S \to S \times S$.
If functions $\varphi$ and $\psi$ correspond to
elements $v\otimes v^*$ and $w\otimes w^*$, then
the product function $\varphi\psi$
is the function
$g\mapsto (v^*g,v)\cdot (w^*g,w)$.

This map is described by the duality of intertwining spaces.
Let $V$ be an irreducible representation of $S$
and $I_{V}$ and $I^*_{V}$ be 
covariant and contravariant intertwining spaces defined by
\begin{align*}
I_{V}(W)=Hom_{S}(V,W),\quad
I_{V}^*(W)=Hom_{S}(W,V).
\end{align*}
Then the composite $g\circ f\in Hom_S(V,V)$
of $f\in I_V(W)$ and $g\in I_V^*(W)$
is regarded as an element of $\bold k$ via the isomorphism
$Hom_G(V,V)=\bold k$.
It is called the contraction and is denoted as $con(f,g):I_V(W)\otimes I_V^*(W)\to \bold k$.
The multiplication map is the sum of the composite 
\begin{align*}
con_{\alpha,\beta}^\gamma:(V^{\alpha}\otimes{}^{\alpha} V)\otimes
(V^{\beta}\otimes{}^{\beta} V) &\simeq
(V^{\gamma}\otimes I_{V^{\gamma}}(V^{\alpha}\otimes V^{\beta}))
\otimes
(^{\gamma}V\otimes I^*_{V^{\gamma}}(V^{\alpha}\otimes V^{\beta})) \\
&\xrightarrow{con} V^{\gamma}\otimes{}^{\gamma}V
\end{align*}

\subsubsection{}
\label{convention of copies of groups}
We introduce copies of $S$ indexed by some
set $X$. For an element $x\in X$ the copy is denoted as $S_x$.
To distinguish the right and left actions of $S$, we use the notation
${}_xS_y$ for
the algebraic group $S$. On the group ${}_xS_y$, $S_x$ and $S_y$
acts from the left and the right. The coordinate ring of ${}_xS_y$
is denoted by ${}_x\Cal O_y$. 
For an element $\alpha\in Irr(S)$, the copy of $V^{\alpha}$
considered as a left $S_x$-module is denoted by
${}_xV^{\alpha}$.
The dual of ${}_xV^{\alpha}$ as a right $S_x$-module is denoted
by ${}^\alpha V_{x}$.
Thus we have an identity
$\displaystyle
{}_x\Cal O_y=\underset{\alpha}\oplus{}_xV^\alpha \otimes {}^{\alpha}V_y.
$
The direct sum is taken over the set $Irr(S)$.

\subsubsection{}
Let $M$ and $N$ be complexes of $\bold k$-vector spaces.
The set of homogeneous maps of degree $n$ from $M$ to $N$
is denoted by $\underline{Hom}^n_{Vect_{\bold k}}(M, N)$.
For an element $f$ in $\underline{Hom}^n_{Vect_{\bold k}}(M, N)$,
the differential  $\partial(f)$ of $f$ is defined by
$$
(\partial(f)(x)=d(f(x))-(-1)^{\deg(f)}f(dx).
$$
Let $e^n$ be the degree $(-n)$-element in the complex $\bold k[n]$
corresponding to $1$.
We define the degree $(-n)$ map $t^n$ in 
$\underline{Hom}^{-n}(\bold k, \bold k[n])$ by
$t(1)=e^n$. We set $t=t^1$ and $e=e^1$. The complex $K\otimes \bold k[n]$ is
denoted by $Ke^n$.

For two complexes $K$ and $L$, we define the complex $K\otimes L$ by 
usual sign rule,
$$
d_{L\otimes K}(x\otimes y)=dx\otimes y+(-1)^{\deg(x)}x\otimes dy.
$$
This rule is applied to a complex object in an abelian category with
tensor product.

For two homogeneous map of complexes $f_1\in \underline{Hom}^i(K_1,L_1)$ and
$f_2\in \underline{Hom}^j(K_2,L_2)$, we define the homogeneous map
$f_1\otimes f_2 \in \underline{Hom}^j(K_1\otimes K_2 ,L_1\otimes L_2)$ by
$$
(f_1\otimes f_2)(k_1\otimes
k_2)=(-1)^{\deg(f_2)\deg(k_1)}f_1(k_1)\otimes f_2(k_2)
$$
for $k_1\in K_1, k_2 \in K_2$.
Thus we have an isomorphism of complexes:
$$
\underline{Hom}^{\bullet}(K_1,L_1)\otimes \underline{Hom}^{\bullet}(K_2,L_2)
\to
\underline{Hom}^{\bullet}(K_1\otimes K_2 ,L_1\otimes L_2).
$$
For example, we have the following isomorphism
\begin{align}
\label{shift and homomorphism and sgn}
&\underline{Hom}^{\bullet+n-m}(K,L)\otimes 
\bold kt^{n-m} \\
\nonumber
=&
\underline{Hom}^{\bullet}(K,L)\otimes 
\underline{Hom}^{\bullet}(\bold
ke^m,\bold ke^n)\to
\underline{Hom}^{\bullet}(Ke^m,Le^n).
\end{align}
Using this rule, the differential of $K\otimes L$ is written as
$d\otimes 1 + 1\otimes d$.

A complex object 
$\{K^i=(K^i,\delta),d^i:K^i \to K^{i+1}\}$ in the category of 
complex is called a double
complex. Then $d^i\otimes t^{-1}:K^ie^{-i}\to K^{i+1}e^{-i-1}$
is a homogeneous map of degree one, and
the sum $\oplus_i K^ie^{-i}$ becomes a complex with the
differential $\delta\otimes 1+ d\otimes t^{-1}$, which is called
the associate simple complex and denoted as $s(K^{\bullet})$.
Using the above sign rules, we have a natural isomorphism
$$
s(K\otimes L)\simeq s(K)\otimes s(L):
(K^i\otimes L^j)e^{-i-j}\simeq K^ie^{-i} \otimes L^{-j} e^{-j}
$$
for two double complexes $K$ and $L$.
\subsubsection{} 
\label{convention of orientation}
Let $K$ and $L$ be complexes. Then we define an isomorphism of complex
$\sigma:K\otimes L \to L\otimes K$ by
\begin{equation}
\label{transposition}
\sigma(x\otimes y)=(-1)^{\deg(x)\deg(y)}y\otimes x.
\end{equation}
This is called a transposition homomorphism.
On a complex $K^{\otimes n}=\underset{n}{\underbrace{K\otimes \cdots \otimes K}}$, the
transposition of $i$-th and $(i+1)$-th component is denoted as
$\sigma_{i,i+1}$. It is easy to show that this action can be extended to 
the action of $\frak S_{n}$. As a consequence, $K^{\otimes n}$ becomes a
$\bold k[\frak S_n]$-module. If $char(\bold k)>n$, then
the symmetric part is the image of
$
\frac{1}{n!}\sum_{\sigma\in \frak S_n}\sigma.
$

\subsubsection{}
\label{orientation is introduced here}
We assume that $char(\bold k)=0$.
Let $A$ be a finite set and $n=\# A$. For each element $a \in A$, we prepare element
$e_a$ of degree $-1$ and we consider a complex
$K=\oplus_{a\in A}\bold ke_a$. The symmetric part $\Lambda(A)$ of
$K^{\otimes n}$ 
under the action of $\frak S_n$ is
isomorphic to $\bold k[n]$ by choosing an ``orientation of $A$''.
If $A=[1,n]$, then $\bold k[n]$ is generated by $e_1\wedge\dots\wedge e_n$.
The group of bijection of $A$ is denoted as $\frak S[A]$.
On the complex $\Lambda(A)$, the group $\frak S[A]$ acts via the sign.
Let $K^*=\oplus_{a\in A}\bold kf_a$ be a complex generated by degree
$1$-element
$f_a$ indexed by $A$. The symmetric part of $(K^*)^{\otimes n}$ is
denoted as $\Lambda^*(A)$. We have the canonical isomorphism
$\Lambda(A\coprod B)=\Lambda(A)\otimes \Lambda(B)$.
The complex $\Lambda([1,n])$ is denoted as $\Lambda_n$.

\section{Relative DGA and relative bar construction}
\label{section:relative DGA and relative bar}
In this section, we define a relative DGA and
relative bar constructions. For a relative DGA $A$ and its
relative augmentation $\epsilon$, we define two bar complexes
$Bar(A/\Cal O,\epsilon)$ and $Bar(A/\Cal O,\epsilon)_{simp}$.
The latter one is called the simplicial relative bar complex.
In \S \ref{comparison of two complexes}, we prove that they are quasi-isomorphic.
In the last subsection, we define coproduct structures on bar complexes.

\subsection{Relative DGA}
\subsubsection{}
Let $S$ be a reductive group over $\bold k$ and $\Cal O_S$ be
the coordinate ring of $S$. The coproduct structure on
$\Cal O_S$ is denoted by 
$\Delta_S:\Cal O_S \to\Cal O_S \otimes \Cal O_S$. 
We define a relative DGA's over $\Cal O_S$ as follows.
\begin{definition}
\label{definition of relative DGA}
Let $A$ be a complex. 
A data of 5-ple $(A, \Delta_A^l,
\Delta_{A}^r, i, \mu_A)$
consisting of
\begin{enumerate}
\item
(the coactions)
a homomorphisms of complexes
$\Delta_A^l:A \to \Cal O_S \otimes A$
and $\Delta_{A}^r:A \to A \otimes \Cal O_S$,
\item
\label{O-structure of A}
a homomorphism of complexes
compatible with the bi-$\Cal O_S$ comodule structures
$i:\Cal O_S\to A$ , and
\item
(multiplication)
a homomorphism
of complexes
$\mu_A:A\otimes_{\Cal O_S}A \to A$ 
\end{enumerate}
is called a relative DGA over $\Cal O_S$, if the following conditions
 are satisfied.

\begin{enumerate}
\item the multiplication $\mu_A$ is associative, and
\item
the left and the right coactions of $\Cal O_S$ are coassociative and counitary.
\end{enumerate}
\end{definition}
\subsubsection{}
We set $\Cal O=\Cal O_S$.
We use notations $V^\alpha$ for irreducible representations
as in \S \ref{many copy of group its elements, etc}.
We set $A^{\alpha\beta}=Hom_{\Cal O}
(V^{\alpha},A\otimes_{\Cal O} V^{\beta})$.
Since $S$ is split, we have 
$$
A=\oplus_{\alpha,\beta}V^{\alpha}\otimes A^{\alpha\beta}\otimes
{}^{\beta}V.
$$
Under this decomposition, we have
$$
A\otimes_{\Cal O}A\simeq 
\oplus_{\alpha,\beta,\gamma}V^{\alpha}
\otimes A^{\alpha\beta}
\otimes A^{\beta\gamma}
\otimes {}^{\gamma}V
$$
and the multiplication map
is obtained by 
\begin{equation}
\label{from rev comp to multiplication}
V^{\alpha}\otimes A^{\alpha\beta}\otimes 
A^{\beta\gamma}\otimes V^{\gamma}
\xrightarrow{1\otimes \eta\otimes 1} 
V^{\alpha}\otimes A^{\alpha\gamma}
\otimes V^{\gamma}
\end{equation}
where $\eta:A^{\alpha,\beta}\otimes A^{\beta,\gamma}\to
A^{\alpha,\gamma}$ 
is the map defined by
\begin{align*}
&Hom_{\Cal O}
(V^{\alpha},A\otimes_{\Cal O} V^{\beta})\otimes
Hom_{\Cal O}
(V^{\beta},A\otimes_{\Cal O} V^{\gamma}) \\
\xrightarrow{\sigma} &Hom_{\Cal O}
(V^{\beta},A\otimes_{\Cal O} V^{\gamma})\otimes 
Hom_{\Cal O}(V^{\alpha},A\otimes_{\Cal O} V^{\beta})
\\
\to &
Hom_{\Cal O}
(A\otimes_{\Cal O}V^{\beta},A\otimes_{\Cal O}A\otimes_{\Cal O} V^{\gamma})\otimes 
Hom_{\Cal O}(V^{\alpha},A\otimes_{\Cal O} V^{\beta})
\\
\to &Hom_{\Cal O}
(V^{\alpha},A\otimes_{\Cal O}A\otimes_{\Cal O} V^{\gamma}) \\
\to &Hom_{\Cal O}
(V^{\alpha},A\otimes_{\Cal O} V^{\gamma}). 
\end{align*}
Here $\sigma$ is the transposition (\ref{transposition}).
The left $\Cal O$-structure $\Delta^l$
is the direct sum of the map
$$
A^{\alpha\beta}\to {}^{\alpha}V\otimes V^{\alpha}\otimes 
A^{\alpha\beta}:x\mapsto 
\delta_{\alpha}\otimes x,
$$
where 
$\{b_i\}$ and $\{b_i^*\}$ are dual bases of 
${}^{\alpha}V$ and $V^{\alpha}$,
and
$\delta_{\alpha}=\sum_ib_i\otimes b_i^*$.
The map $\Delta^r$
can be written similarly.
The natural map $A\otimes_{\Cal O} A \to A\otimes A$
is identified with
\begin{align}
\label{natural map for cotensor}
V^{\alpha}
\otimes A^{\alpha\beta}
\otimes A^{\beta\gamma}
\otimes {}^{\gamma}V  &\to
V^{\alpha}
\otimes A^{\alpha\beta}
\otimes {}^{\beta}V
\otimes V^{\beta}
\otimes A^{\beta\gamma}
\otimes {}^{\gamma}V  \\
\nonumber
x_1\otimes x_2\otimes y_1\otimes y_2 & \mapsto
x_1\otimes x_2\otimes \delta_{\beta} \otimes y_1\otimes y_2.
\end{align}
This natural map is written as 
$\big[x\otimes y\big]\mapsto \big[x]\otimes\big[y]$.

\subsubsection{Example. Relative DGA associated to $\rho:G\to S(\bold k)$}
\label{example of relative DGA associated to group}
Let $G$ be a group, $S$ a
reductive group over a field $\bold k$ and
$\rho:G\to S(\bold k)$ a group homomorphism whose image is Zariski dense.
We give an example of relative DGA associated to $\rho$.
Let $V_1, V_2$ be algebraic representations of $S$.
The following complex is called the canonical cochain complex: 
\begin{align*}
\underline{Hom}_G(V_1,V_2):0\to Hom_{\bold k}(V_1, V_2) & \to 
Hom_{\bold k}(\bold k[G]\otimes V_1, V_2)  \\
& \to
Hom_{\bold k}(\bold k[G]\otimes \bold k[G]\otimes V_1, V_2) \to
\end{align*}
where the differential is given by
\begin{align*}
& d(\varphi)(g_1\otimes \cdots \otimes g_{p+1} \otimes v_1) \\
=&g_1\varphi(g_2\otimes \cdots \otimes g_{p+1}\otimes v_1)-
\varphi(g_1g_2\otimes \cdots \otimes g_{p+1}\otimes v_1)\\
&+
\varphi(g_1\otimes g_2g_3\otimes \cdots \otimes g_{p+1}\otimes v_1)-\cdots \\
&\pm \varphi(g_1\otimes \cdots \otimes g_pg_{p+1}\otimes v_1)
\mp \varphi(g_1\otimes \cdots \otimes g_p\otimes g_{p+1}v_1).
\end{align*}

Then the extension group
$Ext^i_G(V_1,V_2)$ of $V_1$ and $V_2$ 
as $G$-modules
is the $i$-th cohomology group of
$\underline{Hom}_G(V_1,V_2)$. We define the multiplication map
$$
\underline{Hom}_G(V_1,V_2)\otimes \underline{Hom}_G(V_2, V_3) \to
\underline{Hom}_G(V_1, V_3).
$$
by
\begin{align}
\label{Yoneda pair complex level}
\varphi\otimes \psi \mapsto 
&\bigg[g_1\otimes \cdots \otimes g_j\otimes g_{j+1}
\otimes \cdots \otimes g_{i+j}\otimes v_1 \\
\nonumber
& \mapsto
(-1)^{\deg(\varphi)\deg(\psi)}
\psi(g_1\otimes \cdots \otimes g_j\otimes
\varphi(g_{j+1}\otimes \cdots \otimes g_{i+j} \otimes v_1))\bigg]
\end{align}

Let $\Cal O_S$ be the coordinate ring of $S$. 
The module $\Cal O_S$ as a left $S$-module is written as $L\Cal O_S$.
We set $A=\underline{Hom}_G(L\Cal O_S, L\Cal O_S)$.
Using the right $\Cal O_S$-comodule structure of $L\Cal O_S$,
the complex $A$ becomes a bi-$\Cal O_S$-comodule.
The multiplication map $A\otimes A\to A$ induces a map
$\eta:A\otimes_{\Cal O}A \to A$, which is also called the 
multiplication map of $A$. Then $A$ becomes a relative DGA over $\Cal O$.
The natural map $\epsilon:A\to Hom_{\bold k}(L\Cal O,L\Cal O)$
is called the relative augmentation of $A$.

\subsection{Relative bar complex and simplicial relative bar complex}
\label{subsec:relative bar complex}
\subsubsection{}
Let $A$ be a relative DGA over $\Cal O$.
We use notations for copies of a reductive group in \S \ref{convention of copies of groups}.
A homomorphism of complexes 
$$
{}_xA_{y}\overset{\epsilon}\to 
Hom_{\bold k}({}_z\Cal O^{(1)}_{x},{}_z\Cal O^{(2)}_{y})
$$
is called a relative augmentation if 
\begin{enumerate}
\item
it is a homomorphism
of bi-$\Cal O$ comodules, and
\item
the multiplication
$A\otimes_{\Cal O} A \to A$
is compatible with the multiplication homomorphism 
of $Hom_{\bold k}(\Cal O,\Cal O)$
which is defined
by the composite of
\begin{align*}
&Hom_{\bold k}(\Cal O,\Cal O)\otimes_{\Cal O}
Hom_{\bold k}(\Cal O,\Cal O)  \\
\overset{\sigma}\to
&Hom_{\bold k}(\Cal O,\Cal O)\otimes
Hom_{\bold k}(\Cal O,\Cal O) \to
Hom_{\bold k}(\Cal O,\Cal O),
\end{align*}
where $\sigma$ is the transposition
defined in (\ref{transposition}).
\end{enumerate}
Using the relative augmentations,
we have two homomorphisms
$\epsilon^l$ and $\epsilon^r$ by
\begin{align*}
&\epsilon^l:{}_xA_{y}\to 
 {}_xA_{y} \otimes{}_y\Cal O_y^{*(2)}\overset{\epsilon}\to 
Hom_k({}_y\Cal O_x^{(1)}, {}_y\Cal O_y^{(2)})
\otimes {}_y\Cal O_y^{*(2)}
\overset{pair}\to {}_y\Cal O_x^{*(1)}\simeq
{}_x\Cal O_y^{(1)} \\
&\epsilon^r:{}_xA_{y}\to {}_x\Cal O_x^{(1)}\otimes {}_xA_{y} \overset{\epsilon}\to 
{}_x\Cal O_x^{(1)}\otimes Hom_k({}_x\Cal O_x^{(1)}, {}_x\Cal O_y^{(2)})
\overset{ev}\to {}_x\Cal O_y^{(2)}.
\end{align*}
The map $\epsilon^l$ and $\epsilon^r$ are called the left and right
augmentations.
\subsubsection{}
Since $S$ is split, we have
${}_xA_y=\underset{\alpha,\beta}\oplus({}_xV^{\alpha}
\otimes A^{\alpha\beta}\otimes
{}^{\beta}V_y)$.
Then the relative augmentation induces a 
family of homomorphisms
\begin{align}
\label{pre-augmentation componentwise}
A^{\alpha\beta}\to 
&Hom_{\Cal O}(V^{\alpha}, Hom_{\bold k}(\Cal O,\Cal O)\otimes_{\Cal O}V^{\beta}) \\
%
\nonumber
&=Hom_{\bold k}(V^{\alpha},V^{\beta}),
\end{align}
indexed by $\alpha,\beta$.
Via these homomorphisms, the map 
$\eta:A^{\alpha,\beta}\otimes A^{\beta,\gamma}\to A^{\alpha,\gamma}$
is compatible with the map
$$
Hom_{\bold k}(V^{\alpha},V^{\beta})\otimes 
Hom_{\bold k}(V^{\beta},V^{\gamma})\to
Hom_{\bold k}(V^{\alpha},V^{\gamma}):
x\otimes y \mapsto y\circ x.
$$
The map $\epsilon^l$ is written as follows.
\begin{align*}
\epsilon^l:
{}_xV^\alpha\otimes A^{\alpha\beta}
\otimes {}^{\beta}V_{y} 
&\to
{}_xV^\alpha\otimes 
{}^{\alpha}V_{y}\otimes {}_yV^{\beta}
\otimes {}^{\beta}V_y \\ 
&\to
{}_xV^\alpha\otimes 
{}^{\alpha}V_{y}.
\end{align*}

\subsubsection{}
We introduce a relative bar complex 
$Bar(A/\Cal O, \epsilon)$.
For an integer $n\geq 1$, we define 
$Bar_{n}=Bar(A/\Cal O, \epsilon)_{n}$ by
$$
Bar(A/\Cal O, \epsilon)_{n}=
\underset{n}{\underbrace{
A
\otimes_{\Cal O}\cdots
\otimes_{\Cal O}A}}.
$$
and $Bar_0=\Cal O$.
We have a sequence of homomorphisms
\begin{equation}
\label{relative bar double complex}
\bold{Bar}(A/\Cal O, \epsilon):
\cdots \to Bar_{n} \to \cdots \to Bar_{1} \to Bar_0 \to 0. 
\end{equation}
given by
\begin{align*}
x_1\otimes\cdots \otimes x_n \mapsto 
& \epsilon^r(x_1)\otimes x_2 \otimes \cdots \otimes x_n+
\sum_{i=1}^{n-1}(-1)^i
x_1\otimes\cdots \otimes x_ix_{i+1} \otimes\cdots \otimes
x_n \\
& +(-1)^n x_1\otimes\cdots \otimes \epsilon^l(x_n).
\end{align*}
Here the multiplication map $\mu:A\otimes_{\Cal O} A \to A$ 
is denoted by $x\otimes y \mapsto xy$.
\begin{proposition}
\label{complex for relative bar construction}
The sequence (\ref{relative bar double complex}) is a complex.
\end{proposition}
\begin{proof}
This is a consequence of the associativity of the multiplication of
$A$ and the following commutativity arising from the axiom (3).

$$
\begin{matrix}
\Cal O\otimes_{\Cal O}A\otimes_{\Cal O} A
&\overset{1\otimes \mu}\to &
\Cal O\otimes_{\Cal O}A \\
\epsilon^r\otimes 1\downarrow 
& & \downarrow\epsilon^r \\
\Cal O\otimes_{\Cal O}A
&\overset{\epsilon^r}\to &
\Cal O \\
\end{matrix}
\qquad
\begin{matrix}
A\otimes_{\Cal O}A\otimes_{\Cal O} \Cal O
&\overset{\mu\otimes 1}\to &
A\otimes_{\Cal O}\Cal O \\
1\otimes \epsilon^l\downarrow 
& & \downarrow\epsilon^l \\
A\otimes_{\Cal O}\Cal O
&\overset{\epsilon^l}\to &
\Cal O \\
\end{matrix}
$$
\end{proof}
The complex $Bar(A/\Cal O,\epsilon)$ is defined by
the associate simple complex of 
$\bold{Bar}(A/\Cal O,\epsilon)$.
\subsubsection{}
We define a relative simplicial bar complex 
$Bar_{simp}=Bar_{simp}(A/\Cal O,\epsilon)$.
For a sequence of integers 
$\alpha=(\alpha_0<\alpha_1<\cdots <\alpha_n)$,
we define a complex 
$Bar_{simp}^{\alpha}=
Bar_{simp}^{\alpha}(A/\Cal O,\epsilon)$ by
$\underset{n}{\underbrace{A\otimes_{\Cal O}
\cdots \otimes_{\Cal O}A}}$.
It will be denoted as 
$$
\Cal O \overset{\alpha_0}\otimes_{\Cal O}
A
\overset{\alpha_1}\otimes_{\Cal O}
\cdots 
\overset{\alpha_{n-1}}\otimes_{\Cal O}A
\overset{\alpha_{n}}\otimes_{\Cal O}\Cal O
$$
to distinguish the index $\alpha$.
If $\#\beta=\#\alpha+1$ and $\alpha$ is a subsequence of
$\beta$, we define a homomorphism of complexes
$
Bar_{simp}^{\beta} \to Bar_{simp}^{\alpha}
$
as follows.

(1) If 
$\beta=(\alpha_0< \cdots \alpha_i<b<\alpha_{i+1} \cdots <\alpha_n)$,
then the map $\partial_{\beta,i+1}$ is given by
\begin{align*}
&
y_0
\overset{\alpha_0}\otimes x_1
\overset{\alpha_1}\otimes \cdots x_i
\overset{\alpha_i}\otimes y
\overset{\beta}\otimes x_{i+1}
\overset{\alpha_{i+1}}\otimes
x_{i+2}
\cdots
\overset{\alpha_{n-1}}\otimes x_n  
\overset{\alpha_{n}}\otimes y_{n+1}.
\\
&\mapsto 
y_0
\overset{\alpha_0}\otimes
x_1
\overset{\alpha_1}\otimes \cdots x_i
\overset{\alpha_i}\otimes y
x_{i+1}
\overset{\alpha_{i+1}}\otimes
x_{i+2}
\cdots
\overset{\alpha_{n-1}}\otimes x_n
\overset{\alpha_{n}}\otimes y_{n+1}.
\end{align*}

(2) If
$\beta=(b<\alpha_0< \cdots  <\alpha_n)$
(resp. $\beta=(\alpha_0< \cdots  <\alpha_n<b)$), 
the map 
$\partial_{\beta,0}$ (resp. $\partial_{\beta,{n+1}}$)
is given by
\begin{align*}
&
y_0
\overset{\alpha_0}\otimes 
x_1
\overset{\alpha_1}\otimes 
\cdots
\overset{\alpha_{n-1}}\otimes x_n  
\overset{\alpha_{n}}\otimes y_{n+1}
\mapsto 
\epsilon^r(y_0\otimes x_1)
\overset{\alpha_1}\otimes \cdots 
\overset{\alpha_{n-1}}\otimes x_n 
\overset{\alpha_{n}}\otimes y_{n+1} \\
\text{(resp. } &
y_0
\overset{\alpha_0}\otimes 
x_1
\overset{\alpha_1}\otimes 
\cdots
\overset{\alpha_{n-1}}\otimes x_n  
\overset{\alpha_{n}}\otimes y_{n+1}
\mapsto 
y_0\otimes x_1
\overset{\alpha_1}\otimes \cdots 
\overset{\alpha_{n-1}}\otimes 
\epsilon^l(x_n 
\overset{\alpha_{n}}\otimes y_{n+1})
\text{)}.
\end{align*}
We define $\bold {Bar}_{simp,n}$ by
$$
\bold {Bar}_{simp,n}=\oplus_{\alpha=(\alpha_0<\cdots
<\alpha_n)}Bar^{\alpha}_{simp}
$$
and the map $\partial_{n}:\bold {Bar}_{simp,n}
\to \bold {Bar}_{simp,n-1}$
by
\begin{equation}
\label{outer differential for simplical bar complex}
\partial_n=\sum_{\alpha_0<\cdots <\alpha_n}
\sum_{i=0}^n(-1)^i\partial_{\alpha,i}.
\end{equation}
We can prove the following proposition as in
Proposition \ref{complex for relative bar construction}.
\begin{proposition}
The sequence
\begin{equation}
\bold {Bar}_{simp}:
\cdots \to \bold {Bar}_{simp,n} \to
\cdots \to \bold {Bar}_{simp,1} \to
\bold {Bar}_{simp,0} \to 0
\end{equation}
is a double complex.
\end{proposition}
\begin{definition}
The associate simple complex of 
$\bold {Bar}_{simp}(A/\Cal O,\epsilon)$
is called the relative simplicial bar complex and denoted as
$Bar_{simp}(A/\Cal O,\epsilon)$.
\end{definition}
\subsection{Comparison of two complexes}
\label{comparison of two complexes}
In this subsection, we compare the relative bar complex and
the relative simplicial bar complex for a relative DGA $A$ 
over $\Cal O$.
\subsubsection{}
We introduce two double complexes 
$\widetilde{\bold {Bar}}=\widetilde{\bold {Bar}}(A/\Cal O)$ 
and
$\widetilde{\bold {Bar}}^+=\widetilde{\bold {Bar}}^+(A/\Cal O)$
as follows. 
For $n\geq 0$, we define 
$$
\widetilde{{Bar}}_{n} =
\underset{n+2}{\underbrace{
A\otimes_\Cal O \cdots
\otimes_\Cal O A}}
$$
and $d_n:\widetilde{{Bar}}_{n}\to
\widetilde{{Bar}}_{n-1}$ for $n\geq 1$ by
\begin{align*}
d_n(x_0\otimes x_1\otimes\cdots \otimes x_n\otimes x_{n+1}) 
= 
\sum_{i=0}^{n}(-1)^i
x_0\otimes\cdots \otimes x_ix_{i+1} \otimes\cdots \otimes
x_{n+1}.
\end{align*}
The free bar complex $\widetilde{\bold {Bar}}$
is defined by
$$
\cdots 
\to 
\widetilde{{Bar}}_{n}\to\cdots
\to 
\widetilde{{Bar}}_{1}
\to 
\widetilde{{Bar}}_{0}\to 0.
$$
We define the augmentation morphism 
$d_0:\widetilde{Bar}_0=A\otimes_{\Cal O} A\to A$ 
by the multiplication map $\mu$.
We define an augmented free bar complex by
$$
\widetilde{\bold {Bar}}^+=
(\widetilde{\bold {Bar}} \overset{d_0}\to A).
$$
\begin{proposition}
The double complex $\widetilde{\bold {Bar}}^+$ is exact.
\end{proposition}
\begin{proof}
We prove the proposition by constructing a homotopy.
We define 
\begin{align*}
\theta_n:\widetilde{Bar}_{n}=
\underset{n+2}{\underbrace{A\otimes_{\Cal O}\cdots
\otimes_{\Cal O}A}}
&=
\Cal O\otimes_{\Cal O}\underset{n+2}
{\underbrace{A\otimes_{\Cal O}\cdots
\otimes_{\Cal O}A}} \\
&
\xrightarrow[]{i\otimes 1_A\otimes \cdots \otimes 1_A}
\underset{n+3}{\underbrace{A\otimes_{\Cal O}\cdots
\otimes_{\Cal O}A}}=
\widetilde{Bar}_{n+1}
\end{align*}
for $n\geq 0$ and
$$
\theta_{-1}:A=\Cal O\otimes_{\Cal O} A
\xrightarrow[]{i\otimes 1_A}A\otimes_{\Cal O}A
= \widetilde{Bar}_{0}.
$$
Using the equality 
$i\otimes 1=\Delta_A^l:A\to A\otimes_{\Cal O}A$,
and the counitarity of the left coaction 
of $\Cal O$ on $A$,
we have $\mu\Delta^l=1_A$. As a consequence, we have
\begin{align}
\label{homotopy relation}
&\theta_{n-1}d_n+d_{n+1}\theta_{n}=1 \quad (\text{ for }n\geq 0), \\
\nonumber
&d_0\theta_{-1}=1.
\end{align}
Thus we have the proposition.
\end{proof}
\begin{corollary}
\label{bifree resolution}
The complex $\widetilde{\bold {Bar}(A/\Cal O)}
\xrightarrow[]{d_0} A$
is a free $A\otimes A^{\circ}$- resolution of $A$.
\end{corollary}
\subsubsection{}
We define a simplicial free bar complex 
$\widetilde{\bold {Bar}}_{simp}
=\widetilde{\bold {Bar}}_{simp}(A/\Cal O)$
and an augmented 
free simplicial bar complex
$\widetilde{\bold {Bar}}_{simp}^+=
\widetilde{\bold {Bar}}_{simp}^+(A/\Cal O)$.
Let $\alpha=(\alpha_0<\cdots <\alpha_n)$ be
a sequence of integers.
We define $\widetilde{Bar_{simp}}^{\alpha}$
by
$$
\widetilde{Bar_{simp}}^{\alpha}=
A
\overset{\alpha_0}\otimes_{\Cal O} \cdots
\overset{\alpha_n}\otimes_{\Cal O} A.
$$
Let $0\leq p\leq n$.
We define 
$
\partial_{\alpha}:\widetilde{Bar_{simp}}^{\alpha}\to
\widetilde{Bar_{simp}}
$ by
$$
\partial_{\alpha}(
x_0
\overset{\alpha_0}\otimes x_1 \cdots
x_n \overset{\alpha_n}\otimes 
x_{n+1})=\sum_{p=0}^n
(-1)^p
(x_0
\overset{\alpha_0}\otimes \cdots 
\overset{\alpha_{p-1}}\otimes
x_px_{p+1}
\overset{\alpha_{p+1}}\otimes
\cdots
\overset{\alpha_n}\otimes 
x_{n+1}).
$$
We set 
$$
\widetilde{\bold {Bar}_{simp,n}}=\oplus_{\alpha_0<\cdots <\alpha_n}
\widetilde{{Bar}_{simp}}^{\alpha}
$$
By taking the summation on $\alpha$, we have a sequence
$$
\widetilde{\bold {Bar}_{simp}}:
\cdots \to \widetilde{\bold {Bar}_{simp,n}}
\xrightarrow[]{d_n} 
\cdots \xrightarrow[]{d_2} 
 \widetilde{\bold {Bar}_{simp,1}}
\xrightarrow[]{d_1} 
\widetilde{\bold {Bar}_{simp,0}}\to 0.
$$
We can check that $\widetilde{\bold {Bar}}_{simp}$
becomes a complex.
We define the augmentation map 
$d_0:\widetilde{Bar}_{simp}^0 \to A$ by the sum of
the multiplication map 
$\mu:A\overset{\alpha_0}\otimes_{\Cal O}A \to A $.
We define the double complex
$$
\widetilde{\bold {Bar}_{simp}}^+=
(\widetilde{\bold {Bar}_{simp}}
\xrightarrow{d_0} A).
$$
\begin{proposition}
\label{acyclicity for simplicial bar complex}
The double complex $\widetilde{\bold {Bar}_{simp}}^+$
is exact.
\end{proposition}
\begin{proof}
To prove the proposition, we define a
subcomplex $\widetilde{\bold {Bar}}_{N<,simp}$ by
$$
\widetilde{{Bar}}_{N<,simp,n}=
\oplus_{N<\alpha_0<\cdots<\alpha_n}
\widetilde{Bar}^{\alpha_0,\dots, \alpha_n}_{simp}.
$$
We define a map $\theta_n^{\alpha}$ by 
\begin{align*}
\theta_n^{\alpha}:&\widetilde{Bar}^{\alpha_0,\dots, \alpha_n}_{simp}
=
A\overset{\alpha_0}\otimes_{\Cal O} \cdots
\overset{\alpha_n}\otimes_{\Cal O} A= 
\Cal O\overset{N}\otimes_{\Cal O}
A\overset{\alpha_0}\otimes_{\Cal O} \cdots
\overset{\alpha_n}\otimes_{\Cal O} A \\
&
\xrightarrow[]{i\otimes 1\otimes \cdots \otimes 1}
A\overset{N}\otimes_{\Cal O}
A\overset{\alpha_0}\otimes_{\Cal O} \cdots
\overset{\alpha_n}\otimes_{\Cal O} A
=
\widetilde{Bar}^{N,\alpha_0,\dots, \alpha_n}_{simp}\subset
\widetilde{{Bar}}^{n+1}_{N-1<,simp}.
\end{align*}
By taking the summation on $\alpha$, we have
a map
$$
\theta_n:\widetilde{{Bar}}_{N<,simp,n}
\to\widetilde{{Bar}}_{N-1<,simp,n+1}
$$
We define a map 
$\theta_{-1}:A=\Cal O\otimes_{\Cal O} A \to 
A\overset{N}\otimes_{\Cal O} A$
by $i\otimes 1=\Delta^l$.
Then we have the homotopy relation (\ref{homotopy relation}).
Since 
$$
\widetilde{\bold {Bar}}_{simp}^+=
\underset{\xrightarrow[N]{}}{\lim}\ 
\widetilde{\bold {Bar}}_{N<,simp}^+,
$$
we proved the proposition.
\end{proof}
\begin{corollary}
\label{bifree resolution for simp}
The complex $\widetilde{\bold {Bar}}_{simp}(A/\Cal O)
\xrightarrow[]{d_0} A$
is a free $A\otimes A^{\circ}$- resolution of $A$.
\end{corollary}
\subsubsection{}
\label{comparizon for bar and simplicial bar}
We define a homomorphism
$
\sigma:\widetilde{\bold {Bar}}_{simp}(A/\Cal O)\to
\widetilde{\bold {Bar}}(A/\Cal O)
$
of double complexes by
$$
\sigma(
x_0
\overset{\alpha_0}\otimes  \cdots
 \overset{\alpha_n}\otimes 
x_{n+1})=
x_0
\otimes  \cdots
\otimes 
x_{n+1}.
$$
Then the homomorphism $\sigma$ commutes with the augmentations
$d_0$.
By Corollary \ref{bifree resolution}
 and 
\ref{bifree resolution for simp}
, the homomorphism
$$
\Cal O\otimes_{\epsilon^r,A}\widetilde{\bold {Bar}}_{simp}
(A/\Cal O)\otimes_{\epsilon^l,A} \Cal O
\to
\Cal O\otimes_{\epsilon^r,A}
\widetilde{\bold {Bar}}(A/\Cal O)
\otimes_{\epsilon^l,A} \Cal O
$$
is a quasi-isomorphism.
Since
$$
\Cal O\otimes_{\epsilon^r,A}\widetilde{\bold {Bar}}_{simp}
(A/\Cal O)\otimes_{\epsilon^l,A} \Cal O \simeq
\bold {Bar}_{simp}
(A/\Cal O,\epsilon)
$$
and
$$
\Cal O\otimes_{\epsilon^r,A}\widetilde{\bold {Bar}}
(A/\Cal O)\otimes_{\epsilon^l,A} \Cal O \simeq
\bold {Bar}
(A/\Cal O,\epsilon),
$$
we have the following theorem
\begin{theorem}
The natural map
$$
\sigma:\bold {Bar}_{simp}
(A/\Cal O,\epsilon) \to
\bold {Bar}
(A/\Cal O,\epsilon)
$$
is a quasi-isomorphism.
\end{theorem}

\subsection{Coproducts on bar complexes}
In this subsection, we introduce a coproduct structure on 
bar complexes.
\subsubsection{}
We define a homomorphism
$$
\underset{n+m}{\underbrace{A\otimes_{\Cal O}
\cdots \otimes_{\Cal O}A}}\to
(\underset{n}{\underbrace{A\otimes_{\Cal O} 
\cdots \otimes_{\Cal O}A}})
\otimes
(\underset{m}{\underbrace{A\otimes_{\Cal O} 
\cdots \otimes_{\Cal O}A}})
$$
by applying the natural map $A\otimes_{\Cal O}A \to A\otimes A$
for $n$-th and $n+1$-th tensor components.
This map is written as 
$$
\big[x_1\otimes \cdots \otimes x_{n+m}\big] \mapsto
\big[x_1\otimes \cdots \otimes x_{n}\big] \otimes
\big[x_{n+1}\otimes \cdots \otimes x_{n+m}\big]
$$
As for this notation, see also
(\ref{natural map for cotensor}).
The map
\begin{align*}
& A\otimes_{\Cal O}\cdots \otimes_{\Cal O} A
\to
\Cal O\otimes (A\otimes_{\Cal O}\cdots \otimes_{\Cal O} A) \\
 \text{(resp. }&
A\otimes_{\Cal O}\cdots \otimes_{\Cal O} A
\to
(A\otimes_{\Cal O}\cdots \otimes_{\Cal O} A)
\otimes \Cal O 
\text{)}
\end{align*}
obtained by the left (resp. right) 
$\Cal O$ coproduct of $A$ at the first
(resp. the last)
factor is written as
\begin{align*}
&x_1\otimes \cdots \otimes x_{n+m} \mapsto
\Delta^l(x_1)\otimes \cdots \otimes x_{n+m} \\
(\text{resp.} 
&x_1\otimes \cdots \otimes x_{n+m} \mapsto
x_1\otimes \cdots \otimes \Delta^r(x_{n+m})
\quad \text{)}. 
\end{align*}

We introduce a coalgebra structure 
\begin{equation}
\label{comultiplication for bar}
\Delta_{\bold B}:
{\bold {Bar}}(A/\Cal O,\epsilon)\to
{\bold {Bar}}(A/\Cal O,\epsilon)\otimes 
{\bold {Bar}}(A/\Cal O,\epsilon)
\end{equation}
by
\begin{align*}
\Delta_{\bold B}(\big[x_1\otimes \cdots \otimes x_n\big])=
& 
\Delta^l(x_{1}) \otimes \cdots \otimes x_n\\
&+
\sum_{i=1}^{n-1}\big[x_1 \otimes \cdots \otimes x_i\big]\otimes
\big[x_{i+1} \otimes \cdots \otimes x_n\big] \\
&
+x_{1} \otimes \cdots \otimes \Delta^r(x_n).
\end{align*}
The right hand side of (\ref{comultiplication for bar})
becomes a double complex using the tensor product of complexes.

\begin{proposition}
\label{coproduct as a morphism of complex}
This homomorphism $\Delta_{\bold B}$
is a homomorphism of double complexes.
\end{proposition}
\begin{proof}
Since the differential of $\bold {Bar}(A/\Cal O,\epsilon)$
is defined by the multiplications and left and right
augmentations, it commute with that of
$\bold {Bar}(A/\Cal O,\epsilon)\otimes 
\bold {Bar}(A/\Cal O,\epsilon)$ 
by the following commutative diagrams.
$$
\begin{matrix}
\end{matrix}
$$
$$
\begin{matrix}
V^{\alpha}
\otimes A^{\alpha\beta}
\otimes A^{\beta\gamma}
\otimes A^{\gamma\delta}
\otimes {}^{\delta}V &  &
\begin{matrix}
(V^{\alpha}
\otimes A^{\alpha\beta}
\otimes {}^{\beta}V) \\
\otimes (V^{\beta}
\otimes A^{\beta\gamma}
\otimes A^{\gamma\delta}
\otimes {}^{\delta}V) 
\end{matrix}
\\
\Vert & & \Vert \\
A\otimes_{\Cal O} A \otimes_{\Cal O} A &
\overset{\delta\otimes 1}\to &
A\otimes (A \otimes_{\Cal O} A) \\
1\otimes \delta\downarrow 
\phantom{1\otimes \delta}
& & \phantom{1\otimes \epsilon^r}
\downarrow 1\otimes \epsilon^r\\
(A\otimes_{\Cal O} A) \otimes A &
\overset{\epsilon^l\otimes 1}\to &
A\otimes A  \\
\Vert & & \Vert \\
\begin{matrix}
(V^{\alpha}
\otimes A^{\alpha\beta}
\otimes A^{\beta\gamma}
\otimes {}^{\gamma}V) \\
\otimes (V^{\gamma}
\otimes A^{\gamma\delta}
\otimes {}^{\delta}V) 
\end{matrix}
&  &
\begin{matrix}
(V^{\alpha}
\otimes A^{\alpha\beta}
\otimes {}^{\beta}V) \\
\otimes (
V^{\gamma}
\otimes A^{\gamma\delta}
\otimes {}^{\delta}V) 
\end{matrix}
\\
\end{matrix}
$$
This commutativity follows from the identification
(\ref{pre-augmentation componentwise}).
\end{proof}

\subsubsection{}
We define a coproduct 
$$
\Delta_{simp}^{\alpha}:Bar_{simp}^{\alpha} \to Bar_{simp} \otimes
Bar_{simp} 
$$
by the formula
\begin{align*}
&\Delta_{simp}^{\alpha}(\big[y_0
\overset{\alpha_0}\otimes x_1
\overset{\alpha_1}\otimes \cdots 
\overset{\alpha_{n-1}}\otimes x_n
\overset{\alpha_n}\otimes y_{n+1}\big]) \\
=
& 
\Delta^l(y_0\otimes x_{1}) 
\overset{\alpha_1}\otimes \cdots 
\overset{\alpha_{n-1}}\otimes x_n
\overset{\alpha_n}\otimes y_{n+1} \\
&+
\sum_{i=1}^{n-1}
\big[y_0
\overset{\alpha_0}\otimes x_1 
\overset{\alpha_1}\otimes \cdots 
\overset{\alpha_{i-1}}\otimes \Delta^r(x_i)
\big]\otimes
\big[
\Delta^l(x_{i+1}) 
\overset{\alpha_{i+1}}\otimes \cdots 
\overset{\alpha_{n-1}}\otimes x_n
\overset{\alpha_n}\otimes y_{n+1}\big] \\
&
+
y_0
\overset{\alpha_0}\otimes x_{1} 
\overset{\alpha_1}\otimes \cdots 
\overset{\alpha_{n-1}}\otimes \Delta^r(x_n\otimes y_{n+1}).
\end{align*}
Here $\Delta^l(y_0\otimes x_1)$ and $\Delta^r(x_n\otimes y_{n+1})$ are
considered as maps
\begin{align*}
&
\Cal O\otimes_{\Cal O}A=A\to \Cal O\otimes A=
(\Cal O\otimes_{\Cal O}\Cal O)\otimes
(\Cal O\otimes_{\Cal O}A)
\\
&A\otimes_{\Cal O}\Cal O=A\to A\otimes \Cal O=
(\Cal O\otimes_{\Cal O}A)
\otimes
(\Cal O\otimes_\Cal O\Cal O).
\end{align*}

\begin{proposition}
\begin{enumerate}
\item
The sum of the map 
$\Delta_{simp}=\sum_{\alpha}\Delta^{\alpha}_{simp}$
becomes a homomorphism of double complexes
$$
Bar_{simp}(A/\Cal O,\epsilon)\to
Bar_{simp}(A/\Cal O,\epsilon)\otimes
Bar_{simp}(A/\Cal O,\epsilon).
$$
\item
The coproduct structure of $Bar(A/\Cal O,\epsilon)$
and $Bar_{simp}(A/\Cal O,\epsilon)$ are compatible with the 
quasi-isomorphism $\sigma$
defined in \S \ref{comparizon for bar and simplicial bar}.
\end{enumerate}
\end{proposition}
The proof is similar to 
Proposition \ref{coproduct as a morphism of
complex} and we omit it.
\subsubsection{}
The homomorphism $\Cal O \to \bold k$ and
the sum of
$\Cal O\overset{i}\otimes_{\Cal O}\Cal O\to \bold k$
defines counit $Bar(A/\Cal O,\epsilon)\to \bold k$
and $Bar_{simp}(A/\Cal O,\epsilon)\to \bold k$.

\section{Relative DGA and DG category}
\label{section:relative DGA and DG category}
Let $S$ be a split reductive group, $\Cal O$ the coordinate ring of $S$,
$A$ a relative DGA over $\Cal O$, and $\epsilon$ a relative augmentation 
of $A$.
In this section, we introduce the DG category associated to a relative
DGA $A$ over $\Cal O$.
Moreover, we prove that DG category $K\Cal C(A/\Cal O)$
of DG complexes in $\Cal C(A/\Cal O)$ is homotopy equivalent to that of
$Bar(A/\Cal O,\epsilon)$-comodules
(Theorem \ref{main theorem genereal}).

\subsection{DG category associated to a relative DGA}
\subsubsection{}
\label{definition of DG cat associate to relative DGA}
We define a DG category $\Cal C(A/\Cal O)$ for a relative DGA $A$ over
the Hopf algebra $\Cal O$.
An object $V=V^{\bullet}$ of $\Cal C(A/\Cal O)$ is a complex
$V=V^{\bullet}$ of finite dimensional
left $\Cal O$-comodules such that $V^{\bullet}\to \Cal O\otimes V^{\bullet}$
is a homomorphism of complexes.
For objects $V_1$ and $V_2$ in 
$\Cal C(A/\Cal O)$, we define a complex of homomorphisms
$\underline{Hom}_{\Cal C(A/\Cal O)}(V_1, V_2)$ by
$$
\underline{Hom}_{\Cal C(A/\Cal O)}(V_1, V_2) =
\underline{Hom}_{\Cal O}(V_1, A\otimes_{\Cal O}V_2).
$$
Then the composite of the complexes of homomorphisms
is defined by the composite of the following maps:
\begin{align}
\label{from yoneda to composite}
&
Hom_{\Cal O}(V_2,A \otimes_{\Cal O}
V_3)
\otimes
Hom_{\Cal O}(V_1,A \otimes_{\Cal O}
V_2)
 \\
\nonumber
\to & Hom_{\Cal O}(A\otimes_{\Cal O}V_2,
A \otimes_{\Cal O}
A \otimes_{\Cal O}
V_3)
\otimes
Hom_{\Cal O}(V_1,A \otimes_{\Cal O}
V_2)
\\
\nonumber
\to &
Hom_{\Cal O}(V_1,
A \otimes_{\Cal O}
A \otimes_{\Cal O}
V_3)
\\
\nonumber
\overset{\mu_A}\to &
Hom_{\Cal O_S}(V_1,
A \otimes_{\Cal O}
V_3).
\end{align}
Then we have
$A\simeq \underline{Hom}_{\Cal C(A/\Cal O)}(L\Cal O, L\Cal O)$
and the multiplication map is equal to
\begin{align}
\label{eqn:Yoneda pair}
&A\otimes_{\Cal O}A
\to 
A\otimes A 
\xrightarrow{\sigma}
A\otimes A \\
\nonumber
\simeq &
\underline{Hom}_{\Cal C(A/\Cal O)}(L\Cal O, L\Cal O)\otimes
\underline{Hom}_{\Cal C(A/\Cal O)}(L\Cal O, L\Cal O) \\
\nonumber
\overset{\circ}\to &
\underline{Hom}_{\Cal C(A/\Cal O)}(L\Cal O, L\Cal O)
\end{align}
Here $\sigma$ denotes the transposition 
(\ref{transposition}).

The DG category of finite DG complexes in $\Cal C(A/\Cal O)$ is
denoted as $K\Cal C(A/\Cal O)$.
An object in $K\Cal C(A/\Cal O)$ is written as $(V^{i}, d_{ij})$
where 
$$
d_{ij}\in 
\underline{Hom}^{1}_{\Cal C(A/\Cal O)}
(V^je^{-j}, V^ie^{-i})\overset{(*)}\simeq
\underline{Hom}^{1+j-i}_{\Cal C(A/\Cal O)}(V^j, V^i).
$$
Here the map $(*)$ is given in 
(\ref{shift and homomorphism and sgn}).
(See also \cite{T}, \S 2.2.)
By the definition of DG complex, we have
\begin{equation}
\label{recall condition for DG complex}
\partial d_{ik}+\sum_{i<j<k}d_{ij}\circ d_{jk}=0.
\end{equation}

\subsubsection{}
Let $\Cal C_1$ and $\Cal C_2$ be DG categories.
A pair $F=(ob(F),mor(F))$
of map $ob(F):ob(\Cal C_1)\to ob(\Cal C_2)$
and $mor(F):mor(\Cal C_1)\to mor(\Cal C_2)$
is called a DG functor, if it is compatible with the composite,
and preserves identity morphisms.
\subsubsection{}
If $A_{pre}=Hom_{\bold k}(\Cal O,\Cal O)$, 
then $\Cal C(A_{pre}/\Cal O)$ is nothing but
the full sub-DG category of complex of finite dimensional
$\bold k$-vector spaces
consisting of $\Cal O$-comodules. 
For an object $M,N\in \Cal C(A_{pre}/\Cal O)$,
we have
$$
Hom_{\Cal C(A_{pre}/\Cal O)}(M,N)=
Hom_{\Cal O}(M,A_{pre}\otimes_{\Cal O}N)=
Hom_{\bold k}(M,N).
$$

Let $\epsilon:A\to Hom_{\bold k}(L\Cal O,L\Cal O)$ 
be a relative augmentation. We define a DG functor
$\rho_{\epsilon}:\Cal C(A/\Cal O)\to \Cal C(\bold k)$ by forgetting
left $\Cal O$-comodule structure for objects. 
For a morphism $\varphi \in Hom_{\Cal C(A/\Cal O)}(M,N)$,
we define $\rho_{\epsilon}(\varphi)$ by the image of
$\varphi$ under the map
$$
Hom_{\Cal C(A/\Cal O)}(M,N)\to Hom_{\Cal C(A_{pre}/\Cal O)}(M,N)=
Hom_{Vect_{\bold k}}(M,N).
$$

\subsubsection{DG category $(B-com)$}
\label{subsection:B-com}
Let $B$ be a counitary and coassociative differential graded coalgebra over $\bold k$.
The comultiplication and the counit is written as $\Delta_B$ and $u$.
A complex $M$ with a homomorphism $M\to B\otimes M$ of complexes is called a
$B$-comodule if it is coassociative and counitary.
For two $B$-comodules $M$ and $N$, we define the double complex
$\bold R Hom_{B}(M,N)$ by
$$
Hom_{KVect}^{\bullet}(M,N)\to Hom_{KVect}^{\bullet}(M, B\otimes N)\to 
Hom_{KVect}^{\bullet}(M, B\otimes B\otimes N)\to \cdots
$$
where the differential is given by
\begin{align*}
d\varphi=&
(1_B\otimes 1_B\otimes \cdots \otimes 1_B\otimes \Delta_N)\circ \varphi
-(1_B\otimes 1_B\otimes \cdots \otimes \Delta_B\otimes 1_N)\circ \varphi \\
&
+\cdots+(-1)^n(\Delta_B\otimes 1_B\otimes \cdots \otimes 1_B\otimes 1_N)\circ \varphi\\
&
+(-1)^{n+1}(1_B\otimes\varphi)\circ \Delta_M
\end{align*}
for an element 
$\varphi \in Hom_{KVect}^{\bullet}
(M, B^{\otimes n}\otimes N)$.
We introduce a composite map
\begin{equation}
\label{product structure on canonical resol}
c:\bold R Hom_{B}(M,N)\otimes
\bold R Hom_{B}(L,M)\to
\bold R Hom_{B}(L,N).
\end{equation}
For elements $\varphi$ and $\psi$
in $Hom_{KVect}^{\bullet}(M,B^{\otimes n}\otimes N)$ and
$Hom_{KVect}^{\bullet}(L,B^{\otimes m}\otimes M)$, we define
$c(\varphi\otimes \psi)$ in 
$Hom_{KVect}^{\bullet}(L,B^{\otimes (n+m)}\otimes N)$ by
$$
c(\varphi\otimes \psi)=(1_B^{\otimes m}\otimes \varphi)\circ \psi.
$$
We can check that this composite is associative.
The associate simple complex of $\bold RHom_B(M,N)$ is written as
$RHom_B(M,N)$.
\begin{definition}
We define a DG category $(B-com)$
for a differential graded coalgebra $B$ as follows.
The class of objects of $(B-com)$ consists of finite dimensional $B$-comodules,
and for $B$-comodules $M,N$, the complex of homomorphisms is
defined by $RHom_B(M,N)$. The composite of homomorphisms is defined
by the homomorphism (\ref{product structure on canonical resol}).
\end{definition}
\begin{proposition}
\label{prop: exact to cofree comodules}
Let $M$ be a cofree $B$-comodule. Then the functor
$N\mapsto \bold R Hom_{B}(N,M)$ is exact.
\end{proposition}
\begin{proof}
We consider the stupid filtration on
$\bold R Hom_{B}(N,M)$ and reduce to the 
exactness of
$Hom_{KVect_{\bold k}}(N, B\otimes\cdots \otimes B\otimes M)$.
\end{proof}

\subsubsection{}
\label{subsection:cofree resolution}
$N$ be a complex of $B$-comodules. We define a standard
cofree resolution 
$\bold F(N)$ of $N$
by
$$
\bold F(N):B\otimes N \to 
B \otimes B\otimes N \to 
B\otimes B \otimes B\otimes N \to \cdots
$$
where
\begin{align*}
d(b_n \otimes \cdots b_0\otimes n)
=&\sum_{i=0}^n(-1)^{i+1}
b_n\otimes \cdots \otimes \Delta_B(b_i)
\otimes \cdots \otimes b_0 \otimes n \\
&+b_n\otimes \cdots b_0\otimes \Delta_N(n)
\end{align*}
Then the associate simple complex $F(N)$ becomes a complex of $B$-comodules.

Let $N_1, N_2$ be $B$-comodules and $\varphi:N_1 \to N_2$
be a $B$-homomorphism, i.e.
the following diagram commutes:
$$
\begin{matrix}
N_1 &\xrightarrow{\varphi} & N_2 \\
\Delta_{N_1}\downarrow \phantom{\Delta_{N_1}}
& & 
\phantom{\Delta_{N_2}}\downarrow \Delta_{N_2}\\
B\otimes N_1 &\xrightarrow{1\otimes \varphi} & B\otimes N_2 
\end{matrix}.
$$
Note that we do not assume that $\varphi$ commutes with the
differentials.
The set of $B$-homomorphism from $N_1$ to $N_2$ 
is denoted by
$Hom_B(M,N)$. This space becomes a sub complex of
$Hom_{KVect \bold k}(M,N)$, since the differentials and
the comodule structures on $N_1$ and $N_2$ commute. 
\begin{lemma}
\begin{enumerate}
\item
\label{free B-comodule and vect sp}
Let $M$ be a $B$-comodule, and $N$ a complex of 
$\bold k$-vector spaces. Then $B\otimes N$ becomes a $B$-comodule
with the product complex structure. By attaching 
$\varphi \in Hom_B(M,B\otimes N)$ to
the element $\widetilde \varphi$ defined by
$$
(\widetilde{\varphi} :
M \xrightarrow{\varphi}
B\otimes N \xrightarrow{\epsilon \otimes 1} N)
\in Hom_{\bold k}(M, N),
$$
we have an isomorphism
$
Hom_B(M,B\otimes N)\to Hom_{\bold k}(M, N).
$
The inverse map is given by
$$
\varphi:M\xrightarrow{\Delta_M} B\otimes M
\xrightarrow{1\otimes \widetilde\varphi} B\otimes N.
$$
\item
Using the isomorphism of (\ref{free B-comodule and vect sp}),
we have a natural isomorphism of complexes
$$
RHom_{B}(M,N)\simeq Hom_B(M,F(N)).
$$
\end{enumerate}
\end{lemma}
We define a homomorphism of complexes
$$
\alpha:\bold RHom_B(M,N)\to Hom_B(\bold F(M),\bold F(N)).
$$
Let $\varphi$ be an element of 
$\bold RHom_B(M,N)^p=
Hom_\bold k(M, 
\underset{p}{\underbrace{B\otimes \cdots \otimes
B}}\otimes N)$.
Then the map
\begin{equation}
\label{B-com to cofree resol map}
\alpha^{q}(\varphi):
\underset{q+1}{\underbrace{B\otimes\cdots \otimes B}}\otimes M
\xrightarrow{1^{\otimes q}\otimes \varphi}
\underset{q+1}{\underbrace{B\otimes\cdots \otimes B}}\otimes
\underset{p}{\underbrace{B\otimes\cdots \otimes B}}\otimes N
\end{equation}
is an element of $Hom_B(F(M)^q, F(N)^{p+q})$ for $q\geq 0$.
By taking the sum of $\alpha^q(\varphi)$, we have a map
$\alpha(\varphi)\in Hom_B(F(M),F(N))$.
By Proposition \ref{prop: exact to cofree comodules}, 
we have the following lemma.
\begin{lemma}
\begin{enumerate}
\item The homomorphism $\alpha$ is a homomorphism of complexes
and a quasi-isomorphism.
\item
Let 
$$
\mu: Hom_B(F(M),F(N)) \otimes Hom_B(F(L),F(M))
\to Hom_B(F(L),F(N))
$$
be a homomorphism of complexes defined by the composite.
The map $\mu$ commute with the composite map
$c$ in (\ref{product structure on canonical resol})
via the homomorphism $\alpha$.
\end{enumerate}
\end{lemma}

\subsection{Bijection on objects}
\label{subsec:bijection on objects}
\subsubsection{Construction of a bijection $\varphi$}
In this subsection, we construct a map
$$
\varphi:ob(K\Cal C(A/\Cal O))\to 
ob(Bar_{simp}(A/\Cal O,\epsilon)-com).
$$

Let $(V^i, d_{ij})$ be an object of $K\Cal C(A/\Cal O)$.
For an index $\alpha=(\alpha_0< \cdots <\alpha_n)$, we define 
the following degree $n$ left $\Cal O$-homomorphism 
$$
\Delta^{\alpha}:V^{\alpha_0}e^{-\alpha_0} \to 
Bar_{simp}^{\alpha_0, \dots, \alpha_n}
\otimes V^{\alpha_n}e^{-\alpha_n},
$$
where
$$
Bar_{simp}^{\alpha_0, \dots, \alpha_n}=
\Cal O\overset{\alpha_0}\otimes_{\Cal O}A
\overset{\alpha_1}\otimes_{\Cal O}\cdots
\overset{\alpha_{n-1}}\otimes_{\Cal O}A
\overset{\alpha_n}\otimes_{\Cal O}\Cal O.
$$
For $n=0$, the map
$$
\Delta^{\alpha_0}:V^{\alpha_0}e^{-\alpha_0}\to 
\Cal O\otimes V^{\alpha_0}e^{-\alpha_0}=
Bar_{simp}^{\alpha_0}\otimes V^{\alpha_0}e^{-\alpha_0}
$$
is defined as the coproduct structure on $V^{\alpha_0}e^{-\alpha_0}$.
We define $\Delta^{\alpha_0, \dots, \alpha_n}$ for $n\geq 1$ by
the induction on $n$. 
For $n=1$, by the definition of
$K\Cal C(A/\Cal O)$, we have a left $\Cal O$-homomorphism
$
d_{\alpha_1\alpha_0}:V^{\alpha_0}e^{-\alpha_0}\to 
A\otimes_{\Cal O}V^{\alpha_1}e^{-\alpha_1}
$ 
of degree one.
By composing the following $\Cal O$-homomorphisms
\begin{align*}
\Delta^{\alpha_0,\alpha_1}:
V^{\alpha_0}e^{-\alpha_0}
&\xrightarrow{d_{\alpha_1\alpha_0}}
A\otimes_{\Cal O} V^{\alpha_1}e^{-\alpha_1} \\
& \to 
A\otimes  V^{\alpha_1}e^{-\alpha_1}=
(\Cal O
\overset{\alpha_0}\otimes_{\Cal O}A
\overset{\alpha_1}\otimes_{\Cal O}\Cal O) \otimes  V^{\alpha_1}e^{-\alpha_1},
\end{align*}
we have the required homomorphism.
Suppose that a left $\Cal O$-homomorphism
$$
\Delta^{\alpha_1, \dots, \alpha_n}:V^{\alpha_1}e^{\alpha_1}
\to
Bar_{simp}^{\alpha_1, \dots, \alpha_n}
\otimes V^{\alpha_n}e^{-\alpha_n}
$$
of degree $(n-1)$ is given. 
Using the inductive definition of $\Delta^{\alpha_1, \dots,\alpha_n}$,
we have the following composite map
\begin{align*}
V^{\alpha_0}e^{-\alpha_0}
\xrightarrow{d_{\alpha_1\alpha_0}}
A\otimes_{\Cal O}V^{\alpha_1}e^{-\alpha_1}
\xrightarrow{1\otimes \Delta^{\alpha_1, \dots, \alpha_n}}
 &
A\otimes_{\Cal O}
Bar_{simp}^{\alpha_1, \dots, \alpha_n}
\otimes V^{\alpha_n}e^{-\alpha_n} \\
&=
Bar_{simp}^{\alpha_0, \dots, \alpha_n}
\otimes V^{\alpha_n}e^{-\alpha_n},
\end{align*}
and we have
a required degree $n$-homomorphism. 
The map $\Delta^{\alpha}$ is written as
$$
\Delta^{\alpha}=(1\otimes d_{\alpha_{n}\alpha_{n-1}})\circ 
(1\otimes d_{\alpha_{n-1}\alpha_{n-2}})\circ \cdots \circ
d_{\alpha_1\alpha_0}.
$$
We set $V=\oplus_i V^{i}e^{-i}$. Then it is a finite dimensional
vector space. 
We define a homogeneous map 
$\Delta_V:V\to Bar_{simp}(A/\Cal O_S,\epsilon)\otimes V$ 
of degree zero  by 
$$
\Delta_V=\sum_{0\leq n}\sum_{\mid\alpha\mid =n}
(1\otimes t^n\otimes 1)\circ\Delta^{\alpha},
$$
where
$$
1\otimes t^n\otimes 1:
Bar_{simp}^{\alpha_1, \dots, \alpha_n}
\otimes V^{\alpha_n}e^{-\alpha_n}\to
Bar_{simp}^{\alpha_1, \dots, \alpha_n}e^n
\otimes V^{\alpha_n}e^{-\alpha_n}.
$$

We define a differential $\delta$ on the vector space $V$.
We set $\delta_{ii}=\delta_i\otimes 1$
on $V^ie^{-i}$, where $\delta_i$
is the differential on $V^i$, and
$\delta_{ji}:V^i\to V^j$ is the image of $d_{ji}$
under the map
\begin{align*}
&\underline{Hom}^1_{\Cal C(A/\Cal O)}(V^{i}e^{-i},A
\otimes_{\Cal O}V^je^{-j})  \\
 \to &
\underline{Hom}^1_{\Cal C(A/\Cal O)}(V^{i}e^{-i},
A_{pre}\otimes_{\Cal O}V^je^{-j}) \\
= &
\underline{Hom}^1_{KVect}(V^{i}e^{-i},V^je^{-j}) 
\end{align*}
induced by the relative augmentation.
We define $\delta_V$ on $V$ by
$\delta_V=\sum_{i\leq j}\delta_{ji}$.
\begin{lemma}
The map $\delta_V$ defines a differential on $V$.
\end{lemma}
\begin{proof}
The map $\delta_{ji}$ ($i<j$) is the composite of the following map.
$$
V^{i}e^{-i} \to A\otimes_{\Cal O}V^je^{-j} \xrightarrow{\epsilon^r\otimes 1} 
\Cal O \otimes_{\Cal O}V^je^{-j}
=V^je^{-j}.
$$
Let $i_i:V^ie^{-i} \to V$ and $p_j:V\to V^je^{-j}$
be the inclusion and the projection. It is enough to show that
$p_j\delta_V^2 i_i=0$ 
for $i\leq j$. If $i=j$, then the equality
holds, because $\delta_{ii}$ is a differential of $V^i$.
By the commutative diagram
$$
{\tiny
\begin{matrix}
V^ie^{-i} 
\xrightarrow{d_{ki}}
& A\otimes_{\Cal O}V^ke^{-k} &  
\xrightarrow{1\otimes d_{jk}} & 
A\otimes_{\Cal O}A\otimes_{\Cal O} V^je^{-j}
& \xrightarrow{\mu\otimes 1} & 
A\otimes_{\Cal O} V^je^{-j} \\
 & 
\epsilon^r\otimes 1\downarrow\phantom{e^r\otimes 1} 
& & 
\epsilon^r\otimes 1\otimes 1\downarrow\phantom{e^r\otimes 1\otimes 1} 
& & \epsilon^r\otimes 1\downarrow \phantom{e^r\otimes 1}  \\
 & V^ke^{-k} &
\xrightarrow{d_{jk}} & 
A\otimes_{\Cal O} V^je^{-j} &
\xrightarrow{\epsilon^r\otimes 1} & 
V^je^{-j},
\end{matrix}
}
$$
we have 
\begin{align*}
& (\epsilon^r\otimes 1)
(d_{jk}\circ d_{ki})=
(\epsilon^r\otimes 1)(\mu\otimes 1)
(1\otimes d_{jk})d_{ki}=
\delta_{jk}\delta_{ki}, \\
&(\epsilon^r\otimes 1)(\partial d_{ji})
=\delta_{jj}\delta_{ji}+\delta_{ji}\delta_{ii}
\end{align*}
for $i<k<j$. Therefore by the condition 
(\ref{recall condition for DG complex})
 for 
DG complex, we have the lemma.
\end{proof}
\begin{proposition}
The homomorphism $\Delta_V$ defines a 
$Bar_{simp}(A/\Cal O,\epsilon)$-comodule structure on $V$.
\end{proposition}
\begin{proof}
We can easily check the coassociativity and the counitarity.
We show that the homomorphims
$\Delta_V$ is a homomorphism of complexes, in other
words, the homomorphism $\Delta_V$ commutes with
the differential $\delta_V$ on $V$ and 
differential $d_{Bar_{simp}}\otimes 1+1\otimes \delta_V$ 
on $Bar_{simp}\otimes V$.

The outer differential defined in
(\ref{outer differential for simplical bar complex})
for $\bold {Bar}_{simp}$ is denoted 
by $\partial_{\alpha,i}$, 
and the inner differential 
$Bar^{\alpha}_{simp}\to Bar^{\alpha}_{simp}$
is denoted by $d_{in}$.
We compute 
\begin{equation}
\label{checking the commutativity with differential}
(d_{Bar_{simp}}\otimes 1+1\otimes \delta_V)\Delta_V -
\Delta_V\circ\delta_V
\end{equation}
on the component
from $V^{\alpha}e^{-\alpha}$ to 
$Bar^{\alpha_0, \dots,\alpha_n}e^{n}\otimes V^\beta e^{-\beta}$
for $\alpha\leq \alpha_0 < \cdots <\alpha_n \leq \beta$.
It is the sum of the following terms:
\begin{enumerate}
\item
\label{term 1}
$-(1\otimes t^n\otimes 1)\Delta^{\alpha}\delta_{\alpha_0\alpha}$
if $\alpha_n=\beta$ and $\alpha<\alpha_0$.
\item
\label{term 2}
$-(1\otimes t^n\otimes 1)\Delta^{\alpha}\delta_{\alpha\alpha}$
if $\alpha_n=\beta$ and $\alpha=\alpha_0$.
\item
\label{term 3}
$(1\otimes \delta_{\beta\alpha_n})(1\otimes t^n\otimes 1)\Delta^{\alpha}$
if $\alpha=\alpha_0$ and $\alpha_n<\beta$.
\item
\label{term 4}
$(1\otimes \delta_{\beta\beta})(1\otimes t^n\otimes 1)\Delta^{\alpha}$
if $\alpha=\alpha_0$ and $\alpha_n=\beta$.
\item
\label{term 5}
$(d_{in}\otimes 1)(1\otimes t^n\otimes 1)\Delta^{\alpha}$
if $\alpha=\alpha_0$ and $\alpha_n=\beta$.
\item
\label{term 6}
summation 
$(-1)^i(\partial_{\gamma,i}\otimes 1)(1\otimes t^n\otimes 1)\Delta^{\gamma}$
over the index $\gamma$
such that $\alpha$ is obtained from 
$\gamma=(\gamma_0, \dots, \gamma_{n+1})$ 
by deleting $i$-th element,
for $1\leq i \leq n$
if $\alpha=\alpha_0$ and $\alpha_n=\beta$.
\item
\label{term 7}
$(\partial_{\gamma,0}\otimes 1)(1\otimes t^n\otimes 1)\Delta^{\gamma}$
where $\gamma=(\alpha<\alpha_0<\cdots <\alpha_n)$
if $\alpha<\alpha_0$ and $\alpha_n=\beta$.
\item
\label{term 8}
$(-1)^{n+1}(\partial_{\gamma,n+1}\otimes 1)(1\otimes t^n\otimes 1)\Delta^{\gamma}$
where $\gamma=(\alpha_0<\cdots <\alpha_n<\beta)$
if $\alpha=\alpha_0$ and $\alpha_n<\beta$.
\end{enumerate}
We compute (\ref{checking the commutativity with differential}) 
for all components.

(a) The case $\alpha<\alpha_0$ and $\alpha_n=\beta$.
The terms (\ref{term 1}) and (\ref{term 7}) contribute.
They cancel by the definition of $\partial_{\gamma,0}$
and $\delta_{\alpha_0\alpha}$.

(b) The case $\alpha=\alpha_0$ and $\alpha_n<\beta$.
The terms (\ref{term 3}) and (\ref{term 8}) contribute.
They cancel by the following commutative diagrams.
$$
\begin{matrix}
A\otimes_{\Cal O}V^{\alpha_n}e^{\alpha_n}& 
\xrightarrow{1\otimes d_{\beta\alpha_n}} &
A\otimes_{\Cal O}A\otimes_{\Cal O}V^{\beta}e^{-\beta}
&\xrightarrow{} &
(A\otimes_{\Cal O}A)\otimes V^{\beta} e^{-\beta} \\
\downarrow & & \downarrow & & 
(1\otimes \epsilon^l)\otimes 1 
\downarrow 
\phantom{(1\otimes \epsilon^l)\otimes 1 } \\
A\otimes V^{\alpha_n}e^{-\alpha_n}& 
\xrightarrow{1\otimes d_{\beta\alpha_n}} &
A\otimes (A\otimes_{\Cal O}V^{\beta}e^{-\beta})
&\xrightarrow{1\otimes (\epsilon^r\otimes 1)} &
A\otimes V^{\beta}e^{-\beta}.
\end{matrix}
$$

(c) The case $\alpha=\alpha_0$ and $\alpha_n=\beta$.
The summation of the term in (\ref{term 6}) 
for a fixed $i$ is equal to the summation $(-1)^i$ times the
composites of
\begin{align*}
V^{\alpha_0}e^{-\alpha_0}&\to 
Bar_{simp}^{\alpha_0,\dots,\alpha_{i-1}}
\otimes_{\Cal O}V^{\alpha_{i-1}}e^{-\alpha_{i-1}} \\
&\xrightarrow[(\#)]{1_B\otimes 
(d_{\alpha_i\gamma_i}
\circ d_{\gamma_i\alpha_{i-1}})}
Bar_{simp}^{\alpha_0,\dots,\alpha_{i-1}}
\otimes_{\Cal O}A
\otimes_{\Cal O}V^{\alpha_{i}}e^{-\alpha_{i}} \\
&
\to
Bar_{simp}^{\alpha_0,\dots,\alpha_{n}}
\otimes
V^{\alpha_{n}} e^{-\alpha_{n}} \\
&
\xrightarrow{1\otimes t^n\otimes 1}
Bar_{simp}^{\alpha_0,\dots,\alpha_{n}}e^n
\otimes
V^{\alpha_{n}} e^{-\alpha_{n}}
\end{align*}
over $\alpha_{i-1}<\gamma_i< \alpha_i$.
By the relation (\ref{recall condition for DG complex}), 
the morphism (\#) 
is $(-1)$-times the summation of
\begin{enumerate}
\item[(c-1-i)]
\begin{align*}
Bar_i
\otimes_{\Cal O}V^{\alpha_{i-1}}e^{-\alpha_{i-1}} 
&\xrightarrow{1\otimes \delta_{\alpha_{i-1}\alpha_{i-1}}}
Bar_i\otimes_{\Cal O}V^{\alpha_{i-1}}e^{-\alpha_{i-1}} \\
&\xrightarrow{1\otimes d_{\alpha_i\alpha_{i-1}}}
Bar_i\otimes_{\Cal O}A\otimes_{\Cal O}V^{\alpha_{i}}e^{-\alpha_{i}}
\end{align*}
\item[(c-2-i)]
\begin{align*}
Bar_i
\otimes_{\Cal O}V^{\alpha_{i-1}} e^{-\alpha_{i-1}}
&\xrightarrow{1\otimes d_{\alpha_i\alpha_{i-1}}} 
Bar_i\otimes_{\Cal O}A\otimes_{\Cal O}
V^{\alpha_{i}}e^{-\alpha_{i-1}} \\
&\xrightarrow{1\otimes \delta_{\alpha_{i}\alpha_{i}}}
Bar_i\otimes_{\Cal O}A\otimes_{\Cal O}V^{\alpha_{i}}e^{-\alpha_{i}}
\end{align*}
\item[(c-3)]
\begin{align*}
Bar_i
\otimes_{\Cal O}V^{\alpha_{i-1}} e^{-\alpha_{i-1}}
&\xrightarrow{1\otimes d_{\alpha_i\alpha_{i-1}}}
Bar_i\otimes_{\Cal O}A\otimes_{\Cal O}
V^{\alpha_{i}} e^{-\alpha_{i-1}}\\
&\xrightarrow{d_A\otimes 1}
Bar_i\otimes_{\Cal O}A\otimes_{\Cal O}V^{\alpha_{i}}e^{-\alpha_{i}}
\end{align*}
\end{enumerate}
Thus 
\begin{enumerate}
\item[(A)]
the summation for (c-3) cancels with
the term (\ref{term 5}), 
\item[(B)]
the term (c-1-(i+1)) and (c-2-i) cancel for
$i=1,\dots,n-1$,
\item[(C)]
the term (c-1-1) and (\ref{term 2}) cancel, and
\item[(D)]
the term (c-2-n) and (\ref{term 4}) cancel.
\end{enumerate}
Therefore (\ref{checking the commutativity with differential})
is equal to zero and $\Delta_V$ is a homomorphism of complexes.
\end{proof}
\begin{definition}
By associating $(V,\Delta_V)$ to 
$(V^i,d_{ij})$, we have a map
$\varphi:ob(K\Cal C(A/\Cal O_S))\to ob(Bar_{simp}-com)$.
\end{definition}

\subsubsection{ Construction of the functor $\psi$}
We construct the inverse 
$\psi: ob(Bar_{simp}(A/\Cal O_S)-com)\to ob(K\Cal C(A/\Cal O_S))$
of the map $\varphi$.
Let $V$ be a bounded  complex of finite dimensional vector spaces and
$$
\Delta_V:V\to Bar_{simp}\otimes V
$$
be a $Bar_{simp}$-comodule structure on $V$.
Let 
\begin{align*}
&\pi_i:Bar_{simp}\to 
\Cal O{\overset{i}\otimes}_{\Cal O}\Cal O,\quad \\
&\pi_{ij}:Bar_{simp}\to 
\Cal O{\overset{i}\otimes}_{\Cal O}
A{\overset{j}\otimes}_{\Cal O}
\Cal Oe
\xrightarrow{1\otimes t^{-1}}
\Cal O{\overset{i}\otimes}_{\Cal O}
A{\overset{j}\otimes}_{\Cal O}
\Cal O
\end{align*}
be the projection to the 
$Bar^{(i)}_{simp}$ and
$Bar^{(ij)}_{simp}$
components. 
Let $\Delta_i$ be the composite
$$
V \xrightarrow{\Delta_V}
Bar_{simp}\otimes V
\xrightarrow{1\otimes \pi_i}
(\Cal O_S\overset{i}\otimes_{\Cal O_S}\Cal O_S)\otimes V
=\Cal O_S \otimes V 
$$ 
and
$pr_i=(\epsilon^r\otimes 1_V)\Delta_i:
V\longrightarrow V
$.
Let $d_{ji}$ be the composite
$$
V \to Bar_{simp}\otimes V \xrightarrow{\pi_{ij}\otimes 1}
(\Cal O_S\overset{i}\otimes_{\Cal O_S}
A\overset{j}\otimes_{\Cal O_S}\Cal O_S)\otimes V
= A\otimes V.
$$
\begin{proposition}
\begin{enumerate}
\item
The maps $pr_i$ ($i \in \bold Z$) define complete orthogonal idempotents of $V$.
\item
We have
\begin{align*}
\Delta_ipr_i=\Delta_i=(1\otimes pr_i)\Delta_i,\quad
d_{ji}pr_i=d_{ji}=(1\otimes pr_j)d_{ji}
\end{align*}
As a consequence $V^i$ becomes an $\Cal O$-comodule.
\item
The differential on $V$ commute with the projection $pr_i$
and the coaction of $\Cal O$. As a consequence, 
$V^i$ becomes an object in $\Cal C(A/\Cal O)$.
\item
The homomorphism
$d_{ji}$ defines an element in 
$Hom_{\Cal O}(V^i, A\otimes_{\Cal O} V^j)=
Hom_{\Cal C(A/\Cal O)}(V^i,V^j)$.
\end{enumerate}
\end{proposition}
\begin{proof}
(1) By the counitarity of the coaction of $Bar_{simp}$
on $V$, the composite map 
$$
V\xrightarrow{\Delta_V} Bar_{simp}\otimes V
\xrightarrow{\epsilon \otimes 1} V
$$
is the identity map on $V$. Therefore we have
$\sum_ipr_i=1_V$.
The statement (2),(3) follows form the coassociativity.
\end{proof}
By the compatibility of $\Delta_V$ for the differentials of
$V$ and $Bar_{simp}\otimes V$, we have the following proposition.
\begin{proposition}
\begin{enumerate}
\item
The map $d_{ij}$ defines
a DG complex in $\Cal C(A/\Cal O)$
\item
The maps $\varphi$ and $\psi$ are inverse to each other.
\end{enumerate}
\end{proposition}

We define the functor $\psi$ by associating the $Bar$-comodule $(V,\Delta_V)$
to the system $\{V^i,d_{ij}\}$ in $K\Cal C(A/\Cal O_S)$.

\begin{remark}
Without the finite dimensionality of the $Bar_{simp}$-comodule $V$,
we can define the above decomposition $V=\oplus_i V^i$ and
the map $\Delta_i$ by
the counitarity and the coassociativity.
\end{remark}

\subsection{Homotopy equivalence of $(K\Cal C(A/\Cal O))$
and $(Bar_{simp}-com)$}
\label{subsection: homotopy equivalence of DG-cat and Bar}

\subsubsection{}
We set $B=Bar_{simp}(A/\Cal O,\epsilon)$. 
Let $N_1$ and $N_2$ be $B$-comodules.
We set 
$\psi(N_1)=(N_1^{i},d_{ji})$,
$\psi(N_2)=(N_2^{i},d_{ji})$ and let
$i_i:N_1^i \to N_1$ and $p_j:N_2 \to N_2^j$
be the natural inclusion and the projection.
Via the bijections $\varphi$ and $\psi$, we identify the class of
objects in $K\Cal C(A/\Cal O)$ and that of
$(B-com)$ defined in \S \ref{subsec:bijection on objects}.
%

Let $f$ be an element of $Hom_B(N_1,N_2)$.
Since
the coaction of $B$ is compatible with the map $f$,
we have $f(N_1^i)\subset N_2^i$ and
the restriction $f^i$ of $f$ to $N_i^i$
is a $\Cal O$-homomorphism.
Let $\beta(f)^i\in 
Hom_{\Cal O}(N_1^i,A\otimes_{\Cal O}N_2^i)$ be the following composite
homomorphism
$$
N_1^i \xrightarrow{f^i} N_2^i=\Cal O\otimes_{\Cal O} N_2^i\to
A\otimes_{\Cal O}N_2^i.
$$
Then we have 
$\beta(f)\in Hom_{K\Cal C(A/\Cal O)}
(\psi(N_1),\psi(N_2))$.
\begin{definition}
The category of DG complexes in $\Cal C$ without assuming the finiteness of 
$\{i \mid V^i\neq 0\}$ for $\{V^i,d_{ij}\}$ is denoted as
$K\Cal C'$. In this category, the morphism $\{\varphi_{ij}\}$ is assumed to be bounded from
 below, i.e. there exists $m$ such that $\varphi_{i,j}=0$ for $j<i+m$,
 so that the composite of two morphisms are well defined. 
\end{definition}
\begin{lemma}
\begin{enumerate}
\item
The homomorphism
$$
\beta:Hom_B(N_1,N_2)\to Hom_{K\Cal C(A/\Cal O)}(\psi(N_1), \psi(N_2))
$$
commutes with the differentials and is compatible with the
composites.
\item
Moreover if $N_2$ is a cofree $B$-comodule,
the map $\beta$ is a quasi-isomorphism.
\end{enumerate}
\end{lemma}

\subsubsection{}
Let $M$ be an object of $K\Cal C(A/\Cal O)$.
Then the free resolution
$F(\varphi(M))$ defined in \S \ref{subsection:cofree resolution}
of $\varphi(M)$ is a cofree $B$-comodule.
The object $\psi(F(\varphi(M)))\in K\Cal C(A/\Cal O)$ corresponding to
$F(\varphi(M))$ is denoted by $F(M)$.
By the above lemma, we have the following corollary.

\begin{lemma}
Let $N_1,N_2$ be $B$-comodules.
\begin{enumerate}
\item
Then the homomorphism
\begin{equation}
\label{B-com to KC map of morphism}
\beta:Hom_B(F(N_1),F(N_2))\to
Hom_{K\Cal C'(A/\Cal O)}(\psi(F(N_1)),\psi(F(N_2)))
\end{equation}
is a quasi-isomorphism and compatible with the composites.
\item
The natural homomorphisms of complexes
\begin{align}
\label{map KC to free KC}
&Hom_{K\Cal C'(A/\Cal O)}(N_1, N_2)\to
Hom_{K\Cal C'(A/\Cal O)}(N_1, F(N_2)) \\
\label{map free free KC to id-free KC}
&Hom_{K\Cal C'(A/\Cal O)}(F(N_1), F(N_2))\to
Hom_{K\Cal C'(A/\Cal O)}(N_1, F(N_2)) 
\end{align}
induced by the natural homomorphism $N_1 \to F(N_1)$
and $N_2\to F(N_2)$
are quasi-isomorphisms.
\end{enumerate}
\end{lemma}
\subsubsection{}

We define a DG category $BK(A/\Cal O)$.
The class of objects of $BK(A/\Cal O)$ is that of $B-com$.
Let $M,N$ be objects in $BK(A/\Cal O)$.
The complex of homomorphism 
$Hom_{BK(A/\Cal O)}(M,N)$ is defined by the cone of
$$
\begin{matrix}
& & Hom_{K\Cal C(A/\Cal O)}(\psi(M), \psi(N)) \\
& &\downarrow \eta\ 
(\ref{map KC to free KC})
\\
RHom_B(M,N)& \xrightarrow{\xi} &
Hom_{K\Cal C'(A/\Cal O)}(\psi(M),F(\psi(N)).
\end{matrix}
$$
Here the map $\xi$ is the composite
\begin{align*}
RHom_B(M,N)
&\xrightarrow{\alpha 
(\ref{B-com to cofree resol map})
} Hom_B(F(M),F(N)) \\
&\xrightarrow{\beta 
(\ref{B-com to KC map of morphism})
}
Hom_{K\Cal C'(A/\Cal O)}(\psi(F(M)),\psi(F(N))) \\
&
\xrightarrow{
(\ref{map free free KC to id-free KC})
}
Hom_{K\Cal C'(A/\Cal O)}(\psi(M),F(\psi(N)).
\end{align*}
Then the homomorphisms $\xi$ and $\eta$ are quasi-isomorphism.
We introduce a composite structure
$$
\circ:Hom_{BK(A/\Cal O)}(M,N)\otimes Hom_{BK(A/\Cal O)}
(L,M) \to Hom_{BK(A/\Cal O)}(L,M)
$$
by the rule
$$
(a+b+c)\circ(a'+b'+c')=(a\circ a')+(b\circ b')+
(c\circ\eta(a')+\xi(b)\circ c')
$$
for 
\begin{align*}
& a\in Hom_{K\Cal C(A/\Cal O)}(\psi(M), \psi(N)), \quad
 b\in RHom_B(M,N) \\
&c \in Hom_{K\Cal C'(A/\Cal O)}(\psi(M),F(\psi(N)).
\end{align*}
Then this composite is associative. Thus we have the following theorem.
\begin{theorem}
\label{main theorem genereal}
\begin{enumerate}
\item
The natural projections 
\begin{align*}
&Hom_{BK(A/\Cal O)}(M,N)\to 
Hom_{K\Cal C(A/\Cal O)}(\psi(M),\psi(N)) \\
&Hom_{BK(A/\Cal O)}(M,N)\to
RHom_B(M,N)
\end{align*}
defines a DG functor.
\item
The DG functors $BK(A/\Cal O)\to K\Cal C(A/\Cal O)$
and $BK(A/\Cal O)\to (B_{simp}-com)$ are homotopy equivalent.
\end{enumerate}
\end{theorem}

\subsection{Integrable $(A/\Cal O)$-connection}
\label{subsection:integrable connection}
In this and the next subsection, we
show that the category $(Rep_G)^S$ and
that of $H^0(Bar(A/\Cal O),\epsilon)$-comodules
are equivalent when
$A$ is the relative DGA introduced in
\S \ref{example of relative DGA associated to group}.
The contents of \S \ref{subsection:integrable connection}
and \S \ref{subsection:relative completion associated to represenataion}
are not used for the definition of the category
of mixed elliptic motives.

We define the category of integrable 
$(A/\Cal O)$-connections.
\begin{definition}
\begin{enumerate}
\item
Let $\bold V=(V^ie^i,\{d_{ij}\})$ be an object of $K\Cal C(A/\Cal O_S)$.
The object $\bold V$ is said to be concentrated at degree zero if and only
if $V^i$ is an $\Cal O_S$-comodule put at degree zero.
The full subcategory of $K\Cal C(A/\Cal O_S)$ consisting of
objects concentrated at degree zero is denoted by $K\Cal C(A/\Cal O_S)^0$.
\item
We define the homotopy category $H^0(K\Cal C(A/\Cal O_S)^0)$ of 
$K\Cal C(A/\Cal O_S)^0$ whose class of objects is that in
$K\Cal C(A/\Cal O_S)^0$.
The set of morphisms from $\bold M$ to $\bold N$ in $H^0(K\Cal C(A/\Cal O_S)^0)$
is defined by
$H^0(\underline{Hom}_{K\Cal C(A/\Cal O)}(\bold M,\bold N))$.
The category $IC(A/\Cal O)=H^0(K\Cal C(A/\Cal O_S)^0)$ is
called the category of integrable $(A/\Cal O)$-connections.
\end{enumerate}
\end{definition}

\begin{proposition}
An object in $H^0(K\Cal C(A/\Cal O)^0)$ is equivalent to
the following data:
\begin{enumerate}
\item
A finite set of comodules $V^i$.
The comodule structure on $V^i$
 is denoted by $\Delta_i:V^i \to \Cal O_S\otimes_k V^i$.
\item
$\Cal O$-homomorphims $\nabla_{ji}:V^i \to A^1\otimes_k V^j$
for each $i< j$.
\end{enumerate}
We impose the following conditions for data.
\begin{enumerate}
\item
The composite 
$$
V^i\to A^1\otimes_k V^j \xrightarrow{\Delta_A^r\otimes 1-1\otimes \Delta_j}
 A^1\otimes_k\Cal O_S
\otimes_k V^j
$$ 
is zero.
\item
Let 
$$
\nabla_{kji}^2=(1_A\otimes \nabla_{kj})\circ\nabla_{ji}
:V^i \to A^1 \otimes_{k} A^1 \otimes_k V^k
$$
be the composite map which defines an element in
$Hom_{\Cal O_S}(V^i,A^1\otimes_{\Cal O_S}A^1
\otimes_{\Cal O_S}V^k)$. 
The image of $\nabla^2_{kji}$ 
under the map
$$
Hom_{\Cal O_S}(V^i,A^1\otimes_{\Cal O_S}A^1
\otimes_{\Cal O_S}V^k)\to
Hom_{\Cal O_S}(V^i,A^2\otimes_{\Cal O_S}V^j)
$$
can be written as
${\nabla}_{kj}\circ \nabla_{ji}$ using the composite in $\Cal C(A/\Cal O)$.
Then the equality
\begin{equation}
\label{condition for connection in nabla}
\partial \nabla_{ki}+\sum_{i<j<k}{\nabla}_{kj}\circ\nabla_{ji}=0
\end{equation}
holds.
\end{enumerate}
\end{proposition}
\subsubsection{}

Let $V$ and $W$ be objects in $K\Cal C(A/\Cal O)^0$.
A closed homomorphism $\varphi$ in $Z^0\underline{Hom}_{K\Cal C(A/\Cal O)}(V,W)$
of degree zero
is a set of homomorphisms $\{\varphi_{ji}\}$ with $\varphi_{ji}\in 
Hom_{\Cal O_S}(V^i, A^0\otimes_{\Cal O_S} W^j)$ such that
\begin{equation}
\label{closed element in z0 for connection}
\partial\varphi_{ki}+\sum_{k>j}\nabla_{kj}^W\circ\varphi_{ji}
-\sum_{j>i}\varphi_{kj}\circ\nabla_{ji}^V=0,
\end{equation}
using differentials and composites in $\Cal C(A/\Cal O)$.
\subsubsection{}
Let $A\to A'$ be a quasi-isomorphism between relative DGA's over $\Cal
O_S$.
An $A$-connection $V$ relative to $S$ can be regarded as an $A'$-connection.
\begin{proposition}
\label{prop:invariance for quasi-isom}
\begin{enumerate}
\item 
Let $V$ be an $A$-connection relative to $S$.
Then the set of homomorphisms
$Hom_{IC(A/\Cal O)}(V, W)$ and $Hom_{IC(A'/\Cal O)}(V, W)$ are naturally isomorphic.
\item
For any $A'$-connection $W$, there exists an $A$-connection $V$
which is isomorphic to $W$.
As a consequence, the categories $IC(A/\Cal O)$ and $IC(A'/\Cal O)$
of $A$-connections and
$A'$-connections
are equivalent.
\end{enumerate}
\end{proposition}
\begin{proof}
(1) 
We can introduce a natural complex structure
on $Hom_{\Cal O_S}(V, A^{\bullet}\otimes_{\Cal O_S}W)$ and
can show that $Hom_{\Cal O_S}(V, A^{\bullet}\otimes_{\Cal O_S}W)$
and $Hom_{\Cal O_S}(V, {A'}^{\bullet}\otimes_{\Cal O_S}W)$ are quasi-isomorphic
by taking a suitable filtrations on $V$ and $W$
and considering the associate graded objects.

(2) We set $W=\oplus_{i=0}^n W^i$ and we prove 
that there exists an $A$-connection $V$ such that $V^i=W^i$
and a closed homomorphism $\varphi:V\to {A'}^0\otimes W$ such that 
$$
\text{
$\varphi_{ji}=0$ for
 $j<i$
and $\varphi_{ii}=id$.
}\quad (P)
$$
We set $F^1W=\oplus_{i\geq 1}W^i, F^1V=\oplus_{i\geq 1}V^i$ and assume that there exists
a closed isomorphism $\varphi:F^1V\to {A'}^0\otimes F^1W$ by induction.
We extend this isomorphism to $V\to {A'}^0\otimes W$ with the same properties (P).
Since $\varphi_{00}=id$ and $\varphi_{0i}=0$ for $i>0$, 
it is enough to define 
$\varphi_{j0}\in Hom_{\Cal O_S}(V^0,A^0\otimes_{\Cal O_S}W^j)$ for $j>0$
and $\nabla_{j0}^V\in Hom_{\Cal O_S}(V^0,A^1\otimes_{\Cal O_S}V^j)$ for $j>0$
such that
\begin{align}
\label{cond:connection hom}
&\partial\varphi_{j0}
- \varphi_{jj}\circ\nabla_{j0}^V \\
\nonumber
=&
- \nabla_{j0}^W\circ \varphi_{00}
-\sum_{0< k<j}\nabla^W_{jk}\circ\varphi_{k0}
+\sum_{0<l<j}\varphi_{jl}\circ \nabla_{l0}^V
\end{align}
in $Hom_{\Cal O_S}(V^0,{A'}^1\otimes_{\Cal O_S}W^j)$.
We define $\varphi_{j0}$ by the induction on $j$. 
By the assumption of induction,
the right hand side of (\ref{cond:connection hom}) is defined.
Then we have
\begin{align*}
&\partial(-\sum_{0\leq k<j}\nabla^W_{jk}\circ\varphi_{k0}
+\sum_{0<l<j}\varphi_{jl}\circ\nabla_{l0}^V) \\
=&
\sum_{0\leq k<l<j}\nabla^W_{jl}\circ\nabla^W_{lk}\circ\varphi_{k0}
+\sum_{0<l<j}\nabla^W_{jl}\circ (\partial\varphi_{l0})\\
&+\sum_{0<k<j}(\partial \varphi_{jk})\circ\nabla_{k0}^V 
-\sum_{0<k<l<j}\varphi_{jl}\circ\nabla_{lk}^V\circ\nabla_{k0}^V \\
=&
\sum_{0\leq k<l<j}\nabla^W_{jl}\circ\nabla^W_{lk}\circ\varphi_{k0}
-\sum_{0\leq p<l<j}\nabla^W_{jl}\circ \nabla^W_{lp}\circ \varphi_{p0}
+\sum_{0< p\leq l<j}\nabla^W_{jl}\circ \varphi_{lp}\circ \nabla^V_{p0}
\\
&
+\sum_{0<k< p\leq j}\varphi_{jp}\circ \nabla_{pk}^V\circ\nabla_{k0}^V 
-\sum_{0<k\leq p<j}\nabla_{jp}^W\circ \varphi_{pk}\circ\nabla_{k0}^V 
-\sum_{0<k<l<j}\varphi_{jl}\circ\nabla_{lk}^V\circ\nabla_{k0}^V \\
=&\sum_{0<k<j}\varphi_{jj}\circ\nabla_{jk}^V\circ\nabla_{k0}^V.
\end{align*}
Since
$Hom_{\Cal O_S}(V^0,{A'}^{\bullet}\otimes_{\Cal O_S}W^j)$ and
$Hom_{\Cal O_S}(V^0,{A}^{\bullet}\otimes_{\Cal O_S}W^j)$ are
 quasi-isomorphic, we can choose
$\nabla_{j0}'\in Hom_{\Cal O_S}^1(V^0,{A}^{\bullet}\otimes_{\Cal O_S}W^j)$ such that 
$$
\partial\nabla_{j0}'+\sum_{0<k<j}\nabla_{jk}^V\circ\nabla_{k0}^V=0.
$$
Therefore
\begin{equation}
\label{approximate by cohomology}
-\sum_{0\leq k<j}\nabla^W_{jk}\circ\varphi_{k0}
+\sum_{0<l<j}\varphi_{jl}\circ\nabla_{l0}^V+\varphi_{jj}\circ\nabla_{j0}'
\in Z^1Hom_{\Cal O_S}^1(V^0,{A'}^{\bullet}\otimes_{\Cal O_S}W^j)
\end{equation}
and there exist elements 
$$
\nabla_{j0}''\in Z^1Hom_{\Cal O_S}^1(V^0,{A}^{\bullet}\otimes_{\Cal
 O_S}W^j)\text{ and }
\varphi_{j0}\in Hom_{\Cal O_S}^0(V^0,{A'}^{\bullet}\otimes_{\Cal
 O_S}W^j)
$$ 
such that the left hand side of (\ref{approximate by cohomology})
is equal to $-\nabla_{j0}''+d\varphi_{j0}$.
Therefore the equation (\ref{cond:connection hom}) holds for 
$\nabla_{j0}^V=\nabla_{j0}'+\nabla_{j0}''$.
\end{proof}

\begin{definition}[Rigid relative DGA]
Let $A$ be a relative DGA over the coalgebra $\Cal O$,
$i:\Cal O\to A$ be the homomorphism of Definition \ref{definition of relative DGA}
(\ref{O-structure of A}), and $\epsilon:A \to 
Hom_{\bold k}(L\Cal O,L\Cal O)$ be a relative augmentation.
The relative DGA $A$ is said to be rigid if
\begin{enumerate}
\item
$A^i=0$ for $i<0$,
\item
$i:\Cal O\to A^0$ is an isomorphism, and
\item
the differential $A^0\to A^1$
is a zero map.
\end{enumerate}
If $A$ is rigid, then the map $i$ induces an isomorphism
$\Cal O\simeq H^0(A^0)$.
\end{definition}
Let $A$ be a rigid relative DGA.
We define the augmentation ideal $I_{\epsilon}$ by
the kernel of the relative augmentation $\epsilon$.
The
reduced bar complex
$\bold {Bar}_{red}(A/\Cal O,\epsilon)$ is defined by the 
following sub-complex of $\bold {Bar}(A/\Cal O,\epsilon)$.
$$
\cdots \to I_{\epsilon}\otimes_{\Cal O}I_{\epsilon}\otimes_{\Cal O}I_{\epsilon} \to
 I_{\epsilon}\otimes_{\Cal O}I_{\epsilon} \to
 I_{\epsilon}\to \Cal O \to 0.
$$
The associate simple complex of $\bold {Bar}(A/\Cal O,\epsilon)$
is denoted by $Bar(A/\Cal O,\epsilon)$.
The proof of the following lemma is similar to \cite{T}, Theorem 5.2, and we omit it.
\begin{lemma}
\begin{enumerate}
\item
The inclusion 
$\bold {Bar}_{red}(A/\Cal O,\epsilon)\to
\bold {Bar}(A/\Cal O,\epsilon)
$
is a quasi-isomorphism.
\item
$Bar_{red}(A/\Cal O,\epsilon)^{n}=0$ for $n<0$.
As a consequence, we have the following inclusion 
\begin{equation}
\label{inclusion for rigid relative DGA}
H^0(Bar(A/\Cal O,\epsilon))
\to
Bar_{red}(A/\Cal O,\epsilon)^{0} 
\end{equation}
\end{enumerate}
\end{lemma}
Using the above lemma, we have the following proposition.
\begin{proposition}
\label{rigid and equivalence}
Let $A$ be a DGA relative to $\Cal O$, and $\epsilon$ be a relative augmentation
of $A$. Assume that there is a sub-DGA $A_{rig}\to A$ of $A$ which is
 rigid and quasi-isomorphic to $A$.
Then the category $IC(A/\Cal O)$ is isomorphic to the category of
 $H^0(Bar(A/\Cal O,\epsilon))$-comodules.
\end{proposition}
\begin{proof}
By Proposition \ref{prop:invariance for quasi-isom}, 
we may assume that $A$ is a rigid relative DGA over $\Cal O$.
Let $M$ be an object of $IC(A/\Cal O)$. Then we have a 
$Bar(A/\Cal O,\epsilon)$-comodule $\Cal M$
corresponding to $M$.
By taking $0$-th cohomology of 
$\Cal M\to Bar(A/\Cal O,\epsilon)\otimes \Cal M$, 
we have a $H^0(Bar(A/\Cal O,\epsilon))$-comodule 
$\widetilde{\Cal M}=H^0(\Cal M)$, since
$H^i(\Cal M)=0$ for $i\neq 0$.

Conversely, let $\widetilde{\Cal M}$ be a
$H^0(Bar(A/\Cal O,\epsilon))$-comodule. Using the inclusion
(\ref{inclusion for rigid relative DGA}), we have the following map.
$$
\widetilde{\Cal M}\to 
H^0(Bar(A/\Cal O,\epsilon))\otimes \widetilde{\Cal M}\to
Bar_{red}(A/\Cal O,\epsilon)^0 \otimes \widetilde{\Cal M}.
$$
By this map, we have an object $M$ in $IC(A/\Cal O)$.
By this bijections, we have a category equivalence of 
$IC(A/\Cal O)$ and the category of $H^0(Bar(A/\Cal O,\epsilon))$-comodules.
\end{proof}

\subsection{The case associated to $G\to S(\bold k)$}
\label{subsection:relative completion associated to represenataion}
Let $G$ be a group and $G \to S(k)$ be a Zariski dense
homomorphism. Let $A=Hom_G(L\Cal O,L\Cal O)$ be the relative DGA with
respect to $\Cal O=\Cal O_S$
introduced in \S \ref{example of relative DGA associated to group}.

\subsubsection{}
We consider the category of integrable $(A/\Cal O)$-connections.
For two $G$-modules $V_1,V_2$, the extension group of $Ext_G(V_1,V_2)$
is equal to the cohomology of the complex
\begin{align*}
\underline{Hom}_G(V_1,V_2)&=
Hom_{\Cal O}(V_1,A \otimes_{\Cal O}
V_2).
\end{align*}

\label{first statement for group}
An object in $H^0(K\Cal C(A/\Cal
 O)^0)$
is equivalent to the following data.
\begin{enumerate}
\item
Algebraic representations
$(V^1,\rho_1), (V^2,\rho_2),\dots, (V^n,\rho_n)$ 
of $S$,
\item
$\nabla_{i,j}$ an element of $\underline{Hom}^1_G(V_i, V_j)$
for $i<j$. 
\end{enumerate}
The element in $Hom_k(V_i,V_j)$ determined by 
$$
v_i \mapsto \nabla_{i,j}(g\otimes v_i)
$$
is denoted as $\nabla_{i,j}(g)$.
Since 
\begin{align*}
&(\partial \nabla_{i,k})(g_1\otimes g_2\otimes v)
=\rho_k(g_1)\nabla_{i,k}(g_2\otimes v)
-\nabla_{i,k}(g_1g_2\otimes
 v)+\nabla(g_1\otimes \rho_i(g_2)v),\\
&\nabla_{j,k}\circ\nabla_{i,j}(g_1\otimes g_2\otimes v)=
\nabla_{j,k}(g_1\otimes \nabla_{i,j}(g_2\otimes v)),
\end{align*}
by the condition 
(\ref{condition for connection in nabla})
is equivalent to
the relation
\begin{align}
\label{cocycle for group case}
&\nabla_{i,k}(g_1g_2) \\
\nonumber
=&
\nabla_{i,k}(g_1)\rho_{i}(g_2)
+\rho_k(g_1)\nabla_{i,k}(g_2)+\sum_{i<j<k}
\nabla_{j,k}(g_1)\nabla_{i,j}(g_2)
\end{align}
for all $g_1,g_2\in G$.
Therefore the map 
$$
\rho:G\to Aut(V):g\mapsto \sum_{i}\rho_i(g)+\sum_{i<j}\nabla_{i,j}(g)
$$
is a homomorphism of groups if and only if 
the condition (\ref{cocycle for group case}) holds.

For an object $V=(V^i,\nabla_{ij})$ in $H^0(K\Cal C(A/\Cal O)^0)$, 
the $G$-module $(\rho, \sum_iV^i)$ obtained from $V$
as in the last proposition is denoted as $\rho(V)$.
\begin{proposition}
There exists a rigid relative sub-DGA $A_{rig}$ of $A$,
which is quasi-isomorphic to $A$.
As a consequence, $IC(A/\Cal O)$ is equivalent to
the category of $H^0(Bar(A/\Cal O,\epsilon))$-comodules.
\end{proposition}
\begin{proof}
Since the image $G\to S(\bold k)$ is Zariski dense, we have
\begin{align*}
H^0(A^{\alpha,\beta}) 
=&Hom_{\Cal O}(V^{\alpha},V^{\beta})
\\
=&\begin{cases}
\bold k & \text{ if }\alpha=\beta \\
0 & \text{ if }\alpha\neq \beta. \\
\end{cases}
\end{align*}
Let $C^{\alpha,\beta}$ be a complement of the image of the differential
$A^{\alpha,\beta,0}\to A^{\alpha,\beta,1}$.
We set
$$
A_{rig}^i=
\begin{cases}
0 & \text{ for }i<0 \\
\oplus_{\alpha}V^{\alpha}\otimes {}^{\alpha}V & \text{ for }i=0 \\
\oplus_{\alpha}V^{\alpha}\otimes C^{\alpha,\beta}\otimes {}^{\beta}V &
 \text{ for }i=1 \\
A^i &
 \text{ for }i\geq 2. \\
\end{cases}
$$
Then $A_{rig}$ satisfies the required properties.
The latter part is a consequence of 
Proposition \ref{rigid and equivalence}.
\end{proof}

\subsubsection{}
Then for two connections, $V$ and $W$, the set 
$Hom_{IC(A/\Cal O)}(V,W)$ 
is a subset of $Hom_{\bold k}(V, W)$
and can be identified with the set of $G$-homomorphisms from $\rho(V)$
to $\rho(W)$ by the closedness condition
(\ref{closed element in z0 for connection}).
Thus we get a fully faithful functor from $IC(A/\Cal O)$ to 
$(Rep_G)^S$.
We have the following proposition.
\begin{proposition}
\label{sample group case}
\begin{enumerate}
\item
The essential image of $\rho:IC(A/\Cal O) \to (Rep_G)$ is equal to
 $(Rep_G)^S$. As a consequence, $(Rep_G)^S$ is equivalent to $IC(A/\Cal O)$.
\item
The category $(Rep_G)^S$ is equivalent to the category of
$H^0(Bar(A/\Cal O))$-comodules.
\end{enumerate}
\end{proposition}
\begin{proof}
Let $V$ be an element of $(Rep_G)^S$ and $F^*W$ be a filtration whose
 associate graded modules come from $\Cal O$-comodules.
We choose a splitting $\{W^{p}\}$ of $F^{\bullet}W$ in $W$ as $k$-vector spaces.
Then we have $F^{p}W=\oplus_{k\geq p}W^p$.
By the isomorphism $W^p\to Gr_W^p(W)$ of $k$ vector spaces, we
introduce a $\Cal O$ module structure on $W^p$.
The action of $G$ on $W$ defines a map $\nabla_{ji}$ in
$\underline{Hom}_G(W^i,W^j)=Hom_{\Cal O}(W^i,A^1\otimes_{\Cal O}W^j)$.
Since $W$ is a $G$-module, $(W^p,\nabla_{ij})$ defines a $A$-connection
$W_A$ relative to $S$. By Proposition 
\ref{prop:invariance for quasi-isom}, 
the natural functor $IC(A_{rig}/\Cal O)\to IC(A/\Cal O)$
is an equivalence of category, we have an object $W_{A'}$ in
 $IC(A_{rig}/\Cal O)$ such that the image is isomorphic to $W_{A}$.
Then one can check that $\rho(W_{A'})$ is isomorphic to the given $W$.
\end{proof}

\section{Mixed elliptic motif}
In this section, we define quasi-DG categories of naive mixed elliptic motives
$(MEM)$ and virtual mixed elliptic motives $(VMEM)$. Roughly speaking,
the DG category of virtual mixed elliptic motives is obtained
by adding objects which are homotopy equivalent to zero.
Therefore $(MEM)$ and $(VMEM)$ are weak homotopy equivalent.

We can not introduce a natural tensor structure on $(EM)$
with the distributive property as is explained in
\S \ref{reason why naive does not distributive}.
On the other hand, $(VMEM)$ has a distributive tensor structure,
which is necessary to obtain a shuffle product on
the bar complex
$Bar(\Cal C_{VEM})$. Using this shuffle product,
$H^0(Bar(\Cal C_{VEM}))$ becomes a Hopf algebra.
In this paper, Bloch cycle groups, higher Chow groups
are $\bold Q$-coefficients.

\subsection{Injectivity of linear Chow group}
\label{subsection injectivity of linear chow}

In this subsection, we prove the injectivity of linear Chow group,
which is also proved in \cite{BL}. The proof given here
is more direct.
\begin{proposition}
\label{injective to cohomology}
Let $E$ be an elliptic curve over a field $k$ which does not have complex
multiplication, that is $End_{\bar{k}}(E)=\bold Z$. 
Let $CH^*_{lin}(E^{n})$ be the subring of $CH^*(E^{n})$
generated by $f^*([0])\in CH^1(E^n)$, where $f$ is a homomorphism of abelian varieties
$$
f:E^n \to E.
$$
The scalar extension of $E$ to its algebraic closure is written as $\overline{E}$.
Then the cycle map
$$
cl:CH_{lin}^*(E^{n}) \to H^{2*}_{et}(\overline{E}^{n},\bold Q_l)
$$
is injective.
\end{proposition}

Let $\pi_i:E^n\to E$ and $\pi_{ij}:E^n \to E\times E$ 
be the projection to $i$-th and $ij$-th components, and
the diagonal in $E\times E$ is denoted by $\Delta$.
The classes $\pi_i^*([0])$ and $\pi_{ij}^*(\Delta)$ in $CH^1(E^n)$
are denoted by $p_i$ and $\Delta_{ij}$, respectively.
Let $f:E^n \to E$ be a homomorphism defined by $f(x_1, \dots,
x_n)=\sum_{i=1}^n a_ix_i$ for $a_i \in \bold Z$.
Let $\alpha, \beta$ be symplectic basis of
$H^1(\overline{E})=H^1_{et}(\overline{E},\bold Q_l)$.
For the copy of $E_i$, the corresponding symplectic basis in
$H^1(\overline{E_i})$ are denoted by $\alpha_i, \beta_i$.
then the class $cl(f^*([0]))\in H^2(\overline{E}^n)$ is equal to 
\begin{align*}
f^*(\alpha)  f^*(\beta)& =
(\sum_{i}a_i\alpha_i)(\sum_{i}a_i\beta_i)=\sum_{i}a_i^2\alpha_i\beta_i
+\sum_{i<j}a_ia_j(\alpha_i\beta_j-\beta_i\alpha_j) \\
&=\sum_{i}a_i^2cl(p_i)-\sum_{i<j}a_ia_jcl(\Delta_{ij}-p_i-p_j).
\end{align*}
Since $Pic^0(E^n)$ component of $f^*([0])$ is zero,
we have
$$
f^*([0])=\sum_{i}a_i^2p_i-\sum_{i<j}a_ia_j(\Delta_{ij}-p_i-p_j)
$$
in $CH^1(E^n)$. Therefore $CH^*_{lin}(E^n)$ is generated by
$p_i$ and $D_{ij}=-\Delta_{ij}+p_i+p_j$ for $i\neq j$.
We have $D_{ij}=D_{ji}$. By the map $\sigma:E_1\times E_2\to E_1\times E_2:(x,y)\to
(x,-y)$, we have $\sigma(D_{12})=-D_{12}$.
\begin{lemma}
\label{The useful lemma}
We have
\begin{align}
\label{the first rel}
& p_iD_{ij}=0,\\
\label{the second rel}
& D_{ij}D_{ik}+p_iD_{jk}=0, \quad \text{ and }\\
\label{the thrid rel}
& D_{ij}D_{kl}+
D_{ik}D_{jl}+
D_{il}D_{jk}=0
\end{align}
in $CH^2(E^n)$
for $\#\{i,j\}=2$, $\#\{i,j,k\}=3$ or $\#\{i,j,k,l\}=4$, respectively.
\end{lemma}
\begin{proof}
The equality $p_1D_{12}=p_1p_2-p_1\Delta=0$ is trivial.

In $CH^2(E_1\times E_2\times E_3)$, we have $\Delta_{12}\cap
 \Delta_{13}=\Delta_{12}\cap \Delta_{23}$.
Therefore
\begin{align*}
&(p_1+p_2-D_{12})
(p_1+p_3-D_{13})=
(p_1+p_2-D_{12})
(p_2+p_3-D_{23}),\\
& p_1p_2+p_2p_3+p_3p_1-p_3D_{12}-p_2D_{13}+D_{12}D_{13} \\
=&
p_1p_2+p_2p_3+p_3p_1-p_3D_{12}-p_1D_{23}+D_{12}D_{23}
\end{align*}
and we have
$$
-p_2D_{13}+D_{12}D_{13}=-p_1D_{23}+D_{12}D_{23}.
$$
Applying an automorphism $(x_1,x_2,x_3)\mapsto (x_1,-x_2,x_3)$, we have
$$
-p_2D_{13}-D_{12}D_{13}=+p_1D_{23}+D_{12}D_{23}
$$
and we get the equality 
\begin{equation}
\label{typical equality}
p_2D_{13}=-D_{12}D_{23}.
\end{equation}
(Chow groups are considered as $\bold Q$-coefficient.)
We consider the map $\varphi:E_1\times E_2\times E_3\times E_4 \to E_1\times
 E_2\times E_3$ defined by $(x_1,x_2,x_3,x_4) \mapsto (x_1,x_2-x_4,x_3)$.
Then applying $\varphi$ to the equality (\ref{typical equality})
and using the equality
$$
\varphi^*(p_2)=p_2+p_4-D_{42},\quad
\varphi^*(D_{12})=D_{12}-D_{14},\quad
\varphi^*(D_{23})=D_{23}-D_{34},
$$
we have
$$
(p_2+p_4-D_{42})D_{13}=-(D_{12}-D_{14})(D_{23}-D_{34}),
$$
which implies the third equality of the lemma.
\end{proof}
\begin{lemma}
\label{lemma generation}
$CH^m_{lin}(E^n)$ is generated by the set of elements of the form
\begin{equation}
\label{form of monomial}
p_{i_1}\cdots p_{i_p}D_{j_1k_1}\cdots D_{j_qk_q}
\end{equation}
such that 
\begin{enumerate}
\item
\label{first cond for monomials}
$i_1, \dots, i_p,j_1, \dots,j_q,k_1, \dots, k_q$ are distinct and
\item
\label{second cond for monomials}
(Gelfand-Zetlin condition)
$$
\begin{matrix}
j_1 & < & j_2 & <\cdots< & j_q \\
\bigwedge & & \bigwedge & & 
\bigwedge & & \\
k_1 & < & k_2 & <\cdots< & k_q \\
\end{matrix}.
$$
\end{enumerate}
\end{lemma}
\begin{proof}
By the relations
(\ref{the first rel}) and
(\ref{the second rel}), $CH^*_{lin}(E^n)$
is generated by the monomial of the form $(\ref{form of monomial})$
with the condition (\ref{first cond for monomials}) of Lemma
\ref{lemma generation}.
Let $W$ be the subspace of $CH^*(E^n)$ generated by monomials for the form 
$(\ref{form of monomial})$ with the conditions
(\ref{first cond for monomials}) and (\ref{second cond for monomials}) of Lemma
\ref{lemma generation}. Assume that there exists a monomial of the form
$(\ref{form of monomial})$ which is not an element of $W$.
We set $S=\{j_1, \dots, j_q,k_1,\dots, k_q\}$.
We consider the set $\Cal M$ of monomials of the form
\begin{equation}
\label{form of monomial copy}
M'=p_{i_1}\cdots p_{i_p}D_{j'_1k'_1}\cdots D_{j'_qk'_q},
\end{equation}
which is not contained in $W$ such that
$\{j'_1, \dots,j'_q,k'_1, \dots, k'_q\}=S$ and 
$$
\begin{matrix}
j'_1 & < & j'_2 & <\cdots< & j'_q \\
\bigwedge & & \bigwedge & & 
\bigwedge & & \\
k'_1 &  & k'_2 & \cdots & k'_q. \\
\end{matrix}
$$
By the assumption, $\Cal M$ is not an empty set.
Since the monomial (\ref{form of monomial copy}) is not contained in
$W$, there exists a $t$ such that $k'_t>k'_{t+1}$.
The minimal of such $t$ is denoted by $t(M')$.
We introduce a total order of $\Cal M$ by the lexicographic
order of
\begin{equation}
\label{total order}
(-t(M'),j'_1,k'_1, \dots, j'_q,k'_q).
\end{equation}
Let $M'$ be the minimal element in $\Cal M$ and set $t=t(M')<q$.
Then we have the equality
$$
\begin{matrix}
j'_1 & < & j'_2 & <\cdots< & j'_t & < & j'_{t+1} &\cdots \\
\bigwedge & & \bigwedge & & 
\bigwedge & & \bigwedge  \\
k'_1 & < & k'_2 & <\cdots< & k'_t & > & k'_{t+1} &\cdots \\
\end{matrix}.
$$
By the equality (\ref{the thrid rel}), we have
$$
D_{j'_t,k'_t}D_{j'_{t+1},k'_{t+1}}+
D_{j'_t,j'_{t+1}}D_{k'_{t+1},k'_t}+
D_{j'_t,k'_{t+1}}D_{j'_{t+1},k'_t}=0.
$$
Let $M''$ and $M'''$ be monomial obtained by replacing the factor
$D_{j'_t,k'_t}D_{j'_{t+1},k'_{t+1}}$ by
$D_{j'_t,j'_{t+1}}D_{k'_{t+1},k'_t}$ and
$D_{j'_t,k'_{t+1}}D_{j'_{t+1},k'_t}$.
Then $t(M'')\geq t(M')+1$ and $t(M''')\geq t(M')+1$.
Therefore $M'',M'''\in W$. Since $M'+M''+M'''=0$, we have $M'\in W$ and
a contradiction to the choice of $M'$.
\end{proof}
\begin{lemma}
\label{linear indpendentness of standard monomial}
The set of the images $cl(p_{i_1}\cdots p_{i_p}D_{j_1k_1}\cdots D_{j_qk_q})$
with the conditions 
(\ref{first cond for monomials}) and (\ref{second cond for monomials}) of Lemma
\ref{lemma generation} are linearly independent in $H^*(\overline{E}^n)$.
\end{lemma}
\begin{proof}
Since $cl(D_{ij})=\alpha_i\beta_j-\beta_i\alpha_j$,
it is enough to prove the linear independence for the set
of monomial for fixed $\{i_1, \dots, i_p\}$ and $S=\{j_1, k_1, \dots,
 j_q,k_q\}$.
For a monomial 
\begin{equation}
\label{monomial for alpha beta}
\alpha_{i_1}\beta_{i_1}\cdots \alpha_{i_p}\beta_{i_q}
\alpha_{j'_1}\beta_{k'_1}\cdots \alpha_{j'_q}\beta_{k'_q} 
\end{equation}
with $j'_1<\cdots <j'_q, k'_1<\cdots <k'_q$, we introduce lexicographic order by
 deleting ``$-t(M')$''-part of (\ref{total order}).
Let $p_{i_1}\cdots p_{i_p}D_{j'_1k'_1}\cdots D_{j'_qk'_q}$
be an element with the condition
(\ref{first cond for monomials}) and (\ref{second cond for monomials})
of Lemma \ref{lemma generation}.
Then the lowest monomial appeared in $cl(p_{i_1}\cdots p_{i_p}D_{j'_1k'_1}\cdots
 D_{j'_qk'_q})$ of the above form is equal to (\ref{monomial for alpha beta}).
Therefore they are independent.
\end{proof}
\begin{proof}[Proof of Proposition \ref{injective to cohomology}]
The proof of the proposition is a consequence of Lemma 
\ref{lemma generation} and \ref{linear indpendentness of standard monomial}.
\end{proof}

\begin{corollary}
Let $p_1, \dots, p_m$ be elements in 
the correspondence algebra  $CH^n(E^n\times E^n)$
of $E^n$, which are contained in
$CH_{lin}^n(E^n\times E^n)$.
The image of $p_i$ in $H^{2n}(E^n\times E^n,\bold Q)$
is written by $cl(p_i)$.
\begin{enumerate}
\item
If $cl(p_1)=cl(p_2)$, then
$p_1=p_2$ in $CH^n(E^n\times E^n,\bold Q)$.
In particular, the images
$p_1^*CH^i(E^n,j)$ and $p_2^*CH^i(E^n,j)$ of
the correspondences $p_1^*$ and $p_2^*$ in 
the higher Chow group
$CH^i(E^n,j)$ are equal.
\item
The element $p_i$ is a projector if and only if
the class $H^{2n}(E^n\times E^n,\bold Q)$ is 
a projector in the cohomological correspondence ring.
\item
The correspondences $\{p_i\}$ are orthogonal
(resp. complete set of projectors)
if and only if $\{cl(p_i)\}$ are orthogonal
(resp. complete set of projectors).
\end{enumerate}
\end{corollary}
\begin{corollary}
\label{two algebraic correspondence inducing det and Q}
Let $G=\frak S_2\rtimes \langle\sigma\rangle$
and $\rho$ be the character of $G$ defined by
$\rho((1,2))=1$ and $\rho(\sigma)=-1$. Here the element $\sigma$
is the inversion of $E$.
Let $\Delta^+$ and $\Delta^-$ be the diagonal divisor of $E\times E$ and
 the divisor
defined by $x+y=0$, where $(x,y)$ are the coordinates of $E\times E$.
The maps $Z^i(E\times E\times X,j)\to Z^i(E\times E\times X,j)$ 
on Bloch cycle groups induced by
\begin{align*}
 \varphi_1: 
Z\mapsto -\frac{1}{4}(\Delta^+-\Delta^-) \times pr_X(Z\cap
\{ (\Delta^+-\Delta^-)\times X\})
\end{align*}
and
$$
\varphi_2: 
Z\mapsto \frac{1}{\#G}\sum_{g\in G}\rho(g)g^*Z
$$
induce the same maps in the 
higher Chow group $CH^i(E\times E\times X,j)$.
\end{corollary}
\begin{proof}
The pull back of the divisors $\Delta^+$ and $\Delta^-$ by the
 $(i,j)$-component
are denoted by $\Delta_{ij}^+$ and $\Delta^-_{ij}$, respectively.
Since the maps $\varphi_1$ and $\varphi_2$ are induced by 
the algebraic correspondences
$$-\frac{1}{4}(\Delta^+_{12}-\Delta^-_{12}) \times (\Delta^+_{34}-\Delta^-_{34})$$
and
\begin{align*}
&\frac{1}{8}(\Delta_{13}^+\Delta_{24}^+
-\Delta_{13}^-\Delta_{24}^+
-\Delta_{13}^+\Delta_{24}^-
+\Delta_{13}^-\Delta_{24}^- \\
&+\Delta_{14}^+\Delta_{23}^+
-\Delta_{14}^-\Delta_{23}^+
-\Delta_{14}^+\Delta_{23}^-
+\Delta_{14}^-\Delta_{23}^-)
\end{align*}
in $CH^2_{lin}(E_1\times E_2\times E_3\times E_4)$.
By computing the images of two cycles in $H^4(E_1\times E_2\times
 E_3\times E_4)$,
we can check the identity.
\end{proof}

\subsection{Naive mixed elliptic motives}

\subsubsection{}

Let $S=GL(2)$ and $\Cal O=\Cal O_S$.
The natural two dimensional representation is written as $V$.
The set of isomorphism classes of irreducible representations
is written by $Irr_2$.
Then using alternating and symmetric tensor products, we have
$$
Irr_2=\{(Alt^2)^{\otimes n}\otimes Sym^m \mid 
n\in \bold Z, m \in \bold N\}.
$$
Let $E$ be an elliptic curve over a field $K$. We assume that $E$ does not have complex 
multiplication, i.e. $End_{\bar{K}}(E)=\bold Z$. 
Let $Z^i(X,j)=Z^{i,-j}(X)$ be the cubical anti-symmetric Bloch higher cycle group.
This becomes a complex which is denoted by $Z^{i,\bullet}(Z)$.
Let $A$ be a finite set. The $A$-power of the elliptic curve $E$
is denoted by $E^A$. 
The group $(\bold Z/2\bold Z)^A$ acts on $Z^{i,\bullet}(E^A)$ by
the inversions of elliptic curves for each component.
The $(-,\dots, -)$-part of $Z^{i,\bullet}(E^A)$ under the action of
$(\bold Z/2\bold Z)^A$ is denoted by $Z_-^{i,\bullet}(E^A)$.
The group $\frak S[A]$ acts on $Z_-^{i,\bullet}(E^A)$.
We define the complex $\Cal H^{\bullet}(E^A, E^B, k)$ by
\begin{equation}
\label{definition of hom as alg cycle and augmentation}
\Cal H^{\bullet}(E^A, E^B, k)=
\Lambda^* (A) \otimes Z^{a+k,\bullet}_{-}
(E^A\times E^B)\otimes \Lambda(B)[-2k],
\end{equation}
where $a=\# A$. 
The complexes $\Lambda^*(A)$ and $\Lambda(B)$ are defined
in \S \ref{orientation is introduced here}.

Then the groups $\frak S[A]$ and $\frak S[B]$ act on this complex.
If $A=[1,a], B=[1,b]$, then we have
\begin{align*}
\Cal H^i(E^A,E^B,k)
=&
f_a\wedge\cdots \wedge f_1\cdot
Z^{a+k}_{-}
(E^a\times E^b,2k+a-b-i)\cdot 
e_1\wedge\cdots \wedge e_b
\end{align*}

\subsubsection{Object and morphisms of $(EM)$}
For a natural number $n$, $E^{[1,n]}$ are denoted
by $E^n$. We define quasi-DG category 
of naive elliptic motives $(EM)$.
As for the definition of quasi-DG category, 
see \S \ref{definitions quasi-DG category}.
The objects and complexes of
homomorphisms of $(EM)$ are defined
as follows:
\begin{enumerate}
\item
An object of $(EM)$ is a finite dimensional complex of 
$\Cal O$-comodule.
\item
$Sym^a\otimes (Alt^2)^{\otimes (-p)}$ is denoted
by $Sym^a(p)$.
Let $Sym^a(s),Sym^b(t)$ be elements in $Irr_2$.
We set
\begin{align*}
&
\underline{Hom}_{EM}^{\bullet}(
Sym^a(s),
Sym^b(t)
) \\
=&
sym^a
\Cal H^{\bullet}(E^{a}, E^{b}, t-s)sym^b.
\end{align*}
\item
Let $U_1,U_2$ be finite dimensional complexes of
$\Cal O$-comodules.
The complex of homomorphisms is defined by
\begin{align*}
&\underline{Hom}_{EM}(U_1, U_2) \\
=&
\oplus_{V_1,V_2\in Irr_2}
Hom_{\Cal O}(U_1,V_1)
\otimes
\underline{Hom}_{EM}(V_1,V_2)
\otimes
Hom_{\Cal O}(V_2, U_2).
\end{align*}

\end{enumerate}
\subsubsection{Multiplication map for $(EM)$}

Here $\pi$ is the projector to the $\chi$-part 
for the action of group
and $Z^{\bullet}(X,\bullet)$ denotes the Bloch cycle group
of cubical type of $X$.
The intersection theory for Bloch cycle complexes we have
a ``homomorphism'' of complex
\begin{align*}
& \Pi:Z^{p_1}(E^{m_1}\times E^{m_2},
q_1) 
\otimes
Z^{p_2}(E^{m_2}\times
 E^{m_3},
q_2) \\
&\to
Z^{p_1+p_2-m_2}(E^{m_1}\times
 E^{m_3},
q_1+q_2).
\end{align*}
Roughly speaking, this intersection homomorphism $\Pi$ is 
defined by
$$
z\otimes w \mapsto pr_{13*}((z\times w)\cap 
(E^{m_1}\times \Delta_{m_2} \times E^{m_3} \times \square^{q_1+q_2})).
$$
Here $\Delta_m$ is the image of the diagonal map $E^m \to E^m \times E^m$ and
$pr_{13*}$ is the push forward for the map
$$
(E^{m_1}\times E^{m_2} \times E^{m_3} \times \square^{q_1+q_2}))\to
(E^{m_1}\times E^{m_3} \times \square^{q_1+q_2})).
$$
\subsubsection{}
\label{definitions quasi-DG category}
To get the correct definition of $\Pi$, it is necessary to
take a quasi-isomorphic subcomplex of
$Z^{p_1}(E^{m_1}\times E^{m_2},q_1) \otimes
Z^{p_2}(E^{m_2}\times E^{m_3},q_2)$ consisting of elements
which intersect properly to all the boundary stratification of
$(E^{m_1}\times \Delta_{m_2} \times E^{m_3} 
\times \square^{q_1+q_2}))$.
See \cite{Han}, Proposition 1.3, p.112, and \cite{Le}, Corollary 4.8,
p.297, for details. 
A homomorphism of complex which is defined only on
a quasi-isomorphic subcomplex is called a quasi-morphism.
Similarly, a DGA whose ``multiplication'' is defined only
on a quasi-isomorphic subcomplex is called a quasi-DGA.
We can similarly define quasi-DG category, quasi-comodule over a quasi-DG Hopf 
algebra. 
\subsubsection{}
Thus by taking projector and the above intersection pairing,
we have a quasi-morphism of complexes.
\begin{align}
\label{eqn:product str}
&
\underline{Hom}_{EM}^i(
Sym^{m_1}(q_1),
Sym^{m_2}(q_2)
) 
\otimes 
\underline{Hom}_{EM}^j(
Sym^{m_2}(q_2),
Sym^{m_3}(q_3)
) \\
\nonumber
\to &
\underline{Hom}_{EM}^{i+j}(
Sym^{m_1}(q_1),
Sym^{m_3}(q_3)
) 
\end{align}
Here the sign rule for the pairing of ``orientations''
are given by
$$
(e_1\wedge\cdots e_a)\otimes (f_a\wedge\cdots \wedge f_1)
\mapsto 1.
$$
Using the above pairing, we define quasi-DGA's 
$(EM)$, $(MEM)$ as follows.
\begin{definition}
\label{definition:naive mixed elliptic motives}
\begin{enumerate}
\item
We define a relative quasi-DGA $A_{EM}$ by
$$
A_{EM}=\oplus_{V_1,V_2\in Irr_2}V_1\otimes
     \underline{Hom}_{EM}(V_1,V_2)\otimes V_2^*
$$
The multiplication map $\mu$ is given by 
\begin{align*}
&A_{EM}\otimes_\Cal O A_{EM} \\
=&\oplus_{V_1,V_2,V_3\in Irr_2}
V_1\otimes 
\underline{Hom}_{EM}(V_1,V_2)
\otimes 
\underline{Hom}_{EM}(V_2,V_3)
\otimes V_3^* \\
\to &
\oplus_{V_1,V_3\in Irr_2}
V_1\otimes 
\underline{Hom}_{EM}(V_1,V_3)
\otimes V_3^*,
\end{align*}
where the last arrow is induced by the multiplication map
of (\ref{eqn:product str}).
\item
We define the quasi-DG category $(EM)$
by $C(A_{EM}/\Cal O)$ defined in 
\S \ref{definition of DG cat associate to relative DGA}.
\item
The quasi-DG category $K(EM)$ of DG complexes in $(EM)$
is called the category of mixed elliptic motives
and is denoted by $(MEM)$.
\end{enumerate}
\end{definition}
Let $U_1, U_2$ and $U_3$ be an object in $(EM)$.
Then we have an identification
$$
\underline{Hom}_{EM}(U_1,U_2)\simeq
\oplus_{V_1,V_2\in Irr_2}Hom_{\Cal O}(U_1, V_1)\otimes
\underline{Hom}_{EM}(V_1,V_2)\otimes
Hom_{\Cal O}(V_2, U_2).
$$
The multiplication is given by the following
composite map:
\begin{align}
\label{rule for composite in general}
&\underline{Hom}_{EM}(U_1,U_2)\otimes
\underline{Hom}_{EM}(U_2,U_3) \\
\nonumber
=&
\oplus_{V_1,V_2,V_2',V_3\in Irr_2}
Hom_{\Cal O}(U_1, V_1)\otimes
\underline{Hom}_{EM}(V_1,V_2)\otimes
Hom_{\Cal O}(V_2, U_2) \\
\nonumber
&\otimes 
Hom_{\Cal O}(U_2, V_2')\otimes
\underline{Hom}_{EM}(V_2',V_3)
Hom_{\Cal O}(V_3, U_3) \\
\nonumber
\xrightarrow{\alpha} &
\oplus_{V_1,V_2,V_3\in Irr_2}
Hom_{\Cal O}(U_1, V_1)\otimes
\underline{Hom}_{EM}(V_1,V_2) \\
\nonumber
&\hskip 0.5in \otimes 
\underline{Hom}_{EM}(V_2,V_3)
\otimes Hom_{\Cal O}(V_3, U_3) \\
\nonumber
\xrightarrow{\beta} &
\oplus_{V_1,V_3\in Irr_2}
Hom_{\Cal O}(U_1, V_1)\otimes
\underline{Hom}_{EM}(V_1,V_3)
\otimes Hom_{\Cal O}(V_3, U_3)\\
\nonumber
&=\underline{Hom}_{EM}(U_1,U_3), 
\end{align}
where $\alpha$
is induced by the multiplication in $\Cal O$-homomorphisms
$$
Hom_{\Cal O}(V_2, U_2)
\otimes 
Hom_{\Cal O}(U_2, V_2)
\to
Hom_{\Cal O}(V_2, V_2)\simeq \bold k
$$
for $V_2=V_2'$ 
and the map $\beta$ is defined in
(\ref{eqn:product str}).
The composite map for $(EM)$ is defined by the rule
(\ref{from yoneda to composite}).
\subsubsection{Augmentation}
We introduce an augmentation $\epsilon_{EM}:A_{EM}\to \bold k$.
The augmentation is defined by the composite map
\begin{align*}
&\underline{Hom}_{EM}^0(Sym^n(q),Sym^n(q))
\\
\simeq
&sym^n(f_n\wedge\cdots \wedge f_1)Z_{-}^n(E^n\times E^n,0)(e_1\wedge\cdots e_n)sym^n
\\
\to
&sym^n(f_n\wedge\cdots \wedge f_1)CH_{-}^n(E^n\times E^n,0)(e_1\wedge\cdots e_n)sym^n
\\
\to
&sym^n(f_n\wedge\cdots \wedge f_1)CH_{hom,-}^n(E^n\times E^n,0)(e_1\wedge\cdots e_n)sym^n
\\
\simeq &\bold Q,
\end{align*}
where $CH_{hom,-}^*$ is the $(-)$-part of the Chow group modulo
homological equivalence.
Here we use the assumption that the elliptic curve $E$ has no complex multiplication.
On the component $\underline{Hom}_{EM}^i(Sym^{n_1}(q_1),Sym^{n_2}(q_2))$,
where either $i\neq 0$, $n_1\neq n_2$ or $q_1\neq q_2$, the augmentation
is set to the zero map.
By Theorem \ref{main theorem genereal}, 
we have  the following theorem.
\begin{theorem}
\label{main theorem for naive mixed elliptic motives}
The quasi-DG category $(MEM)$ of naive mixed elliptic motives 
is homotopy equivalent to the
quasi-DG category $(Bar(A_{EM}/\Cal O,\epsilon_{EM})-com)$
of $Bar(A_{EM}/\Cal O,\epsilon_{EM})$-comodules. 
\end{theorem}
\subsection{An application of Schur-Weyl reciprocity}
In this section, we review some properties of group ring of symmetric group.
For a finite set $A$, the symmetric group of the set $A$ is denoted by
$\frak S[A]$. Assume that a finite set $A$ is equipped with a total order.
A sequence of integers $(\lambda_1, \dots, \lambda_p)$ is called
a partition if $\lambda_1 \geq \cdots \geq \lambda_n\geq 0$.
A tableau $Y$ consisting of a set $A$ is defined in \cite{M}.
Then the support of $Y$ is a partition.

Let $W$ be a vector space and $W_i$ be a copy of $W$ indexed by
$i \in A$. The tensor product $\otimes_{i\in A} W_i$ is
denoted by $W^{\otimes A}$.
The set of tableaux consisting of a finite set $A$
is denoted by $Tab(A)$.

For a tableau $Y$ consisting of $A$, we can define
a projector $e_Y\in \bold Q[\frak S[A]]$ by
$e_Y=e_{sym_Y}e_{alt_Y}$,
where $e_{sym_Y}$ and $e_{alt_Y}$ are
symmetric and anti-symmetric projectors for
horizontal and vertical symmetric subgroups.
(See \cite{M}.) 
The projector $e_Y$ defines a $GL(W)$-homomorphism
$W^{\otimes A} \to W^{\otimes A}$ and the image 
$$
e_YW^{\otimes A}=M_Y(W)
$$
is an irreducible 
representation of $GL(W)$ and $M_Y$ becomes a functor from 
$(Vect_{\bold k})$ to $(Vect_{\bold k})$.
The isomorphism class of the functor $M_Y$ depends only on the 
support of the tableaux $Y$ of $A$, and
the set of isomorphism class of irreducible representation
contained in $W^{\otimes n}$ corresponds in one to one
with the set of partition of $n$.

On the other hand, 
for any projector $p:W^{\otimes A}\to W^{\otimes A}$ of $GL(W)$-modules,
there exists an 
idempotent $e_P \in \bold Q[\frak S[A]]$ which induces the projector $p$.
Moreover, if $\#A < \dim(W)$, 
the idempotent $e_p$ which induces the projector
$p$ is unique by Schur-Weyl reciprocity \cite{W}, p.150.

\begin{definition}
\label{definition of adjoint for group ring}
The linear map $\bold Q[Isom(S,S')]\to \bold Q[Isom(S',S)]$ defined by
$g\mapsto g^{-1}$ is called the adjoint map. The adjoint of $x$
is denoted by $x^*$. An element $x$ in $\bold Q[\frak S[A]]$ is called
self-adjoint if $x=x^*$. We have $e_Y^*=e_{alt_Y}e_{sym_Y}$.
\end{definition}

For three sets $S,S'$ and $S''$ with the same cardinality,
and $x\in \bold Q[Isom(S,S')]$ and $y\in \bold Q[Isom(S',S'')]$,
we have $(yx)^*=x^*y^*$.

We can reformulate for the set of
isomorphisms between $A$ and $A'$ instead of automorphism of $A$.
Let $A$ and $A'$ be a finite set such that $\# A =\# A'<\dim(W)$.
Then any $GL(W)$-equivariant homomorphism
$\varphi:W^{\otimes A} \to W^{\otimes A'}$ is induced by
a unique element $e_{\varphi}$ in $\bold Q[Isom(A,A')]$.
This action is written from the left.
An element $e$ of $\bold Q[Isom(A',A)]$
acts on $W^{\otimes A}$ from the right via the conjugate 
$e^*$ of $e$.
Thus properties for elements in
$\bold Q[\frak S[A]]$ and $\bold Q[Isom(A,A')]$ is reduced to
that of $GL(W)$ equivariant homomorphisms between $W^{\otimes A}$
and $W^{\otimes A'}$. In other words, the natural map
\begin{equation}
\label{schur-weyl duality}
\bold Q[Isom(S,S')] \to Hom_{GL(W)}(W^{\otimes S},W^{\otimes S'})
\end{equation}
is an isomorphism.
Via this isomorphism, we have the following isomorphism
$$
Hom_{GL(W)}(M_{Y_1},M_{Y_2})\simeq e_{Y_2}\bold 
Q[Isom(\Supp(Y_1),Supp(Y_2))]e_{Y_1}.
$$

The symmetric product $Sym_A(W)$ and alternating product $Alt_A(W)$
are defined as subspaces of $W^{\otimes A}$ corresponding to
the Young tableaux with supports $(n )$ and
$(1,1, \dots, 1)=(1^n)$, where $n=\#A$. The associate idempotents
are written as $sym_A$ and $alt_A$.
For a subset $B$ of $A$, the element
$sym_B$ and $alt_B$ can be regarded as 
elements in $\bold Q[\frak S[A]]$.

\begin{definition}[Category $(GL_{\infty})$]
Let $Y_i$ be Young tableaux and $V_i$ complexes of vector spaces.
The direct sum $\oplus_iM_{Y_i}\otimes V_i$ becomes
a functor from $(Vect)$ to $(KVect)$.
A functor which is isomorphic to this form 
is called a Schur functor.
The full subcategory of $(Vect)\to (KVect)$
consisting of Schur functors 
is denoted by $(GL_{\infty})$.
\end{definition}
\begin{lemma}
\label{schur closed under tensor}
The category $(GL_{\infty})$
is closed under tensor product.
\end{lemma}
The set of partitions is denoted by $\Cal P$
We choose a tableau $Y$ for each partition $\mid Y \mid$ and
the set of chosen tableaux is denoted by $\widetilde{\Cal P}$.
\begin{proof}[Proof of Lemma \ref{schur closed under tensor}]
Let $Y_1$ and $Y_2$ be Young tableaux.
The tensor product $M_{Y_1}\otimes M_{Y_2}$ of
$M_{Y_1}$ and $M_{Y_2}$ is
isomorphic to
$$
\oplus_{Y \in \widetilde{\Cal P}}
Hom_{(GL_{\infty})}(M_Y,M_{Y_1}\otimes M_{Y_2})\otimes M_Y.
$$
The the subspace $Hom_{(GL_{\infty})}(M_Y,M_{Y_1}\otimes M_{Y_2})$
is finite dimensional and it is zero for finite number 
of partitions $\mid Y \mid\in \Cal P$.
\end{proof}

\subsection{Virtual mixed elliptic motives}
We use the set $\widetilde{\Cal P}$ of Young tableaux chosen in the last subsection.

\subsubsection{}
In this section, we define the quasi-DG category of
virtual elliptic motives $(VEM)$ for the elliptic curve 
$E$. 
\begin{definition}
\begin{enumerate}
\item
An object of $(VEM)$ is a direct sum of
symbolic Tate twisted object $M(p)$, where $M$
is an object in $ob(GL_{\infty})$.
\item
Let $M_1(p_1)$ and $M_2(p_2)$ be objects in $(VEM)$.
We define the complex of 
homomorphisms by
\begin{align*}
&\underline{Hom}_{VEM}^{\bullet}(
M_1(p_1),
M_2(p_2)) \\
= &
\oplus_{Y_1,Y_2}\big(
Hom_{GL_{\infty}}(M_1,M_{Y_1}) 
\otimes e_{Y_1}
\Cal H^{\bullet}
(E^{s(Y_1)},
E^{s(Y_2)},p_2-p_1)
e_{Y_2} \\
&\qquad\otimes 
Hom_{GL_{\infty}}(M_{Y_2},M_2)\big),
\end{align*}
where $\Cal H^{\bullet}(*,*)$ is defined in
(\ref{definition of hom as alg cycle and augmentation}).
\item
The multiplication is
defined by the intersection theory and
the rule (\ref{rule for composite in general}).
Composite of complex of homomorphisms is defined by 
\[f\circ g=(-1)^{{\rm deg}(f){\rm deg }(g)}g\cdot f,\]
where $\cdot$ is the multiplication.
\item
We define the quasi-DG category of virtual mixed elliptic motives
$(VMEM)$ by
$K(VEM)$.
\end{enumerate}
\end{definition}

\subsubsection{}
We define Young tableaux $Y_{p,n}$ as follows.
\begin{equation*}
Y_{p,m}=
\left(\begin{matrix}
1 & 3 & \cdots & 2p-1 & 2p+1 & \cdots 2p+m \\
2 & 4 & \cdots & 2p   &      &
\end{matrix}
\right)
.
\end{equation*}

\begin{proposition}
\label{essentially surjectivity}
\begin{enumerate}
\item
The objects $Sym^n$ and $M_{Y,p}(p)$ of $(VEM)$ are homotopy
equivalent in the sense of 
Definition \ref{definition weak homotopy equivalence} (2).

\item
If either the depth of $Y_1$ or that of $Y_2$ is
greater than 2, then
the complex
$
\underline{Hom}^i(
M_{Y_1}(s),
M_{Y_2}(t))
$
is acyclic.
\end{enumerate}
\end{proposition}
\begin{proof}
(1) 
We denote the set $\{1,\cdots, 2p+n\}$ by $A$. Let $p_A\in \bold Q[{\frak S}[A]]$
be the projector 
\[alt(1,2)alt(3,4)\cdots alt(2p-1,2p)sym^n.\]
The projector $p_A^*$ induces the isomorphism
\[e_{Y_{p,n}}^*H^1(E)^{\otimes A}\simeq p_A^*H^1(E)^{\otimes A}\]
of the cohomology groups. By Corollary 4.5(1) and the Schur-Weyl
reciprocity, it induces a homotopy equivalence
between the objects $M_{p_A}(p)$ and $M_{Y_{p,n}}(p)$. So it suffices to show that
the objects $Sym^n$ and $M_{p_A}(p)$ are homotopy equivalent. Let 
$pr_2:E^n\times E^{2p+n}=E^n\times E^{2p}\times E^n\to E^{2p}$
be the projection to the second factor and 
$pr_{13}:E^n\times E^{2p+n}=E^n\times E^{2p}\times E^n\to E^n\times E^n$
be the projection to the product of the first and the third factors.

\no Let
\begin{align}
\label{quasi-iso for depth 2(I)}
&I \in \underline{Hom}^0(
Sym^n,
M_{p_A}(p))
\end{align}
be the closed homomorphism of degree zero defined by
$$
sym^nf_n\wedge \cdots \wedge f_1
\{(pr_2^*(-\frac{1}{4}(\Delta^+-\Delta^-))^p)\cap pr_{13}^*\Delta_{E^n}\}_-e_1\wedge
\cdots \wedge e_{2p+n}p_A,
$$
where the subscript ${}_-$ means the $(-)$-part by the action of the group
$(\Z/2\Z)^n\ti (\Z/2\Z)^A$,
and let
\begin{align}
\label{quasi-iso for depth 2(I)}
&J \in \underline{Hom}^0(
M_{p_A}(p),
Sym^n)
\end{align}
be the element defined by
$$
p_Af_{2p+n}\wedge\cdots \wedge f_1\{\Delta^p
\ti \Delta_{E^n}\}_-e_1\wedge\cdots \wedge e_nsym^n.
$$
The composite class $[J]\circ [I]\in 
H^0(\underline{Hom}(Sym^n,Sym^n))$
is equal to the class of the diagonal by the intersection
equality 
$$
[-\frac{1}{4}(\Delta^+-\Delta^-)]\cap [\Delta]=\{0\}
.
$$ 
By changing the cycles $I$ and $J$ to their rationally equivalent
ones $[I']$ and $[J']$,
the class
$[I'\circ J']\in 
H^0\underline{Hom}(M_{Y_{p,n}}(V)(p),M_{Y_{p,n}}(V)(p))$
is homologous to the diagonal from Corollary 4.6.

\no (2) Let $p_1=\sharp |Y_1|$. If the depth of $Y_1$ is greater than 2,
then the cohomology class of 
\[{\bf 1}_{M_{Y_2}}=e_{Y_1}\{\Delta_{E^{p_1}}\}_-e_{Y_1}
\in Z^{p_1}_-(E^{p_1}\ti E^{p_1})\] is zero. It follows from
Proposition 4.1 that this class is rationally trivial. 
\end{proof}

\begin{definition}
\label{definition weak homotopy equivalence}
Let 
$F:\Cal C_1 \to \Cal C_2$
be a functor of two DG-categories $\Cal C_1$ and $\Cal C_2$.
\begin{enumerate}
\item
The functor $F$ is said to be homotopy fully faithful if
the map
$$
\underline{Hom}^{\bullet}_{\Cal C_1}(A, B) \to
\underline{Hom}^{\bullet}_{\Cal C_2}(F(A), F(B))
$$
is a quasi-isomorphism for all $A,B \in ob(\Cal C_1)$.
\item
Two objects $M,N$ in $\Cal C_2$ are said to be homotopy equivalent
if there exist closed homomorphisms $I,J$ of degree zero
$$
I \in \underline{Hom}^0(M,N), \quad
J \in \underline{Hom}^0(N,M), 
$$
such that the cohomology classes $I\circ J$ and $J\circ I$ are equal to
identities.
\item
The functor $F$ is said to be homotopy essentially surjective
if for any object $M$ in $\Cal C_2$, there exists an object $N$
in $\Cal C_1$ such that $F(N)$ is homotopy equivalent to $M$.
\item
The functor $F$ is said to be weak homotopy equivalent
if it is homotopy fully faithful and homotopy essentially surjective.
\end{enumerate}
\end{definition}
The following proposition is a consequence of Proposition
\ref{essentially surjectivity}.
\begin{proposition}
The natural functor
$(MEM)\to (VMEM)$ is weak homotopy equivalent.
\end{proposition}

\subsection{Bar complex of a small DG-category}

\subsubsection{}
Let $T_2$ be a set. We consider a small DG category $\Cal C_2$ whose class of objects is
the set $T_2$.
The complex of homomorphisms $\underline{Hom}_{\Cal C_2}(V_1,V_2)$
is denoted by $H(V_1, V_2)$.
The multiplication
$$
\eta:H(V_1,V_2)\otimes H(V_2,V_3) \to H(V_1, V_3)
$$
is defined by the composite of the transposition
and the composite of the complex of homomorphisms.
Let $ev:\Cal C_2\to Vect_{\bold k}$ be a
functor of DG categories.
Here the category $Vect_{\bold k}$ of vector spaces is a 
DG category in obvious way.
The functor $ev$ is called an augmentation of $\Cal C_2$.

In this subsection, we define $Bar(\Cal C_2,ev)$.
For a finite sequence $\alpha=(\alpha_0<\cdots <\alpha_n)$ of integers, we define
$Bar^{\alpha}(\Cal C_2,ev)$ by
$$
\oplus_{V_0,\dots, V_n\in T_2}
\bigg(
ev(V_0)
\overset{\alpha_0}\otimes H(V_0,V_1)
\overset{\alpha_1}\otimes
\cdots \overset{\alpha_{n-1}}\otimes
H(V_{n-1},V_n)
\overset{\alpha_n}\otimes
ev(V_n)^*
\bigg).
$$
We define a complex $\bold{Bar}_n(\Cal C_2,ev)$ by
$$
\oplus_{\mid \alpha\mid =n}\bold{Bar}^{\alpha}(\Cal C_2,ev).
$$
For $\beta=(\alpha_0, \dots, \widehat{\alpha_k},\dots, \alpha_n)$,
we define
$$
d_{\beta\alpha}:
\bold{Bar}^{\alpha}(\Cal C_2,ev) \to
\bold{Bar}^{\beta}(\Cal C_2,ev)
$$
by
\begin{align*}
&d_{\beta,\alpha}(
v_0\otimes \varphi_{01}\otimes \cdots \otimes
 \varphi_{n-1n}\otimes v_n^*) \\
=&(-1)^k\cdot
\begin{cases}
ev(\varphi_{01})(v_0)\otimes \varphi_{12}\otimes\cdots 
\otimes
 \varphi_{n-1n}\otimes 
v_n^*  &(\text{if }k=0) \\
v_0\otimes 
\cdots \otimes
(\varphi_{k-1,k}\cdot\varphi_{k,k+1})\otimes \cdots 
\otimes 
v_n^*  &(\text{if }1\leq k\leq n) \\
v_0\otimes \varphi_{01}\otimes\cdots 
\otimes\varphi_{n-2,n-1}
\otimes
(v_n^*\circ ev(\varphi_{n-1n}))  &(\text{if }k=n).
\end{cases}
\end{align*}
By taking summation for $\alpha$ and $\beta$,
we have a double complex
\begin{align*}
\cdots \to \bold{Bar}_n(\Cal C_2,ev) &\to
\bold{Bar}_{n-1}(\Cal C_2,ev)\to \dots \\
&\to\bold{Bar}_1(\Cal C_2,ev)\to
\bold{Bar}_0(\Cal C_2,ev)\to 0.
\end{align*}
\begin{definition}
The bar complex $Bar(\Cal C_2,ev)$
is defined by the associate simple complex of $\bold{Bar}(\Cal C_2,ev)$.
\end{definition}

\subsubsection{}
Let $T_1$ and $T_2$ be sets and we consider a surjective map
$\varphi:T_1\to T_2$ and a section 
$\psi:T_2\to T_1$, i.e. $\varphi\circ\psi=id_{T_2}$.
We introduce a DG category $\varphi^*\Cal C_2$ as follows.
\begin{enumerate}
\item
The class of objects is defined to be $T_1$.
\item
For elements $V_1,V_2$ in $T_1$, the complex of homomorphism
$\underline{Hom}_{\varphi^*\Cal C_2}(V_1,V_2)$
is defined to be $\underline{Hom}_{\Cal C_2}
(\varphi(V_1),\varphi(V_2))$.
\item
The composite of the complex of homomorphisms are defined to be those 
in $\Cal C_2$.
\end{enumerate}
Note that if $\varphi(V_1)=\varphi(V_2)=W$, then $V_1$ and $V_2$ are
canonically isomorphic in $\varphi^*\Cal C_2$. The composite of functors 
$ev\circ\varphi:\varphi^*\Cal C_2\to \Cal C_2\to Vect_{\bold k}$ defines an augmentation
of $\varphi^*\Cal C_2$.

Then we also have a bar complex $Bar(\varphi^*\Cal C_2,ev\circ\varphi)$.
The natural functor $\varphi:\varphi^*\Cal C_2 \to \Cal C_2$ defines a
homomorphism of bar complexes
$$
Bar(\varphi):
Bar(\varphi^*\Cal C_2, ev\circ \varphi) \to Bar(\Cal C_2, ev).
$$
\begin{proposition}
\label{quasi-isomorphism for surjective map}
The homomorphism $Bar(\varphi)$ is a quasi-isomorphism.
\end{proposition}
\begin{proof}
The map $\psi:T_2 \to T_1$ defines a homomorphism of complex
$$
Bar(\psi):Bar(\Cal C_2, ev)\to Bar(\varphi^*\Cal C_2,ev\circ\varphi).
$$
One can check that the composite $Bar(\varphi)\circ Bar(\psi)$
is the identity. We will show that $Bar(\psi)\circ Bar(\varphi)$ induces
the identity on the cohomologies.

Let $N$ be an integer. By taking summations for 
$\alpha=(\alpha_0<\cdots <\alpha_n)$ such that
$N<\alpha_0$, we obtain a subcomplex $Bar(\Cal C_2,ev)_{N<}$
of $Bar(\Cal C_2,ev)$. 
(In the proof of Proposition \ref{acyclicity for simplicial bar
 complex}, we used the similar notation.)
We define a map 
$$
\theta:Bar(\varphi^*\Cal C_2,ev\circ\varphi)_{N<} \to
Bar(\varphi^*\Cal C_2,ev\circ\varphi)_{N-1<}
$$
of degree $-1$
such that
\begin{equation}
\label{homotopy for bar retraction}
d\theta+\theta d=Bar(\psi)\circ Bar(\varphi)-1
\end{equation}
on
$VBar(\varphi^*\Cal C_2,ev\circ\varphi)_{N<}$. The composite map
 $\psi\circ\varphi$ is denoted by $r$ and
$H(V_{i},V_{i+1}),H(V_{i},r(V_{i}))$ and $H(r(V_{i}),r(V_{i+1}))$
are denoted by $H_{i,i+1},H_{i,\overline{i}}$ 
and $H_{\overline{i},\overline{i+1}}$.
The element in $H_{\overline{i},\overline{i+1}}$
corresponding to $\varphi_{i,i+1}$ in $H_{i,i+1}$ is denoted by
$\varphi_{\overline{i},\overline{i+1}}$.
The element in $H_{i\overline{i}}$ corresponding to the identity is
denoted by $id_{i\overline{i}}$.
For $i=0, \dots, n$, we define $\theta_1^{(i)}$ as follows.
\begin{align*}
\theta_1^{(i)}:&ev(V_0)^*
\overset{\alpha_0}\otimes H_{0,1}
\overset{\alpha_1}\otimes
\cdots \overset{\alpha_{n-1}}\otimes
H_{n-1,n}
\overset{\alpha_n}\otimes
ev(V_n) \\
\to\  & ev(V_0)^*
\overset{N}\otimes H_{0,1}
\overset{\alpha_0}\otimes 
\cdots 
\overset{\alpha_{i-2}}\otimes
H_{i-1,i}
\overset{\alpha_{i-1}}\otimes
H_{i,\overline{i}}
\overset{\alpha_{i}}\otimes
H_{\overline{i},\overline{i+1}}
\cdots 
\overset{\alpha_n}\otimes
ev(r(V_n)) \\
%
&v_0^*
\overset{\alpha_0}\otimes \varphi_{0,1}
\overset{\alpha_1}\otimes
\cdots \overset{\alpha_{n-1}}\otimes
\varphi_{n-1,n}
\overset{\alpha_n}\otimes
v_n \\
\mapsto\  & v_0^*
\overset{N}\otimes \varphi_{0,1}
\overset{\alpha_0}\otimes 
\cdots 
\overset{\alpha_{i-2}}\otimes
\varphi_{i-1,i}
\overset{\alpha_{i-1}}\otimes
id_{i,\overline{i}}
\overset{\alpha_{i}}\otimes
\varphi_{\overline{i},\overline{i+1}}
\cdots 
\overset{\alpha_n}\otimes
v_{\overline{n}}) 
\end{align*}
The map $\theta^{(i)}_2$
\begin{align*}
\theta_1^{(i)}:&ev(V_0)^*
\overset{\alpha_0}\otimes H_{0,1}
\overset{\alpha_1}\otimes
\cdots \overset{\alpha_{n-1}}\otimes
H_{n-1,n}
\overset{\alpha_n}\otimes
ev(V_n) \\
\to\  & ev(V_0)^*
\overset{N}\otimes H_{0,1}
\overset{\alpha_0}\otimes 
\cdots 
\overset{\alpha_{i-2}}\otimes
H_{i-1,i}
\overset{\alpha_{i-1}}\otimes
H_{i,i}
\overset{\alpha_{i}}\otimes
H_{i,i+1}
\cdots 
\overset{\alpha_n}\otimes
ev(V_n) \\
\end{align*}
for $i=0, \dots, n$ is defined similarly.
Then one can check that the identity (\ref{homotopy for bar retraction})
holds for $\theta=\sum_{i=0}^n(-1)^i(\theta_1^{(i)}-\theta_2^{(i)})$.
By taking the inductive limit for $N$, we have the proposition.
\end{proof}

\subsubsection{}

Let $TT=\{(Y,p)\}$ the set of twisted tableaux, i.e. pairs of $Y \in
\widetilde{\Cal P}$ and $p\in \bold Z$. 
We define a small DG category $\Cal C_{VEM}$
as follows.
\begin{enumerate}
\item
The class of objects
is $TT$.
\item
For elements $V_1=(Y_1,p_1),V_2=(Y_2,p_2)\in TT$, the complex of
homomorphism is given by
$$
H(V_1,V_2)=
\underline{Hom}_{VEM}(M_{Y_1}(p_1),M_{Y_2}(p_2)).
$$
\end{enumerate}
For a two dimensional vector space $V$ let $ev_V$ be a functor from $(VEM)$
to $(KVect)$ defined by 
$$
M_Y(p)\mapsto 
M_Y(V)\otimes (Alt^2(V))^{\otimes (-p)}.
$$
We have defined the bar complex $Bar(\Cal C_{VEM},ev_V)$.
Let $TT_{\leq 1}$ be the set of twisted tableaux of depth 
smaller than or equal to one
and $\Cal C_{EM}$ be the full sub DG category of 
$M_Y$ with $Y\in TT_{\leq 1}$. If $V$ is the standard representation
of the group $GL_2$ then we have 
$$
Bar(A_{EM},\epsilon)=Bar(\Cal C_{EM},ev_V).
$$

By Proposition \ref{essentially surjectivity} and Proposition
\ref{quasi-isomorphism for surjective map}, we have the following
proposition.
\begin{proposition}
\label{proposition: homotopy equiv for VEM and Bar(VEM)}
The natural homomorphism 
$$
Bar(\Cal C_{EM},ev_V) \to 
Bar(\Cal C_{VEM},ev_V)
$$
is a quasi-isomorphism.
\end{proposition}
\begin{proof}
Let $TT_{\leq 1}$ be the set of twisted tableaux of depth smaller than
or equal to one. We define a surjective map $\varphi:TT\to TT_{\leq 1}$
 by
$$
\varphi(Y)=
\begin{cases}
Y_{0,m}(q-p) & \text{ if }Y=Y_{p,m}(q) \\
0 & \text{ otherwise}.
\end{cases}
$$
Then the natural inclusion $\psi:TT_{\leq 1}\to TT$ is a section of $\varphi.$
By Proposition \ref{quasi-isomorphism for surjective map}, the map
$
Bar(\varphi):Bar(\varphi^*\Cal C_{EM})\to Bar(\Cal C_{EM})
$
is a quasi-isomorphism. As a consequence, the natural inclusion
\begin{equation}
\label{eq:1st quasi-iso in proof}
Bar(\psi):Bar(\Cal C_{EM})\to Bar(\varphi^*\Cal C_{EM}) 
\end{equation}
is also a quasi-isomorphism.
We define a DG functor $\alpha:\varphi^*\Cal C_{EM}\to \Cal C_{VEM}$ as follows.
\begin{enumerate}
\item 
The map on the set of objects are identity.
\item
For objects $M_{Y_1}, M_{Y_2}$ with $Y_1, Y_2\in TT$, the
map
\begin{align*}
\alpha_{Y_1,Y_2}:&\underline{Hom}_{\varphi^*\Cal C_{EM}}(M_{Y_1},M_{Y_2}) \\
=&
\underline{Hom}_{\Cal C_{VEM}}(M_{\varphi(Y_1)},M_{\varphi(Y_2)})
\to
\underline{Hom}_{\Cal C_{VEM}}(M_{Y_1},M_{Y_2})
\end{align*}
is obtained by the homotopy equivalence given in Proposition 4.12. 
\end{enumerate}
Then the composite of the functor $\alpha$ and the augmentation map $ev_V:\Cal C_{VEM}\to Vect$
is equal to the augmentation map $ev_V$ of $\varphi^*\Cal C_{EM}$.
Therefore we have the following homomorphism of bar complexes:
\begin{equation}
\label{eq:2nd quasi-iso in proof}
Bar(\varphi^*\Cal C_{EM},ev_V)\to Bar(\Cal C_{VEM},ev_V).
\end{equation}
Since the maps $\alpha_{Y_1,Y_2}$ are quasi-isomorphisms,
the above map is a quasi-isomorphism.
By the quasi-isomorphisms (\ref{eq:1st quasi-iso in proof}) and
(\ref{eq:2nd quasi-iso in proof}), we have the required quasi-isomorphism.
\end{proof}

By the similar argument as in 
\S \ref{subsection: homotopy equivalence of DG-cat and Bar}, 
we have the following proposition.
\begin{proposition}
\label{equivalence for virtual mixed elliptic motives}
The quasi-DG category of comodules over
the bar complex $Bar(\Cal C_{VEM},ev_V)$
is homotopy equivalent to
the quasi-DG category $(VMEM)$ of
virtual mixed elliptic motives.
\end{proposition}
Before giving the outline of the proof, we define the contraction map $con$.

\subsubsection{Definition of the map $con$}
\label{definition of traces}
Let $N_0,N_1,N_2$ be objects in $(VEM)$
and $p$ be an integer.
We assume that $N_2=\oplus_iM_{Y_i}(p)\otimes V_i$.
We introduce a contraction homomorphism $con$
by the composite of
\begin{align*}
&H(N_1,N_2)\otimes
H(N_2,N_3) \\
=&
\oplus_{Y_1(p),Y_2(p)\in TT}
H(N_1,M_{Y_1}(p))\otimes
Hom_{GL_{\infty}}(M_{Y_1}(p),N_2) \\
&\quad\otimes
Hom_{GL_{\infty}}(N_2,M_{Y_2}(p)) \otimes
H(M_{Y_2}(p),N_3) \\
\overset{\tau}\to&
\oplus_{Y(p)\in TT}
H(N_1,M_Y(p))
\otimes
H(M_{Y}(p),N_3)
\end{align*}
Here the map $\tau$ is the map induced by
the multiplication:
\begin{align*}
\tau:
&
Hom_{GL_{\infty}}(M_{Y_1},N_1)
\otimes
Hom_{GL_{\infty}}(N_1,M_{Y_2}) 
  \\
&\to
Hom_{GL_{\infty}}(M_{Y_1},M_{Y_2})\simeq 
\begin{cases}
\bold k &(Y_1=Y_2) \\
0 &(Y_1\neq Y_2)
\end{cases}
\end{align*}
Similarly, we can define homomorphisms $con$:
\begin{align*}
&ev_V(N_0)\otimes
H(N_0,N_1) 
\to
\oplus_{Y(p)\in TT}
ev_V(M_{Y}(p))\otimes
H(M_Y(p),N_1), \\
&
H(N_0,N_1) 
\otimes ev_V(N_1)^*
\to
\oplus_{Y(p)\in TT}
H(N_0,M_Y(p)) \otimes
ev_V(M_{Y}(p))^*.
\end{align*}
By composing the above contraction, we define the following map,
which is also called a contraction:
\begin{align}
\label{composite of trace}
&
H(U_0,U_1)
\otimes 
\cdots 
\otimes 
H(U_{n-1},U_{n})
\\
\nonumber
&\to
\oplus_{Y_0, \dots, Y_{n}\in TT}
Hom_{GL_{\infty}}(U_0,M_{Y_0})
\otimes 
H(M_{Y_0},M_{Y_1})
\otimes 
\cdots \\
\nonumber
& \hskip 0.7in 
\otimes 
H(M_{Y_{n-1}},M_{Y_{n}})
\otimes 
Hom_{GL_{\infty}}(M_{Y_{n}},U_n).
\end{align}

\subsubsection{Outline of the proof of Prop
\ref{equivalence for virtual mixed elliptic motives}
}
\begin{proof}
The proof is similar to Theorem \ref{main theorem genereal}.
We only give the correspondence on objects.
Let $\bold W=(W^i,d_{ji})$ be a DG-complex in $(VEM)$.
We introduce a $Bar(\Cal C_{VEM},ev_V)$-comodule
structure 
$$
\Delta_W:ev_V(W)\to Bar(\Cal C_{VEM},ev_V)\otimes ev_V(W)
$$
on $ev_V(W)=\oplus_i ev_V(W^i)e^{-i}$.
For an index $\alpha=(\alpha_0<\dots <\alpha_n)$
such that $\alpha_0=i,\alpha_n=j$,
the component
$$
\Delta_W^{\alpha}: \ ev_V(W^i)\to Bar(\Cal C,ev_V)^{\alpha} \otimes
ev_V(W^j)
$$
can be described as follows.
By the data $d_{ji}\in
H(W^i,W^j)$ 
of the DG complex $\bold W$,
we define an element
$$
D^{\alpha}=d_{\alpha_0,\alpha_1}\otimes\cdots \otimes
d_{\alpha_{n-1},\alpha_n}\in
H(W^{\alpha_0},W^{\alpha_1})\otimes\dots \otimes
H(W^{\alpha_{n-1}},W^{\alpha_{n}}).
$$
By the contraction map (\ref{composite of trace}),
we have the following map
\begin{align*}
&H(W^{\alpha_0},W^{\alpha_1})\otimes\dots \otimes
H(W^{\alpha_{n-1}},W^{\alpha_n}) \\
\to &
\oplus_{Y_0, \dots, Y_n\in TT}
Hom_{GL_{\infty}}(W^{\alpha_0},M_{Y_0})\otimes
H(M_{Y_0},M_{Y_1})
\otimes\dots \\
&\hskip 0.8in\otimes
H(M_{Y_{n-1}},M_{Y_{n}})\otimes
Hom_{GL_{\infty}}(M_{Y_n},W^{\alpha_n}) \\
\to &
\oplus_{Y_0, \dots, Y_n\in TT}
Hom_{KVect}(ev_V(W^{\alpha_0}),ev_V(M_{Y_0}))\otimes
H(M_{Y_0},M_{Y_1})
\otimes\dots \\
&\hskip 0.8in\otimes
H(M_{Y_{n-1}},M_{Y_{n}})\otimes
Hom_{KVect}(ev_V(M_{Y_n}),ev_V(W^{\alpha_n})).
\end{align*}
The image of $D^{\alpha}$ under the map 
defines the required map 
\begin{align*}
\Delta^{\alpha}_W:\ &ev_V(W^{\alpha_0})\mapsto 
\bigg[ev_V(M_{Y_0})
\overset{\alpha_0}\otimes
H(M_{Y_0},M_{Y_1})
\otimes\dots \\
&\hskip 0.3in
\overset{\alpha_{n-1}}\otimes
H(M_{Y_{n-1}},M_{Y_{n}})
\overset{\alpha_n}\otimes
ev_V(M_{Y_n})^*\bigg]\otimes ev_V(W^{\alpha_n}).
\end{align*}

\end{proof}
\begin{definition}
The quasi-DG category of comodule over 
$Bar(\Cal C_{VEM})$ is called
a quasi-DG category of mixed elliptic motives. 
The homotopy category becomes a triangulated category
and it is called the triangulated category of 
virtual mixed elliptic motives.
\end{definition}
\begin{remark}
In this section, we assume that $E$ is an elliptic curve over
a field $K$. In general we can modify the definition of
relative quasi-DG category and relative quasi-DG algebra
for an elliptic curve $X\to S$ over an arbitrary scheme $S$ 
over a field $k$. Then  
$Bar(\Cal C_{VEM})(E/S)$ and ${VMEM}(E/S)$ 
become a contravariant 
functor and
a fibered quasi-DG category
over non-CM points $(Sch/S)_{non-CM}$ of $(Sch/S)$.
\end{remark}

\subsection{Tensor and antipodal structure on DG category}
\subsubsection{}
Let $\Cal C$ be a DG category. 
\begin{definition}
\begin{enumerate}
\item
A tensor structure $(\Cal C,\otimes)$on $\Cal C$
consists of
\begin{enumerate}
\item
The biadditive correspondence
$$
\bullet \otimes\bullet:ob(\Cal C)\times ob(\Cal C)\to ob(\Cal C)
:(M,N)\mapsto M\otimes N 
$$
on pairs of objects, and
\item
a natural homomorphism of complexes
$$
\underline{Hom}^{\bullet}(M_1, M_2)\otimes
\underline{Hom}^{\bullet}(N_1, N_2)\to
\underline{Hom}^{\bullet}(M_1\otimes N_1, M_2\otimes N_2).
$$
\end{enumerate}
\item
A tensor structure $(\Cal C, \otimes)$ satisfies
the distributive law if the relation
$$
(f_1\otimes f_2)\circ (g_1\otimes g_2)=(-1)^{\deg(f_2)\deg(g_1)}
(f_1 \circ g_1)\otimes (f_2\circ g_2),
$$
is satisfied for
$f_i \in\underline{Hom}^{\bullet}(M_i, N_i)$,
$g_i \in\underline{Hom}^{\bullet}(L_i, M_i)$
for $i=1,2$.
\item
Let $(\Cal C,\otimes)$ be a tensor structure on the
category $\Cal C$.
A system $\{c_{A,B}\}$ of  closed degree zero isomorphisms 
$c_{A,B}:A\otimes B\to B\otimes A$ is called the commutativity
constrain if $c_{A,B}$ and $c_{B,A}$ are inverse to each other
and they satisfies the relation:
\begin{align*}
(f_N\otimes f_M)\circ c_{M_1,N_1}=&
(-1)^{\deg(f_M)\deg(f_N)}
c_{M_2,N_2}\circ (f_M\otimes f_N) \\
&\in
\underline{Hom}(M_1\otimes N_1, N_2\otimes M_2)
\end{align*}
for
$
f_M\in\underline{Hom}(M_1, M_2),
f_N \in
\underline{Hom}(N_1, N_2)
$.
\item
Let $(\Cal C,\otimes)$ be a tensor structure on the
category $\Cal C$.
A system $\{c_{A,B,C}\}$ of  closed degree zero isomorphisms 
$c_{A,B,C}:(A\otimes B)\otimes C\to 
A\otimes (B\otimes C)$ is called the 
associativity constrain
if the following holds:
\begin{align*}
c_{L_2,M_2,N_2}\circ((f_L\otimes f_M)\otimes f_N)=
&(f_L\otimes  (f_M\otimes f_N)) \circ c_{L_1,M_1,N_1}\\
\in &
\underline{Hom}((L_1\otimes M_1)\otimes N_1, L_2\otimes (N_2\otimes M_2))
\end{align*}
for
$
f_L\in\underline{Hom}(L_1, L_2),
f_M\in\underline{Hom}(M_1, M_2),
f_N \in
\underline{Hom}(N_1, N_2).
$
\end{enumerate}
\end{definition}
\subsubsection{}
\label{reason why naive does not distributive}
The following example illustrates the category of naive
elliptic motives does not have a natural tensor structure
with the distributive
property.
On the category of $Rep_{GL(V)}$, we have
$V\otimes V\simeq Sym^2(V)\oplus \bold Q(-1)$.
Let $V_1, \dots, V_6$ be copies of $V$
and choose an isomorphism $V_3\otimes V_4=
Sym^2(V)\oplus \bold Q(-1)$.
We consider the composites of
\begin{align*}
&id_3\otimes id_4\in 
\underline{Hom}^0(V_3\otimes V_4,V_5\otimes V_6),\quad
id_1\otimes id_2\in 
\underline{Hom}^0(V_1\otimes V_2,V_3\otimes V_4)
\end{align*}
and consider a decomposition
\begin{align*}
&id_3\otimes id_4=S_1+A_1,\quad
id_1\otimes id_2=S_2+A_2
\end{align*}
according to the decomposition
\begin{align*}
&id_3\otimes id_4\in 
\underline{Hom}^0(V_3\otimes V_4,V_5\otimes V_6)=
\underline{Hom}^0(Sym^2\oplus \bold Q(-1),V_5\otimes V_6), \\
&id_1\otimes id_2\in 
\underline{Hom}^0(V_1\otimes V_2,V_3\otimes V_4)=
\underline{Hom}^0(V_1\otimes V_2,Sym^2\oplus \bold Q(-1)).
\end{align*}
Let $\Delta_{ij}^+$ and $\Delta_{ij}^-$ be divisors defined by $x_i=x_j$
and $x_i+x_j=0$,
where $x_i$ is the coordinate of the copy $E_i$.
Thus $A_1$ and $A_2$ are 
equal to 
$\frac{1}{2}(\Delta_{56}^++\Delta^{-}_{56})$
and
$\frac{1}{2}(\Delta_{12}^++\Delta^{-}_{12})$
as elements in
$Z^{\bullet}(pt\times (E_5\times E_6),\bullet)$ and
$Z^{\bullet}((E_1\times E_2)\times pt,\bullet)$.
Therefore the $A_1\circ A_2$ composite is equal to
$$
\frac{1}{4}(\Delta_{56}^+-\Delta^{-}_{56})\cap
(\Delta_{12}^+-\Delta^{-}_{12})
$$
On the other hand, the anti-symmetric part $A$ of 
$(id_3\circ id_1)\otimes (id_4\circ id_2)$
is equal to the intersection 
$$
\frac{1}{8}\bigg[(\Delta_{15}^+-\Delta^{-}_{15})
\cap 
(\Delta_{26}^+-\Delta^{-}_{26})
+
(\Delta_{16}^+-\Delta^{-}_{16})
\cap(\Delta_{25}^+-\Delta^{-}_{25})\bigg]
$$
Though $A$ and $A_1\circ A_2$ are homotopy equivalent,
they are different as cycles.
\subsubsection{}
Using the inverse of $S$, we have
a transpose representation $M^t$ on $M^*=Hom_{\bold k}(M,\bold k)$,
which is a left $S$-comodule on
the underlying right $S$-comodule $M^*$.
In general, we introduce an antipodal structure on DG categories.
\begin{definition}
Let $\Cal C$ be a DG category with a tensor structure. A pair of
\begin{enumerate}
\item
contravariant functor 
$\Cal C \to \Cal C:M\to M^t$, and 
\item
a system of degree zero closed isomorphisms
$$
\theta:(M\otimes N)^t \to M^t \otimes N^t
$$ 
\end{enumerate}
is called an antipodal structure on $\Cal C$ if the following
diagram commutes.
$$
\begin{matrix}
Hom_{\Cal C}(M_1, N_1)\otimes Hom_{\Cal C}(M_2, N_2)
&\xrightarrow{\otimes}&
Hom_{\Cal C}(M_1\otimes M_2, N_1\otimes N_2) \\
\downarrow & &
\downarrow
\\
Hom_{\Cal C}(N_1^t, M_1^t)\otimes Hom_{\Cal C}(N_2^t, M_2^t)
&\xrightarrow{\otimes}&
Hom_{\Cal C}(N_1^t\otimes N_2^t, M_1^t\otimes M_2^t) \\
& &\Vert\quad ad(\theta)
\\
& &
Hom_{\Cal C}((N_1\otimes N_2)^t, (M_1\otimes M_2)^t).
\end{matrix}
$$
\end{definition}

\subsubsection{}
\label{subsubsection: primitive difinition of shuffle product}
Assume that $\Cal C$ has a tensor structure $(\Cal C,\otimes)$.
We extend a tensor structure $(K\Cal C,\otimes)$ on the category 
$K\Cal C$ as follows.
Let $(V^i,\{d_{ij}\})$ and $(W^i,\{e_{ij}\})$ be DG complexes in
$\Cal C$.
Then $V^i\otimes W^j$ is an object in $\Cal C$.
The relative DG complex structure 
$(\oplus_{i+j=k} (V^i\otimes W^j), f_{kl})$
is given by
$$
f_{kl}=\sum_{i+j=k,p+j=l}\tau(d_{ip}\otimes 1_{W^j})
+\sum_{i+j=k,i+q=l}\tau(1_{V^i}\otimes e_{jq}).
$$
We can check the following lemma.
\begin{lemma}
If the tensor structure $\otimes$ on $\Cal C$ is distributive (resp. commutative),
the induced tensor structure on $K\Cal C$ is also distributive (resp. commutative).
\end{lemma}
\subsubsection{}
\label{subsubsection: primitive difinition of shuffle product
for bar}
Let $\Cal C$ be a small DG-category and $ev:\Cal C\to Vect_{\bold k}$ be
an augmentation.
This tensor structure on $K\Cal C$ gives rise to
the shuffle product on the relative bar complex 
$Bar=Bar(\Cal C,ev)$ via the equivalence of
categories as follows. 
Let $\bold {Bar}$ be the object in $K\Cal C$
corresponding to the left $Bar$ comodule $Bar$.
By the tensor structure, the object $\bold {Bar}\otimes \bold {Bar}$
in $K\Cal C$ is defined.
Therefore we have
the corresponding object $Bar\otimes Bar$ in $Bar$-comodule.
Let
$$
\Delta_{Bar\otimes Bar}:Bar\otimes Bar \to Bar\otimes
(Bar \otimes Bar)
$$
be the $Bar$-comodule structure on $Bar\otimes Bar$.
By composing the counit $u\otimes u:Bar\otimes Bar \to \bold k$,
we have a shuffle product
$$
Bar\otimes Bar \to Bar.
$$
One can check the following properties.
\begin{proposition}
If there is a commutative constrain (resp. associativity constrain),
the shuffle product on $Bar$ is commutative (resp. associative).
Moreover, an antipodal structure on $\Cal C$ gives an antipodal of
$Bar$.
\end{proposition}
 
\subsection{Tensor product for virtual mixed elliptic motif}

We define a tensor structure on $(VEM)$ in this subsection.
Let $M_{Y_1}(p_1),M_{Y_2}(p_2),M_{Y_3}(p_3)$ and 
$M_{Y_4}(p_4)$ be objects of $(VEM)$.
We define a homomorphism
\begin{align}
\label{tensor map for virtual elliptic motif}
& \underline{Hom}(M_{Y_1}(p_1),M_{Y_3}(p_3))
\otimes
\underline{Hom}(M_{Y_2}(p_2),M_{Y_4}(p_4)) \\
\nonumber
& =e_{Y_1}
\Cal H^{\bullet}(E^{s(Y_1)}, E^{s(Y_3)},p_3-p_1)e_{Y_3}
\otimes
e_{Y_2}
\Cal H^{\bullet}(E^{s(Y_2)}, E^{s(Y_4)},p_4-p_2)e_{Y_4}
\\
\nonumber
& \to
\underline{Hom}(M_{Y_1}(p_1)\otimes M_{Y_2}(p_2)
 ,M_{Y_3}(p_3)\otimes M_{Y_4}(p_4)) \\
\nonumber
& =
\oplus_{Z_1,Z_2}
Hom_{GL_{\infty}}(M_{Y_1}\otimes M_{Y_2},M_{Z_1})\\
\nonumber
&\otimes
\underline{Hom}(M_{Z_1}(p_1+p_2),M_{Z_2}(p_3+p_4)) \\
\nonumber
&\otimes
Hom_{GL_{\infty}}(M_{Z_2},M_{Y_3}\otimes M_{Y_4}) \\
\nonumber
& =
\oplus_{Z_1,Z_2}
Hom_{GL_{\infty}}(M_{Y_1}\otimes M_{Y_2},M_{Z_1})\\
\nonumber
&\otimes
e_{Z_1}
\Cal H^{\bullet}(E^{s(Z_1)}, E^{s(Z_2)},p_3+p_4-p_1-p_2)e_{Z_2} 
\\
\nonumber
&\otimes
Hom_{GL_{\infty}}(M_{Z_2},M_{Y_3}\otimes M_{Y_4}).
\end{align}
We choose bases $\{\varphi_{Z_1,i}\}$ and $\{\psi_{Z_2,j}\}$ 
of 
$$
Hom_{GL_{\infty}}(M_{Z_1}, M_{Y_1}\otimes M_{Z_2}) \text{ and }
Hom_{GL_{\infty}}(M_{Z_2} ,M_{Y_3}\otimes M_{Y_4}).
$$
Their dual bases in
$$
Hom_{GL_{\infty}}(M_{Y_1}\otimes M_{Z_2}, M_{Z_1}) \text{ and }
Hom_{GL_{\infty}}(M_{Y_3}\otimes M_{Y_4}, M_{Z_2}).
$$
under the composite paring are written as 
$\{\varphi_{Z_1,i}^*\}$ and $\{\psi_{Z_2,j}^*\}$.
Let $(i,j,k)=(1,1,3)$ or $(2,2,4)$. For
\begin{align*}
f_j=e_{Y_i}\Lambda(s_i)\Cal Z_j\Lambda(s_k)e_{Y_k} \in
e_{Y_i}
\Cal H^{\bullet}(E^{s(Y_i)}, E^{s(Y_k)},p_k-p_i)e_{Y_k},
\end{align*}
we set
$$
f_1\times f_2=(-1)^{\# s_3(\#s_2+\deg(\Cal Z_2))+\#s_2\deg(\Cal Z_1)}
\Lambda(s_1)\Lambda(s_3)
(\Cal Z_1\times
\Cal Z_2)
\Lambda(s_2)\Lambda(s_4).
$$
The map (\ref{tensor map for virtual elliptic motif})
is defined by
\begin{align*}
&e_{Y_1}f_1
e_{Y_3}
\otimes
e_{Y_2}f_2e_{Y_4}
\mapsto 
\sum_{Z_1,Z_2,i,j}
\varphi_{Z_1,i}^*\otimes 
\big(\varphi_{Z_1,i}
(f_1\times f_2)
\psi_{Z_2,j}^*\big)
\otimes \psi_{Z_2,j}.
\end{align*}
for 
Here $s_i=s(Y_i)$ and
$\Cal Z_1 \times \Cal Z_2$ is the product of algebraic cycles
in $E^{(s(Y_1)\coprod \cdots \coprod s(Y_4))}\times
\bold A^{\bullet}$ and $\varphi_{Z_1,i}$ and $\psi_{Z_2,j}^*$
acts on the space 
$\Cal H^{\bullet}(E^{(s(Y_1)\coprod s(Y_3))},
E^{(s(Y_2)\coprod s(Y_4))},\bullet)$
from the left and the right via the 
Schur-Weyl reciprocity (\ref{schur-weyl duality}).

\begin{proposition}[Distributive relation]
\label{associativity for each component}
This tensor structure satisfies distributive law.
\end{proposition}
To prove the above proposition, we use the following easy lemma.
\begin{lemma}
\label{adjoint of GL action}
Let $S$ be a finite set and
$$
(*,*)_S:Z^i(E^a\times E^S,p)\otimes Z^j(E^S\times E^b,q) \to Z^{i+j}(E^a \times
E^b,p+q)
$$
be the composite with the push forward 
$E^a\times \Delta_{E^S}\times E^b \to E^a\times E^b$
and the intersection with $E^a\times \Delta_{E^S}\times E^b$.
Let $\psi \in \bold Q[Isom(S',S)]$.
The left and right action of $\psi$ defines a homomorphisms
\begin{align*}
& Z^i(E^{S'}\times E^{b},q) \to Z^i(E^{S}\times E^{b},q), \\
& Z^i(E^a\times E^S,p) \to Z^i(E^a\times E^{S'},p).
\end{align*}
Then we have
$
(g\psi, f)_{S'}=(g,\psi f )_S
$
for $f \in Z^i(E^{S'}\times E^{b},q)$ and
$g\in Z^i(E^a\times E^S,p)$.
\end{lemma}
\begin{proof}
The lemma follows from the identity:
\begin{align*}
& \{(x,w) \in E^a\times E^b \mid (x,\sigma(y))\in \sigma(f) , (z,w) \in
 g, \sigma(y)=z\} \\
=& \{(x,w) \in E^a\times E^b \mid (x,y)\in f , (\sigma^{-1}(z),w) \in
 \sigma^{-1}(g), y=\sigma^{-1}(z)\},
\end{align*}
for $\sigma\in Isom(S,S')$.
\end{proof}
\begin{proof}
[Proof of Proposition \ref{associativity for each component}]
Let
$f_i \in\underline{Hom}^{\bullet}(L_i, M_i)$,
$g_i \in\underline{Hom}^{\bullet}(M_i, N_i)$
for $i=1,2$.
We check the distributive law for the multiplication map $\mu$.
We fix a bases 
$\{\varphi_{Z_1,i}\}$,
$\{\phi_{Z_2,j}\}$ and
$\{\psi_{Z_3,k}\}$
of
$ Hom_{GL_{\infty}}(M_{Z_1}, L_{1}\otimes L_{2})$,
$Hom_{GL_{\infty}}(M_{Z_2}, M_{1}\otimes M_{2})$
and $Hom_{GL_{\infty}}(M_{Z_3}, N_{1}\otimes N_{2})$.
The dual bases are written by
$\{\varphi_{Z_1,i}^*\}$,
$\{\phi_{Z_2,j}^*\}$ and
$\{\psi_{Z_3,k}^*\}$.
We use the same notations for $f_1,f_2$ etc. as before. Then
\begin{align*}
&\mu((f_1\otimes f_2)\otimes (g_1\otimes g_2))  \\
= &\sum_{Z_1,Z_2,Z'_2,Z_3,i,j,j',k}
\mu\bigg(\big(\varphi_{Z_1,i}^* \otimes 
(\varphi_{Z_1,i}(f_1\times f_2) \phi_{Z_2,j}^*)\otimes \phi_{Z_2,j}\big) \\
&
\hskip 0.8in
\otimes \big(\phi_{Z_2',j'}^* \otimes 
(\phi_{Z_2',j'}(g_1\times g_2) \psi_{Z_3,k}^*)\otimes \psi_{Z_3,k}\big)\bigg) \\
=&
\sum_{Z_1,Z_2,Z_3,i,j,k}
\varphi_{Z_1,i}^* \otimes 
\mu\big(
(\varphi_{Z_1,i}(f_1\times f_2) \phi_{Z_2,j}^*)
\otimes 
(\phi_{Z_2,j}(g_1\times g_2) \psi_{Z_3,k}^*)\big)\otimes \psi_{Z_3,k} \\
\overset{(\ref{adjoint of GL action})}
=&
\sum_{Z_1,Z_2,Z_3,i,j,k}
\varphi_{Z_1,i}^* \otimes 
\mu\big(
(\varphi_{Z_1,i}(f_1\times f_2))\otimes \big( \phi_{Z_2,j}^*
\phi_{Z_2,j}(g_1\times g_2) \psi_{Z_3,k}^*)\big)\otimes \psi_{Z_3,k} \\
=&
\sum_{Z_1,Z_3,i,k}
\varphi_{Z_1,i}^* \otimes 
(\varphi_{Z_1,i}
\mu\big(
(f_1\times f_2)\otimes 
(g_1\times g_2)\big) \psi_{Z_3,k}^*)\otimes \psi_{Z_3,k} \\
=&
(-1)^{\deg(f_2)\deg(g_1)}
\sum_{Z_1,Z_3,i,k}
\varphi_{Z_1,i}^* \otimes 
(\varphi_{Z_1,i}
\bigg[\mu(f_1\otimes g_1)\times 
\mu(f_2\otimes g_2)\bigg] \psi_{Z_3,k}^*)\otimes \psi_{Z_3,k} \\
=&
(-1)^{\deg(f_2)\deg(g_1)}
\mu(f_1\otimes g_1)\otimes
\mu(f_2\otimes g_2).
\end{align*}
\end{proof}

\subsubsection{Commutativity constraint and associativity constraint}
We define a system of closed homomorphisms of degree
zero $c_{M,N}$ and $c_{L,M,N}$
for $L,M,N\in (VEM)$ by
\begin{align*}
c_{M,N}=\sum_{Y,i}\alpha_i\otimes \Delta_{Y}\otimes
\alpha_i^{t*} 
&\in 
\underline{Hom}_{VEM}(M\otimes N,N\otimes M) \\
=& \oplus_Y
Hom_{GL_{\infty}}(M \otimes N,M_{Y}) \\
&\otimes
\underline{Hom}_{VEM}(M_{Y},M_Y) 
\otimes
Hom_{GL_{\infty}}(M_{Y},N \otimes M) \\
c_{L,M,N}=\sum_{Y,i}\beta_i\otimes \Delta_{Y}\otimes
\beta_i^{a*} 
&\in 
\underline{Hom}_{VEM}((L\otimes M)\otimes N,
L\otimes (M\otimes N)) \\
=& \oplus_Y
Hom_{GL_{\infty}}((L\otimes M) \otimes N,M_{Y}) \\
&\otimes
\underline{Hom}_{VEM}(M_{Y},M_Y) 
\otimes
Hom_{GL_{\infty}}(M_{Y},L\otimes (M \otimes N))
\end{align*}
where $\{\alpha_i\}$ (resp. $\{\beta_i\}$)
and $\{\alpha^{t*}\}$ (resp. $\{\beta_i^{a*}\}$)
are dual bases under the contraction pairing induced by the
natural isomorphism
$$
M\otimes N\simeq N\otimes M,\text{ (resp.}
(L\otimes M)\otimes N \simeq L\otimes (M\otimes N)
\text{ )}
$$
in $GL_{\infty}$.
The following proposition is direct from the definitions.
\begin{proposition}
\label{proposition:assiciativity and commutativity for virtual}
The systems of the above maps $c_{M,N}$ and $c_{L,M,N}$ 
satisfy the commutativity axiom and
the associativity axiom, respectively.
\end{proposition}
\subsubsection{Antipodal for $(VEM)$}
We define a self-contravariant functor $a:(VEM)\to (VEM)$
as follows. Let $Y$ be a Young tableaux.
We define 
$$
(M_Y(p))^t=e_{Y}^*(V^{\otimes s(Y)})(\# s(Y) -p),
$$
where $e_Y^*$ is the adjoint of $e_Y$ defined in 
Definition \ref{definition of adjoint for group ring}.
We have the following isomorphism of complexes:
\begin{align*}
&\Cal H^{\bullet}(E^A,E^B,k) \\
=&
\Lambda^*(A)\otimes
Z_{-}^{a+k}(E^A\times E^B,\bullet)\otimes \Lambda(B)[-2k]\\
\xrightarrow{\tau}&
\Lambda^*(B)\otimes
Z_{-}^{b+(k-b+a)}(E^A\times E^B,\bullet)\otimes \Lambda(A)[-2k+2b-2a]\\
=&\Cal H^{\bullet}(E^B,E^A,k-b+a),
\end{align*}
where $a=\# A, b=\# B$.
Here $\tau$ is defined by
\begin{align*}
&\tau(
f_a\wedge \cdots \wedge f_1\otimes z\otimes 
e_1\wedge \cdots e_b\cdot e^{-2k}) \\
=&(-1)^{(a+b)\deg(z)+ab}
f_b\wedge \cdots \wedge f_1\otimes z\otimes 
e_1\wedge \cdots e_a\cdot e^{-2k+2b-2a}.
\end{align*}
Using the map $\tau$, the functor $a$ for homomorphisms is defined by
the following composite map:
\begin{align*}
&\underline{Hom}_{VEM}((M_{Y_1}(p_1))^t, (M_{Y_2}(p_2))^t) \\
=&e_{Y_1}^*\underline{Hom}_{VEM}(V^{\otimes S_1}(s_1-p_1), 
V^{\otimes S_2}(s_2-p_2))e_{Y_2}^* \\
=&e_{Y_1}^*\Cal H^{\bullet}(E^{S_1},E^{S_2},
s_2-s_1+p_1-p_2)e_{Y_2}^* \\
\xrightarrow{\tau}&e_{Y_2}\Cal H^{\bullet}(E^{S_2},E^{S_1},
p_1-p_2)e_{Y_1} \\
\simeq
&\underline{Hom}_{VEM}(M_{Y_2}(p_2), M_{Y_1}(p_1)).
\end{align*}
Here $S_1=s(Y_1)$ and $S_2=s(Y_2)$. We can check the following
proposition.
\begin{proposition}
\label{antipodal for virtual elliptic motives}
The map $a$ defines a contravariant functor from $(VEM)$
to $(VEM)$. Moreover $a$ defines an antipodal structure.
\end{proposition}
Thus we have the following theorem.
\begin{theorem}
\label{main theorem for virtual mixed elliptic motives}
\begin{enumerate}
\item
The natural DG functors
$$
(EM)\to (VEM),\quad (MEM)\to (VMEM)
$$
are weak homotopy equivalent.
\item
The quasi-DG category $(VMEM)$ has distributive, commutative and associative
tensor structure with an antipodal.
\end{enumerate}
\end{theorem}
\subsection{Shuffle product of virtual bar complex $Bar(VEM)$}
In this subsection, we give an explicit description of the shuffle product
of $Bar(VEM)$ introduced in 
\S \ref{subsubsection: primitive difinition of shuffle product}.
Let 
\begin{align*}
\sha(m,n)=\{(p^{(0)},\dots, p^{(m+n)})\in (\bold N^2)^{m+n+1}\mid
& p^{(0)}=(0,0),p^{(m+n)}=(m,n),\\
& p^{(i+1)}-p^{(i)}\in\{(0,1),(0,1)\}
\}
\end{align*}
be the set of shuffles.
Let $\sigma=(p^{(0)},\dots, p^{(m+n)})$ be an element in $\sha$. 
For $V_0,\dots, V_m,W_0,\dots, W_n\in TT$,
we define $U^{\sigma}_i=V_{a_i}\otimes W_{b_i}$,
where $p^{(k)}=(a_k,b_k)$.
For $\alpha=(\alpha_0<\dots <\alpha_m)$ and 
$\beta=(\beta_0<\dots <\beta_n)$, we define
$\gamma^{\sigma}=(\gamma^{\sigma}_0<\dots <\gamma^{\sigma}_{m+n})$ by
$\gamma^{\sigma}_k=\alpha_{a_k}+\beta_{b_k}$.
The complex of homomorphisms $\underline{Hom}_{VEM}(A,B)$
is denoted as $H(A,B)$.
Let $\varphi_{i,i+1}\in H(V_{i},V_{i+1})$
and $\psi_{j,j+1}\in H(W_{j},W_{j+1})$.
We define $\tau^{\sigma}_{k,k+1}\in
H(U^{\sigma}_{k},U^{\sigma}_{k+1})$
by
\begin{align*}
\tau^{\sigma}_{k,k+1}=
\begin{cases}
\varphi_{a_k,a_{k+1}}\otimes 1 &\text{ if } p^{(k+1)}-p^{(k)}=(1,0)\\
1\otimes \psi_{b_k,b_{k+1}} &\text{ if } p^{(k+1)}-p^{(k)}=(0,1).
\end{cases}
\end{align*}
For elements 
\begin{align*}
v=&v_0\overset{\alpha_0}\otimes \varphi_{01}\overset{\alpha_1}\otimes 
\cdots \overset{\alpha_{n-1}}\otimes \varphi_{n-1,n}
\overset{\alpha_n}\otimes v_n^*  \\
\in\ 
&ev(V_0)\overset{\alpha_0}\otimes 
H(V_0,V_1)
\overset{\alpha_1}\otimes 
\cdots \overset{\alpha_{n-1}}
\otimes 
H(V_{n-1},V_{n})
\overset{\alpha_n}\otimes ev(V_n)^*,\\
w=&w_0\overset{\beta_0}\otimes \psi_{01}\overset{\beta_1}\otimes 
\cdots \overset{\beta_{m-1}}\otimes \psi_{m-1,m}
\overset{\beta_m}\otimes w_m^*  \\
\in\ 
&ev(W_0)\overset{\beta_0}\otimes 
H(W_0,W_1)
\overset{\beta_1}\otimes 
\cdots \overset{\beta_{m-1}}
\otimes 
H(W_{m-1},W_{m})
\overset{\beta_m}\otimes ev(W_m)^*,
\end{align*}
we define $(v\widetilde{\sha} w)^{\sigma}$ by
\begin{align*}
(v\widetilde{\sha} w)^{\sigma}
=&(v_0\otimes w_0)
\overset{\gamma^{\sigma}_0}\otimes \tau^{\sigma}_{01}
\overset{\gamma^{\sigma}_1}\otimes 
\cdots \overset{\gamma_{m+n-1}^{\sigma}}
\otimes \tau^{\sigma}_{m+n-1,m+n}
\overset{\gamma_{m+n}^{\sigma}}\otimes (v_n^*\otimes w_m^*)  \\
\in\ 
&ev(U^{\sigma}_0)
\overset{\gamma_0^{\sigma}}\otimes 
H(U_0^{\sigma},U_1^{\sigma})
\overset{\gamma_1^{\sigma}}\otimes 
\cdots \overset{\gamma^{\sigma}_{m-1}}
\otimes 
H(U_{m+n-1}^{\sigma},U_{m+n}^{\sigma})
\overset{\gamma_m^{\sigma}}\otimes ev(U_{m+n}^{\sigma})^*. 
\end{align*}

\subsubsection{}
Using the maps $con$ defined in \ref{definition of traces},
we have the following map:
\begin{align*}
&ev(U^{\sigma}_0)^*\overset{\gamma_0^{\sigma}}\otimes 
H(U_1^{\sigma},U_0^{\sigma})
\overset{\gamma_1^{\sigma}}\otimes 
\cdots \overset{\gamma_{m-1}^{\sigma}}
\otimes 
H(U_{m+n}^{\sigma},U_{m+n-1}^{\sigma})
\overset{\gamma_m^{\sigma}}\otimes ev(U_{m+n}^{\sigma}) \\
&\to
\oplus_{Y_0, \dots, Y_{m+n}\in TT}
ev(M_{Y_0})^*\overset{\gamma_0^{\sigma}}\otimes 
H(M_{Y_1},M_{Y_0})
\overset{\gamma_1^{\sigma}}\otimes 
\cdots \\
& \hskip 0.7in \overset{\gamma_{m-1}^{\sigma}}
\otimes 
H(M_{Y_{m+n}},M_{Y_{m+n-1}})
\overset{\gamma_m^{\sigma}}\otimes ev(M_{Y_{m+n}}).
\end{align*}
The image of $(v\widetilde{\sha} w)^{\sigma}$
is denoted as $(v\sha w)^{\sigma}$.
The shuffle product of the bar complex is defined similarly.

\begin{example}
If
$\sigma=\{(0,0),(1,0),(1,1),(2,1)\}$, then
$U^{\sigma}_{i}$ and
$\tau_{j,j+1}^{\sigma}$ look as follows:
$$
\begin{matrix}
& & V_1\otimes W_1 &\xrightarrow{\varphi_{12}\otimes 1} & V_2\otimes W_1 \\
& & \hskip 0.2in \uparrow \text{\tiny $1\otimes \psi_{01}$} & \\
V_0\otimes W_0 & 
\xrightarrow{\varphi_{01}\otimes 1}
& V_1\otimes W_0.
\end{matrix}
$$
For bar complex, we have
and
\begin{align*}
&\big([v_0\mid \varphi_{01}\mid \varphi_{12}\mid v_2^*]
\widetilde{\sha}[w_0\mid \tau_{01}\mid w_1^*]\big)^{\sigma} \\
=&con
\big((v_0\otimes w_0)\otimes
(\varphi_{01}\otimes 1)\otimes(1\otimes \psi_{01})
\otimes (\varphi_{12}\otimes 1)\otimes (v_2^*\otimes w_2^*)\big)
\end{align*}
\end{example}
\begin{definition}
We define the shuffle product $v\sha w$ of
$v, w \in Bar(\Cal C_{VEM},ev_V)$
by
\begin{equation}
\label{explicit shuffle formula}
v\sha w=\sum_{\sigma \in \sha(m,n)}(v\sha w)^{\sigma}.
\end{equation}
We can check the following proposition by definition.
\end{definition}
\begin{proposition}
The shuffle product (\ref{explicit shuffle formula})
coincides with that defined in
\S \ref{subsubsection: primitive difinition of shuffle product
for bar}.
\end{proposition}

By Proposition \ref{proposition:assiciativity and 
commutativity for virtual},
\ref{antipodal for virtual elliptic motives},
we have the following theorem.
\begin{theorem}
\label{shuffle product commutative associative}
The shuffle product on $Bar(\Cal C_{VEM},ev_V)$
is commutative and associative.
Therefore $Bar(\Cal C_{VEM},ev_V)$ is a differential graded
quasi-Hopf algebra and
$H^0(Bar(\Cal C_{VEM},ev_V))=H^0(Bar(A_{EM},ev_V))$
is a Hopf algebra. Therefore $Spec(H^0(Bar(A_{EM},ev_V))$
is a pro-algebraic group
\end{theorem}
\begin{definition}[Algebraic group $\Cal G_{MEM}$]
\label{defintion algebraic group}
\begin{enumerate}
\item
We define the pro-algebraic group $\Cal G_{MEM}=
Spec(H^0(Bar(A_{EM},ev_V)))$
\item
We define the Tannakian category $\Cal A_{MEM}$
of mixed elliptic motives as the category of algebraic representations
of $\Cal G_{MEM}$.
It is equivalent to the category of comodules over
$Spec(H^0(Bar(A_{EM},ev_V)))$.
\end{enumerate}
\end{definition}

\section{Mixed elliptic motif associated to elliptic polylogarithm}
In the paper of [BL], they introduced an elliptic polylog motif $Pl_n$.
In this section, we construct an object $Pl_n$ in $(MEM)$ concentrated at
degree zero. Therefore the corresponding object defines a
comodule over $H^0(Bar_{MEM})$. We write down the explicit comodule structure
of $Pl_n$ in \S \ref{explicit comodule structure of polylog}.

\subsection{Simplicial, Cubical and Cubical-simplicial log complexes}
In this section, we consider three double complexes $\bold K_S^i$, 
$\bold K_C^i$, and $\bold K_{CS}^i$.
Let $E$ be an elliptic curve over a field $K$ and 
$E_B$ the constant family of $E$ over the base scheme $B=E$.
Let $s$ be the section of
$E_B$ defined by the diagonal map $B\to E_B$.
The map $\epsilon:x \mapsto s-x$ defines an action of 
$\bold Z/2\bold Z$ on $E_B$.

\subsubsection{}

By using localization sequences for higher Chow group \cite{B2}, 
we have the following proposition.
\begin{proposition}[Levin \cite{L} Proposition 1.1]
\label{symmetric part vanishing}
Let $G$ be a finite group and $\chi:\bold G \to \{\pm 1\}$
be a character of $G$. Let $X$ be a variety with
an action $\rho:G \to Aut(X)$ of $G$.
Let $Z_i\subset X$ ($i=1, \dots, k$) be a closed subset and
$$
G_i=\{g\in G\mid \rho(x)=x \text{ for all }x\in Z_i\}
$$ 
be the stabilizer of points in $Z_i$.
Assume that $\chi(G_i)=\{\pm 1\}$. Then we have
\begin{equation}
\label{gp fixed char part}
CH^{p}(X-\cup_i Z_i, q)_{\chi}=
CH^{p}(X, q)_{\chi}
\end{equation}
\end{proposition}

\subsubsection{Simplicial log complex $\bold K_{S,n}$}
We set $\Cal I_{S,p}=\{I \subsetneq [1,n+1]\mid \#I = p\}$ and
$G_S=\frak S_{n+1}$. We set 
$$(E^{n+1}_B)_s
=\{p=(p_1, \dots, p_{n+1})\in
E_B^{n+1}\mid \sum_{i=1}^{n+1}p_i=s\}\subset E_B^{n+1}
$$ 
and
define divisors $D_{S,i}=\{p\in (E_B^{n+1})_s\mid p_i=0\}$ 
for $i=1,\dots, n+1$ and $D_{S,I}=\cap_{i\in I}D_i$.
The complex $\bold K_{S,n}$ is defined by
\begin{align*}
\cdots \to
\oplus_{I \in \Cal I_{S,2}}\Cal Z^{n-1}(E\times D_{S,I},\bullet)
&\to 
\oplus_{I \in \Cal I_{S,1}}\Cal Z^{n}(E\times D_{S,I},\bullet) \\
&\to 
\Cal Z^{n+1}(E\times (E_B^{n+1})_s,\bullet) \to 0.
\end{align*}
The group $G_S$ acts naturally on the complex $\bold K_{S,n}$.
We define an open set 
$(U_B^{n+1})_s$ of $(E_b^{n+1})_s$
by 
$$
(U_B^{n+1})_s=(E_B^{n+1})_s-\cup_{i=1}^{n+1}D_{S,i}.
$$
By the localization sequence \cite{B}, we have
$$
H^j(\bold K_{S,n})\simeq CH^{n+1}(E\times (U_B^{n+1})_s,n+1-j).
$$
and $H^{j}(\bold K_{S,n,sgn})\simeq 
CH^{n+1}(U^{n+1},n+1-j)_{sgn}$, where
$\bold K_{S,n,sgn}$ is
the alternating part of the complex $\bold K_{S,n}$
for the action of $\frak S_{n+1}$.

\subsubsection{Cubical log complex $\bold K_{C,n}$}
We define an index set 
$$
\Cal I_{C,p}=\{(I,\varphi)\mid I\subset [1,n], 
\#I = p,\varphi:I \to \{0,s\}\},
$$
and a group 
\begin{equation}
\label{definition of group alt on cubic}
G_{C,n}=G_C=\frak S_n \ltimes (\bold Z/2\bold Z)^n=\{(\tau, \sigma_1, \dots,
\sigma_n)\}.
\end{equation}
Let $\psi$ be the non-trivial character of $\bold Z/2\bold Z$,
and $\chi$ be a character of $G_C$ defined by
\begin{equation}
\label{definition of cubic alt character}
\chi_n(\tau,\sigma_1, \dots, \sigma_n)=\chi(\tau,\sigma_1, \dots, \sigma_n)=sgn(\tau)
\psi(\sigma_1)\cdots \psi(\sigma_{n}).
\end{equation}

We define a family of subvarieties $\{D_{C,J}\}$ of $E^n_B$ 
of codimension $p$ 
indexed by $J\in \Cal I_{C,p}$ by
$$
D_{C,J}=\{(x_1, \dots, x_n)\in E_B^n \mid
x_i = \varphi(i) \text{ for all } i \in I \},
$$
where $J=(I,\varphi)$.
The complex $\bold K_{C,n}$ is defined by
\begin{align*}
\cdots \to
\oplus_{J \in \Cal I_{C,2}}Z^{n-1}(E\times D_{C,J},\bullet)
&\to 
\oplus_{J \in \Cal I_{C,1}}Z^{n}(E\times D_{C,J},\bullet) \\
&\to 
Z^{n+1}(E\times E_B^n,\bullet) \to 0.
\end{align*}
The group $(\bold Z/2\bold Z)^n$ acts on
the complex $\bold K_{C,n}$ by the action
$$
(\sigma_1,\dots, \sigma_n)(x_1,\dots, x_n)
=(\sigma_1(x_1), \dots, \sigma_n(x_n)),
$$
where
$$
\sigma_i(x_i)= 
\begin{cases}
x_i & \text{ if }\sigma=id \\
s-x_i & \text{ if }\sigma\neq id \\
\end{cases}
$$
using the
action $\epsilon:x \mapsto s-x$. 
The $\chi$ part of the complex is denoted by $\bold K_{C,n,\chi}$.

\subsubsection{Cubical-simplicial log complex $\bold K_{CS,n}$}
We use the same notations for the group $G_C$ and
the action of $G_C$ on $E^n_B$ and the set of indices
$\Cal I_{C,p}$.
For $1\leq i\neq j\leq n$, we define
a divisor $D_{ij}^{\pm}$ of $E_B^n$ by
\begin{align*}
D_{ij}^+=\{x_i=x_j\},\quad
D_{ij}^-=\{x_i=\sigma_j(x_j)\}.
\end{align*}
The group $G_C$ acts on the configuration 
$\{D_{ij}^{\pm}\}_{0\leq i<j\leq n}$.
We set $\Cal D=\cup_{1\leq i < j\leq n}D_{ij}^{\pm}$ and
$(E^n)^0=E^n-\Cal D$.
We define a family of subvarieties $\{D_{C,J}^0\}$ in
$E_B^n$ of codimension $p$ 
($J\in \Cal I_{C,p}$) by
$D_{C,J}^0=D_{C,J}\cap (E_B^n)^0$.
The complex $\bold K_{CS,n}$ is defined by
\begin{align*}
\cdots \to
\oplus_{J \in \Cal I_{C,2}}Z^{n-1}(E\times D_{C,J}^0,\bullet)
&\to 
\oplus_{J \in \Cal I_{C,1}}Z^{n}(E\times D_{C,J}^0,\bullet) \\
&\to 
Z^{n+1}(E\times (E_B^n)^0,\bullet) \to 0.
\end{align*}
The $\chi$-part of the complex is 
denoted by $\bold K_{CS,n,\chi}$.

\begin{proposition}
\label{prop braid config sgn}
Let $x\in E^n$ be a point in $\Cal D$.
The the restriction of $\chi$ to the stabilizer $Stab_x(G_\bold C)$
of $x$ in $G_C$ is non-trivial.
\end{proposition}
\begin{proof}
If $x$ is contained in $D_{ij}^+$, then the permutation
$(ij)$ fixes $x$ and $\chi((ij),1)=-1$. If $x$ is contained in $D_{ij}^-$,
then the element $((ij), \delta)$
fixes the point $x$ and $\chi((ij),\delta)=-1$, where 
$
\delta=(0, \dots,\overset{i}1,\dots, \overset{j}1,\dots, 0).
$
\end{proof}
\begin{corollary}
The homomorphism of complexes 
$\bold K_{C,n,\chi} \to \bold K_{CS,n,\chi}$ is a quasi-isomorphism.
\end{corollary}
\begin{proof}
By the argument in Proposition \ref{symmetric part vanishing} and
Proposition \ref{prop braid config sgn},
the complex $\bold K_{CS,n,\chi}$ 
is quasi-isomorphic to the $\chi$-part of
the complex
\begin{align*}
\cdots \to
\oplus_{J \in \Cal I_{C,2}}Z^{n-1}(E\times D_{C,J},\bullet)
&\to 
\oplus_{J \in \Cal I_{C,1}}Z^{n}(E\times D_{C,J},\bullet) \\
&\to 
Z^{n+1}(E\times E_B^n,\bullet) \to 0.
\end{align*}
Therefore the natural restriction map
$\bold K_{C,n,\chi}\to \bold K_{CS,n,\chi}$
is a quasi-isomorphism.
\end{proof}

\subsection{Linear correspondences and their translations}
For an element $s\in E^n(K)$, the algebraic correspondence associated
to the translation by $s$ is denoted by $T_s\in CH^n(E^n\times E^n)$.
Let $F/K$ be a Galois extension of $K$,
$N$ a positive integer.  
Assume that all points $\tilde {s}$ such that
$N\tilde{s}=s$ are defined over $F$.
Then
$$
T_{s,N}=\frac{1}{N^{2n}}\sum_{N\tilde s=s}T_{\tilde s}
$$
defines a correspondence in $CH^n(E^n\times E^n)$.
Note that the coefficient of Chow groups are $\bold Q$.
We have a natural inclusion:
$$
CH^i(E^n\times E^n,j)\to CH^i(X_F^n\times_F E_F^n,j).
$$
\begin{lemma}
\begin{enumerate}
\item
Under the above notations,
the image of $T_{s,N}\in CH^n(E^n\times E^n)$ in 
$CH^n(E_F^n\times E_F^n)$
is equal to $T_{\tilde s}$.
\item
Let $a_1,\dots, a_p,b_1,\dots, b_p$ be points in $E^n(K)$.
Assume that $\sum_ia_i=\sum_jb_j$. Then
$\sum_i T_{a_i}$ is equal to $\sum_{j}T_{b_j}$
as correspondences in $CH^n(E^n\times E^n)$.
\end{enumerate}
\end{lemma}
Let $\rho$ be an action of a finite group $G$ on the variety
$E^n$, preserving origin, $\chi$ be a character of $G$ to
$\{\pm 1\}$.
Let  
$P(\chi)=\displaystyle
\frac{1}{\# G}\sum_{\sigma\in G}\chi(\sigma)\sigma^*
$
be the projector to $\chi$-part.
Let 
\begin{align*}
&i_\chi:P(\chi)CH^i(X\times E^n,j)\to CH^i(X\times E^n,j), \\
&\pi_\chi:CH^i(X\times E^n,j)\to P(\chi)CH^i(X\times E^n,j)
\end{align*}
be the natural inclusion and projection for the projector 
$P(\chi)$.
\begin{proposition}
\label{pararell transport}
Under the above notations, the composite
\begin{align*}
P(\chi)CH^i(X\times E^n,j)&\xrightarrow{i_{\chi}}
CH^i(X\times E^n,j)\xrightarrow{T_{s,N}}
CH^i(X\times E^n,j) \\
&\xrightarrow{\pi_{\chi}}
P(\chi)CH^i(X\times E^n,j)
\end{align*}
is equal to $T_{tr(s), M}$ 
where $M=N\cdot \#G$ and $\tr(s)=\sum_{\sigma\in G}\sigma(s)$.
\end{proposition}
\begin{proof}
We may assume that all $N$-torsion points of $s$ and $M$-torsion points
 of
$\tr(s)$ are defined over $K$.
Let $\displaystyle x=\frac{1}{\# G}\sum_\sigma\chi(\sigma)[\sigma^*Z]$
be an element in $P(\chi)CH^i(X\times E^n,j)$.
We choose $\tilde s$ and $\tilde t$ such that $N\tilde s=s$  and 
$\#G\cdot\tilde t=\tr(\tilde s)$. Thus, we have
\begin{align*}
P(\chi)T_{s,N}(x)
&=\frac{1}{(\# G)^2}\sum_{\sigma,\tau\in
 G}\chi(\tau\sigma)
[\tau^*\sigma^*Z+\tau^*\tilde s]
\\
&=\frac{1}{(\# G)^2}\sum_{\sigma,\tau\in
 G}\chi(\sigma)
[\sigma^*Z+\tau^*\tilde s]
\\
&\sim \frac{1}{(\# G)}\sum_{\sigma\in
 G}\chi(\sigma)
[\sigma^*Z+\tilde t]=T_{\tilde t}(x).
\end{align*}
\end{proof}

\subsubsection{Quasi-isomorphism
$\bold K_{S,n,sgn} \to \bold K_{CS,n,\chi}$ }
\label{subsubsection:two definitions coincide}
We define an isomorphism
$$
\widetilde{\alpha}:(E_B^{n+1})_s\to E_B^n:
(p_1, \dots, p_{n+1})\mapsto (x_1, \dots, x_n)
$$ 
by
\begin{align*}
x_1=p_1,\ 
x_2=p_1+p_2,\dots,  
x_n=p_1+\cdots + p_n.
\end{align*}
Let $I\in \Cal I_{S,p}$. If 
$\widetilde\alpha(D_{S,I})\nsubset \Cal D$,
then there exists a unique $J\in \Cal I_{C,p}$
such that $\widetilde\alpha(D_{S,I})=D_{C,J}$.
Therefore the map $\widetilde\alpha$
defines a homomorphism 
\begin{equation}
\label{S to B}
Z^i(E\times D_{S,I},j)\to Z^i(E\times D_{C,J}^0,j),
\end{equation}
and by taking summations, we have a map
of complexes 
$
\tilde\alpha:\bold K_{S,n} \to
\bold K_{CS,n}.
$
By composing inclusion $\bold K_{S,n,sgn}\to \bold K_{S,n}$
and projector $\bold K_{CS,n}\to \bold K_{CS,n,\chi}$,
%
we have the map of complex 
$\beta:\bold K_{S,n,sgn} \to
\bold K_{CS,n,\chi}$.
\begin{proposition}
\label{quasi-iso for SCB}
The map $\beta:\bold K_{S,n,sgn} \to \bold K_{CS,n,\chi}$
is a quasi-isomorphism.
As a corollary, we have
$$
CH^{n+1}(E\times (U_B^{n+1})_s,n+1-j)_{sgn}
\simeq H^j(\bold K_{C,n,\chi})
$$
\end{proposition}
\begin{proof}
Let $I\in \Cal I_{S,p}$ and $J\in \Cal I_{C,p}$,
$\frak S_I$ and $G_{C,J}$
be the stablizer of the component $D_{S,I}$
and $D_{C,J}$
in $\frak S_{n+1}$ and $G_{C}$, respectively. 
The restriction of the characters $sgn$ and $\chi$
to $\frak S_I$ and $G_{C,J}$ are denoted by
$(sgn,I)$ and $(\chi,J)$.
The proposition is reduced to the quasi-isomorphism of
$$
\widetilde\alpha:P(sgn,I)Z^i(E\times D_{S,I},j)_{sgn,I}\to 
P(\chi,J)Z^i(E\times D_{C,J},j)_{\chi,J}
$$
for $\widetilde\alpha(D_{S,I})=D_{C,J}$.
Therefore it is enough to show that the induced homomorphism
$$
\widetilde\alpha:P(sgn,I)CH^i(E\times D_{S,I},j)_{sgn,I}\to 
P(\chi,J)CH^i(E\times D_{C,J},j)_{\chi,J}
$$ 
is an isomorphism.
This is reduced to the case 
where $I=\emptyset$ and $J=\emptyset$.

We consider the algebraic correspondence
\begin{align*}
\Gamma_{\Sigma}=\{(x_1, \dots,x_n;p_1,\dots, p_n) 
&=(p_1,p_1+p_2,\dots,
p_1+\cdots+p_n;
p_1, \dots, p_n
) \\
&\in E^n_B\times_B E^n_B\}
\in CH^n(E^n_B\times_B E^n_B).
\end{align*}
The action of $\sigma\in \frak S_{n+1}$ and 
$\tau \in G_C$
is obtained by the graph
$\Gamma(\sigma)$ and $\Gamma(\tau)$ of
$\sigma$ and $\tau$.


Let $\tilde B$ be a finite flat Galois extension of $B$.
Let $\frak p_0$ and $\frak q_0$ be sections in $E_{\tilde B}$
such that $(n+1)\frak p_0=s$ and $2\frak q_0=s$.
We set 
$\frak p=(\frak p_0, \dots, \frak p_0)$ and
$\frak q=(\frak q_0, \dots, \frak q_0)$.
The translations $T_{\frak p}$ and
$T_{\frak q}$ by $\frak p$ and $\frak q$ is obtained by
the algebraic correspondences 
$\Gamma(T_{\frak p})$ and $\Gamma(T_{\frak q})$.
Then 
\begin{align*}
&\Gamma(\sigma_{lin})=\Gamma(T_{-\frak p})\circ \Gamma(\sigma) \circ
\Gamma(T_{\frak p}) \\
&\Gamma(\tau_{lin})=
\Gamma(T_{-\frak q})\circ \Gamma(\tau) \circ
\Gamma(T_{\frak q}) 
\end{align*}
are linear algebraic correspondences.
By considering cohomological classes, we have the following lemma.
\begin{lemma}
We have the following equality in $CH^n(E_B^n\times_B E_B^n)$
for projectors:
\begin{align*}
&\frac{1}{\#\frak S_{n+1}}\sum_{\sigma\in \frak S_{n+1}}
sgn(\sigma)\ \Gamma_{\Sigma}\circ
\Gamma(\sigma_{lin})
\circ\Gamma_{\Sigma^{-1}}
\\
=&
\frac{1}{\# G_C}
\sum_{\tau\in G_C}
\chi(\tau)\ \Gamma(\tau_{lin})
\end{align*}
\end{lemma}



We set $CH=CH^i(E\times E^n_B,j)$.
We have the following commutative diagram:
$$
\begin{matrix}
P(sgn_{lin})CH&\xrightarrow[\simeq]{T_{\frak p}} &
P(sgn)CH&\xrightarrow{\overset{(*)}{\Gamma_{\Sigma}}} &
P(\chi)CH&\xrightarrow[\simeq]{T_{-\frak q}} &
P(\chi_{lin})CH \\
i_{sgn_{lin}} \downarrow & & \downarrow & & \uparrow & & \uparrow \pi_{\chi_{lin}}\\
CH&\xrightarrow{T_{\frak p}} &
CH&\xrightarrow{\Gamma_{\Sigma}} &
CH&\xrightarrow{T_{-\frak q}} &
CH \\
\end{matrix}
$$
By this commutative diagram, $(*)$ is an isomorphism, since we have
\begin{align*}
\pi_{\chi_{lin}}\circ 
T_{-\frak q}\circ \Gamma_{\Sigma}\circ T_{\frak
 p}\circ i_{sgn_{lin}}  
&=
\pi_{\chi_{lin}}\circ 
T_{-\frak q+\Sigma \frak p}\circ \Gamma_{\Sigma}\circ 
i_{sgn_{lin}} 
\\
&=
\pi_{\chi_{lin}}\circ 
T_{-\frak q+\Sigma \frak p}\circ i_{\chi_{lin}}\circ 
\Gamma_{\Sigma} \\
&=
T_{tr(-\frak q+\Sigma \frak p),\#G}\circ 
\Gamma_{\Sigma},
\end{align*}
by Proposition \ref{pararell transport}.

\end{proof}
\begin{remark}
In this section, we treat an elliptic curve over a field $K$.
One can define the notion of mixed elliptic motives
for a family of elliptic curve $\Cal E \to B$ over a base 
scheme $B$.
\end{remark}

\subsection{Definition of $Log_{MEM}$}
First we construct a mixed elliptic motif $Log_{MEM}$ in the category 
$K\Cal C(A_{MEM}/\Cal O_S)^0$. There are two constructions of
$Log_{MEM}$ in [BL], one is in \S 1 and the other is in \S 6.
In \S \ref{subsubsection:two definitions coincide}
we show that two double complexes associated to
two constructions
are quasi-isomorphic.

Let $E$ be an elliptic curve over $K$ and 
$B$ be a copy of $E$.
Let $E_B$
be the trivial family $E\times E$ of $E$ over $B$
and the inversion of $E$ is written as $\iota$.
Then the diagonal subvariety $s$ defines a
section of $E_B$ over $B$.

Let $s$ be the section in $E_B$ defined by the diagonal
map. Let $\widetilde{B}$ be a flat finite
variety over $B$ such that $2$-torsion points
of $s$ are defined in $\widetilde{B}$.
We define a cycle $\tilde i_{s,n}$
in $Z^{n+1}(E^n_{\widetilde{B}}\times E^{n+1}_{\widetilde{B}},0)$ by
$$
\frac{1}{4}\sum_{2\tilde s=s}
[(x_1, \dots, x_n:\tilde s,x_1, \dots, x_n)].
$$
Then this correspondence is defined over $B$, which is also
denoted as $\tilde i_{s,n}$.
An element $(\sigma,e_1, \dots, e_n)$ in the groups 
$\bold Z/2\bold Z$ 
acts on $E^n_B$ by
$$
(x_1, \dots, x_n)\mapsto ([(-1)^{e_1}]x_1, \dots, 
[(-1)^{e_n}]x_n) 
$$
which is extended to the action of the group $G_{C,n}$ defined in 
(\ref{definition of group alt on cubic}).
Let $\chi_n$ be the character of $G_{C,n}$ defined in
(\ref{definition of cubic alt character}).
Then the group $G_{C,n}\times G_{C,n+1}$ acts on the group
$Z^{n+1}(E^n_{\widetilde{B}}\times E^{n+1}_{\widetilde{B}},0)$.
By applying the projector $P(\chi_n)\cdot P(\chi_{n+1})$
to $\tilde i_{s,n}$, we get an element $i_{s,n}$
in 
$$
P(\chi_n)
Z^{n+1}(E^n_{\widetilde{B}}\times 
E^{n+1}_{\widetilde{B}},0) P(\chi_{n+1})
$$
By restricting to the open set $B^0=B-\{0\}$, we have
$$
i_{s,n}\in
Hom^1_{MEM/{B^0}}(Sym^{n}(V)(n),Sym^{n+1}(V)(n+1))
$$ 
which
is also denoted by $i_{s,n}$.
We introduce an object $Log_{MEM,n}$ in 
$K\Cal C(A_{MEM}/\Cal O_S)$ over 
$B^0$ by setting
$$
Log_{MEM,n}^{p}=Sym^p(V)(p)\quad (0\leq p \leq n).
$$
\begin{lemma}
We have
$$
i_{s,p+1}\circ i_{s,p}=0.
$$
\end{lemma}
By the above lemma, 
by setting $d_{p+k,p}=0$ for $k\geq 2$,
we have a DG complex $Log_{MEM,n}$
$$
\begin{matrix}
0\to
&\bold Q
&\to &Sym^1(V)(1)
&\to  \cdots
\to &Sym^n(V)(n)& \to 0 \\
&  \Vert& & \Vert& &  \Vert \\
0\to
&Log_{MEM,n}^0
&\to &Log_{MEM,n}^1
&\to  \cdots
\to &Log_{MEM,n}^n& \to 0,
\end{matrix}
$$
using
$$
i_s\in \underline{Hom}_{EM}^1(Log_{MEM,n}^{i},Log_{MEM,n}^{i+1}).
$$
The morphism $Log_{MEM,n+1} \to Log_{MEM,n}$ 
defines a projective system in the
category of mixed
elliptic motives. 
Using the homomorphisms $i_{s,q}$, we get the following double complex
\begin{align*}
&\underline{\bold{Hom}}_{MEM}(V,Log_{MEM,n}):\\
&0\to Z^{1}_{-}(E\times B,\bullet)
\xrightarrow{i_{s,0}} 
Z^{2}_{-}(E\times E_B,\bullet) \xrightarrow{i_{s,1}}\cdots 
 \\
&
\cdots \xrightarrow{i_{s,n-2}}
Z^{n}_{-}(E\times E_B^{n-1},\bullet)alt^{n-1}
\xrightarrow{i_{s,n-1}} 
Z^{n+1}_{-}(E\times E_B^{n},\bullet)alt^{n}\to 0
\end{align*}
The associate simple complex of 
$\underline{\bold{Hom}}_{MEM}(V,
Log_{MEM,n})$
is denoted by
$\underline{{Hom}}_{MEM}(V,
Log_{MEM,n})$.
The relation between the differential $d_{n+1,n}$
and the complex $\bold K_{C,n,\chi}$
is given by the following proposition.

The $(-1)$-actions $\iota$ on the left most factor $E$ of 
$E\times E_B^{q}$ and $E\times (U_B^{n+2})_s$ are compatible,
where the subscript $-$ means the $-1$ part for the involution $\iota$.

\begin{proposition}
\label{comparizon pararell}
\begin{enumerate}
\item
The map
$$
\bigg[\underset{{J \in \Cal I_{C,i}}}
\oplus Z^{n+1-i}(E\times D_{C,J},\bullet)\bigg]_{-,\chi}
\to 
Z^{n+1-i}_{-}(E\times E_B^{n-i},\bullet)alt^{n-i}
$$
defined by
$\displaystyle\frac{1}{4}\sum_{2\tilde s=s}T_{-\tilde s}$
induces a homomorphism 
$$
\bold K_{C,n,\chi} \to \underline{\bold {Hom}}_{MEM}(V,Log_{MEM,n})
$$
of double complexes.
\item
The $(-1)$-actions $\iota$ on the left most factor $E$ of 
$E\times E_B^{q}$ and $E\times (U_B^{n+2})_s$ are compatible.
The above homomorphism is a quasi-isomorphism.
As a consequence, we have the following isomorphism:
\begin{equation}
\label{comparizon with simplicial and mixed elliptic}
CH^{n+1}(E\times (U_B^{n+1})_s,n+1-j)_{-,sgn}\simeq 
H^j\underline{Hom}_{MEM}(V,Log_{MEM,n}).
\end{equation}

\end{enumerate}
\end{proposition}

\subsection{Polylog object as a mixed elliptic motif}

We recall the definition of elliptic polylog class in
$CH^{n+1}(E\times (U_B^{n+1})_s,n)_{-,sgn}$ defined in \cite{BL}
using the eigen decomposition of the higher Chow group.

Let $[a]_B^{n+1}:(E_B^{n+1})_s\to (E_B^{n+1})_s$
be a map defined by
$$
(p_1, \dots, p_{n+1},s)\mapsto
([a]p_1, \dots, [a]p_{n+1},[a]s)
$$ 
We set $(U_{a,B}^{n+1})_s=([a]^{n+1}_B)^{-1}(U_B^{n+1})_s$
and define a homomorphism $Nm$ by the composite of 
the following homomorphisms:
\begin{align*}
CH^{n+1}(E\times (U_B^{n+1})_s,n)_{-,sgn}
&\xrightarrow{j^*}
CH^{n+1}(E\times 
(U_{a,B}^{n+1})_s,n)_{-,sgn} \\
&\xrightarrow{Nm} CH^{n+1}(E\times (U_B^{n+1})_s,n)_{-,sgn}.
\end{align*}
Here $j^*$ is induced by the open immersion
$(U_{a,B}^{n+1})_s\to (U_{B}^{n+1})_s$ and
$Nm$ is the norm map for the map $[a]^{n+1}_B$.
\begin{definition}
We define 
$CH^{n+1}(U\times (U_B^{n+1})_s,n)_{-,sgn}^{(1)}$
by the $a$-eigen space for the action of the map $Nm$.
\end{definition}
The following proposition is proved in \cite{BL}.
\begin{proposition}[\cite{BL}]
\label{key to define polylog class}
The residue map
$$
CH^{n+1}(U\times (U_B^{n+1})_s,n)_{-,sgn}^{(1)}\to
CH^{n}(U\times (U_B^{n})_s,n-1)_{-,sgn}^{(1)}
$$ 
is an isomorphism.
\end{proposition}
\begin{definition}[\cite{BL}]
\begin{enumerate}
\item
By applying Proposition \ref{key to define polylog class}
successively, we have an isomorphism
$$
CH^{n+1}(U\times (U_B^{n+1})_s,n)_{-,sgn}^{(1)}\to
CH^{1}(U\times (U_B^{1})_s,0)_{-,sgn}^{(1)} \simeq
\bold Q\Delta,
$$
where $\Delta$ is the diagonal in $U\times (U_B^1)_s$.
We define the polylog class $P_n$
by the element in
$CH^{n+1}(U\times (U_B^{n+1})_s,n)_{-,sgn}^{(1)}$
corresponding to $\Delta$.
\item
The element in 
$H^1\underline{Hom}_{MEM}(V,Log_{MEM,n})$ 
corresponding to
$P_n$ via the isomorphism
(\ref{comparizon with simplicial and mixed elliptic})
is also denoted as $P_n$ and called the elliptic polylog
class.
\end{enumerate}
\end{definition}

Let $\widetilde{P_n}$ be a representative of the cohomology
class $P_n$ in the following associate simple complex
$\underline{Hom}_{MEM}(V,Log_{MEM,n})$:
$$
{\tiny
\begin{matrix}
 & \partial\downarrow & & \partial\downarrow \\
%
\to &
Z^{n}(E\times E^{n-1}_B, n)_{sgn}e^{-n+1}
&\overset{i_s}\to &
Z^{n+1}_{-}(E\times E_B^{n}, n)_{sgn}e^{-n}
&\to 0 \\
%
& \partial\downarrow & & \partial\downarrow \\
%
\to &
Z^{n}(E\times E^{n-1}_B, n-1)_{sgn}e^{-n+1}
&\overset{i_s}\to &
Z^{n+1}_{-}(E\times E_B^{n}, n-1)_{sgn}e^{-n}
&\to 0 \\
%
& \partial\downarrow & 
& \partial\downarrow \\
%
\to &
Z^{n}(E\times E^{n-1}_B, n-2)_{sgn}e^{-n+1}
&\overset{i_s}\to &
Z^{n+1}_{-}(E\times E_B^{n}, n-2)_{sgn}e^{-n}
&\to 0 \\
%
& \partial\downarrow & & \partial\downarrow \\
\end{matrix}
}
$$
Then $\widetilde{P_n}$ is a direct sum of elements
$$
pl_{i}=pl_i^{\#}e^{-i} \in Z^{i+1}_{-}(E\times E_B^{i}, i)_{sgn}e^{-i}
=\underline{Hom}_{MEM}^1(V, Log_{MEM,n}^ie^{-i}).
$$ 
which is called an elliptic polylog extension data.
The closedness of $\widetilde{P_n}$ is equivalent to the relation:
$
\partial(pl_{i})+(-1)^ii_s(pl_{i-1})=0.
$
Then we have 
$$
\partial pl_{i}+(-1)^{i-1}pl_{i}\cdot i_s=0,
$$
where $\cdot$ is the multiplication.
\begin{definition}
We define an object $Pl_{MEM,n}=\{Pl_{MEM,n}^i,d_{ji}\}$
in the quasi-DG category $(MEM)=K(EM)$
of DG complex in $\Cal C_{EM}$
as follows.
$$
Pl_{MEM,n}^i=\begin{cases}
V & (i=-1) \\
Log_{MEM,n}^i & (0\leq i\leq n) \\
0 & (\text{ otherwise })
\end{cases}
$$
The maps $d_{ji}\in \underline{Hom}^1(Pl_{MEM,n}^i,Pl_{MEM,n}^j)$
are defined by
$$
d_{ji}=\begin{cases}
d_{j,-1}=pl_j & (0\leq j \leq n) \\
d_{j+1,j}=(-1)^{j+1}i_s & (0\leq j\leq n-1) \\
0 & (\text{ otherwise })
\end{cases}
$$
\end{definition}

\subsection{Description of the comodule associated to 
polylog motif}
\label{explicit comodule structure of polylog}
The polylog motif is concentrated at degree zero.
By the functor $\varphi$, the object $Pl_{MEM,n}$ corresponds to an 
$H^0(Bar(A_{EM}/\Cal O),\epsilon)$-comodule.
In this subsection, we write down the $H^0(Bar)=H^0(Bar(A_{EM}/\Cal
O),\epsilon)$-comodule structure of $\varphi(Pl_{MEM,n})$.
As a vector space, we have
\begin{align*}
Pl_n&=\varphi(Pl_{MEM,n})=\oplus_{i=-1}^n Pl_n^i,\\
Pl_n^i&=\begin{cases}
V &(i=-1) \\
Sym^i(V)(i) & (0\leq i\leq n).
\end{cases}
\end{align*}
Let $M(Pl_n, H^0(Bar))$ be the endomorphism of $Pl_n$ 
obtained by the scalar extension to the commutative ring $H^0(Bar)$.
By the isomorphism
$$
Hom_{\bold k}(Pl_n, H^0(Bar)\otimes Pl_n)\simeq
\oplus_{i,j=-1}^n(Pl_n^j)^*\otimes H^0(Bar)\otimes Pl_n^i,
$$
the coproduct $\Delta_{Pl_{MEM,n}}$ is given by the sum of 
$a_{ij}\in (Pl_n^j)^*\otimes H^0(Bar)\otimes Pl_n^i
$, where
$$
a_{ij}=
\begin{cases}
\delta_V\otimes \sum_{u=0}^i(-1)^{u(i-u+1)}pl_u \otimes 
{i_s}^{\otimes (i-u)}\otimes \delta_{S^iV(i)}
 &j=-1,\,\,i\geq 0\\
(-1)^{(i-j)(j-1)}\delta_{S^jV(j)}\otimes {i_s}^{\otimes (i-j)}\otimes \delta_{S^iV(i)}
& j\geq 0,\,\,i>j\\
\Delta_V & i=j=-1\\
\Delta_{S^iV(i)}& i=j\geq 0\\
0& \text{ otherwise},
\end{cases}
$$
and $\delta_V$ is an element in $V^*\otimes V$ corresponding to
the identity element in $V^*\otimes V=Hom_{\bold k}(V,V)$.

\end{document}